\documentclass[11pt,usenames,dvipsnames]{article}

\usepackage[utf8]{inputenc} 

\usepackage{xr} 
\usepackage{authblk}

\usepackage{comment}
\usepackage{lipsum}
\usepackage{bm} 
\usepackage{eufrak}
\usepackage{upgreek}
\usepackage[dvipsnames]{xcolor}
\usepackage{enumitem, calligra}
\usepackage{lipsum,tikz}
\usepackage{rsfso}
\usepackage{abraces} 
\usepackage[letterpaper,
			left = 1.1truein,  
			right = 1.1truein, 
			top = 1.1truein, 
			bottom = 1.1truein]{geometry} 
\usepackage{amsthm, amsfonts, amssymb, amsmath,enumitem, bbm, mathabx, mathrsfs}
\allowdisplaybreaks
\usepackage{mathtools}
\mathtoolsset{showonlyrefs,showmanualtags}
\numberwithin{equation}{section}
\newcommand{\revsag}{\textcolor{black}}
\newcommand{\revsagc}{\color{black}}
\newcommand{\revsagr}[1]{\textcolor{black}{#1}}

\newtheorem{defn}{Definition}[section]
\newtheorem{thm}{Theorem}[section]
\newtheorem{lem}{Lemma}[section]

\newtheorem{prop}{Proposition}[section]
\theoremstyle{remark}
\newtheorem{rem}{Remark}[section] 
\newcommand{\E}{\mathbb{E}}

\newcommand{\smmse}{\mathrm{mmse}}

\newcommand{\sbrac}[1]{[#1]}

\newcommand{\R}{\mathbb{R}}
\newcommand{\Sigs}{\Psi}
\newcommand{\wt}{\widetilde}
\newcommand{\rs}{\mathrm{RS}}
\newcommand{\sr}{{\revsag{\mathcal H_{k,t;\bm \eta}}}}
\newcommand{\rt}{\right}
\newcommand{\lt}{\left}
\newcommand{\revmark}[1]{\textcolor{black}{#1}}
\usepackage{subcaption}
\usepackage{hyperref}
\usepackage{array}

\DeclareMathOperator*{\argmin}{arg\,min}

\usepackage{graphicx} 
\usepackage{color, float}
\usetikzlibrary{shapes}

\usepackage{booktabs} 
\usepackage{array} 
\usepackage{verbatim} 
\usepackage[numbers,sort]{natbib} 

\newcommand{\graphreg}{{\sf{Reg-Graph}}\,}

\title{Bayes optimal learning in high-dimensional linear regression with network side information}
\author[1]{Sagnik Nandy}
\author[2]{Subhabrata Sen}
\affil[1]{\emph{University of Pennsylvania}}
\affil[2]{\emph{Harvard University}}
\affil[1]{\emph{Email id: sagnik@wharton.upenn.edu.}}
\affil[2]{\emph{Email id: subhabratasen@fas.harvard.edu}}
\begin{document}
	\maketitle
%
%
%

\begin{abstract}
Supervised learning problems with side information in the form of a network arise frequently in applications in genomics, proteomics and neuroscience. For example, in genetic applications, the network side information can accurately capture background biological information on the intricate relations among the relevant genes. In this paper, we initiate a study of Bayes optimal learning in high-dimensional linear regression with network side information. To this end, we first introduce a simple generative model (called the \graphreg model) which posits a joint distribution for the supervised data and the observed network through a common set of latent parameters. Next, we introduce an iterative algorithm based on Approximate Message Passing (AMP) which is provably Bayes optimal under very general conditions. In addition, we characterize the limiting mutual information between the latent signal and the data observed, and thus precisely quantify the statistical impact of the network side information. Finally, supporting numerical experiments suggest that the introduced algorithm has excellent performance in finite samples.

~
\\
\end{abstract}

\section{Introduction}
Given data $\{(y_i, \revsag{\bm \phi_i}): 1\leq i \leq n\}$, $y_i \in \mathbb{R}$, $\bm \phi_i \in \mathbb{R}^p$, the classical linear model 
\begin{align}
    \bm{y} = \bm{\Phi} \bm{\beta}_0 + \bm \varepsilon, \label{eq:linear_model} 
\end{align}
where $\bm y = (y_1,\ldots,y_n)^\top$, $\bm \varepsilon \in \R^n$ and 
$
\bm \Phi = \sbrac{\bm \phi_1\;\cdots\;\bm \phi_n}^\top \in \R^{n \times p},
$
furnishes an ideal test bed to study the performance of diverse supervised learning algorithms. 
%
%
In the modern age of big data, the number of observations $n$ and the feature dimension $p$ are often both large and comparable. 

In these challenging high-dimensional scenarios, scientists have recognized the importance of incorporating domain knowledge into the relevant statistical inference methodology. Success in this direction can substantially boost the performance of statistical procedures, and facilitate novel discoveries in critical applications. Arguably, the most well-known instance of this philosophy is the incorporation of sparsity into high-dimensional statistical methods (see e.g. \cite{tibshirani1996regression,chen2001atomic,park2008bayesian,ishwaran2005spike}). 

In this paper, we consider a setting where in addition to the supervised data, one observes pairwise relations among the features in the dataset. This pairwise relation can be conveniently captured using a graph $\bm{G}=(V,E)$. The vertices of the graph represent the features. The edges represent pairwise relations among the features, e.g., an edge might indicate that the two endpoints are likely to be both included in the support of the linear model. 

This setup is motivated by datasets arising in diverse application areas, e.g., genomics, proteomics, and neuroscience. 
In the genomic context, the response $y$ represents phenotypic measurements on an individual, while the features $\bm \phi$ represent genetic expression. In addition, scientists often have background knowledge about genetic co-expressions---this information can be efficiently captured in terms of the graph described above. We refer the interested reader to \cite{li_and_li,leiserson2015pan} for detailed discussions of settings where such data sets arise naturally. To study this problem in depth, we first introduce a simple generative model, which we refer to as the \graphreg model. 
\begin{itemize}
    \item[(i)] Generate $\sigma_{01}, \cdots, \sigma_{0p} \stackrel{i.i.d.}{\sim} \mbox{Bernoulli}(\rho)$ for some $\rho \in (0,1)$.
    \item[(ii)] Given $\bm{\sigma}_0 = (\sigma_{01}, \cdots, \sigma_{0p})^\top \in \R^p$, generate regression coefficients $\beta_{0i} | \sigma_{0i} \sim P(\cdot|\sigma_{0i})$, where $P(\,\cdot
\,| \,\cdot\,)$ is a Markov kernel. Throughout, we assume that $P(\cdot\,|\,0)$ and $P(\cdot\,|\,1)$ either have bounded supports or  log-concave densities. 
    \item[(iii)] Given the regression coefficients $\bm{\beta}_0 = (\beta_{01}, \cdots, \beta_{0p})^\top \in \R^p$ and features $\{\bm \phi_{\mu} \in \mathbb{R}^p$, $1\leq \mu \leq n\}$,  the response $\bm y \in \R^n$ is sampled using a linear model
    \begin{align}
    \label{eq:lin_model_1}
        y_{\mu} = \sum_{i=1}^{p} \beta_{0i} \phi_{\mu i} + \varepsilon_\mu, 
    \end{align}
    where $\varepsilon_\mu \sim N(0, \Delta)$ are i.i.d for $1 \le \mu \le n$. 
    \item[(iv)] Finally, one observes a graph $\bm{G}=(V,E)$ on $p$ vertices. As mentioned above, the vertices represent the observed features. We assume that given $\sigma_{01}, \cdots, \sigma_{0p}$, the edges are added independently with probability
    \begin{align}
    \label{eq:graph}
        \mathbb{P}[\{i,j\} \in E| \sigma_{01}, \cdots, \sigma_{0p}] = \begin{cases}
        \frac{a_p}{p} & \textrm{if}\,\,\, \sigma_{0i} \sigma_{0j} = 1 \\
        \frac{b_p}{p} & \textrm{o.w.} 
        \end{cases} 
    \end{align}
    In this formulation, $\{a_p: p \geq 1\}$ and $\{b_p: p \geq 1\}$ represent general sequences dependent on $p$. 
\end{itemize}

We note that our \graphreg model ties the generation of the regression data and the graph $\bm{G}$ via the \emph{same} underlying variables $\sigma_{01}, \cdots, \sigma_{0p}$. In this context, one naturally wishes to combine the two data sources to carry out inference on these common latent parameters. Throughout, we assume that the model parameters $\rho$, $P(\,\cdot
\,| \,\cdot\,)$, $\Delta$, $a_p, b_p$ are known to the statistician. 

The \graphreg model naturally ties together some popular ideas in statistics and machine learning: 
\begin{itemize}
    \item[(i)] Note that the marginal distribution of the supervised data  $\{(y_\mu, \bm \phi_{\mu}): 1\leq \mu \leq n\}$ includes the celebrated spike and slab model from Bayesian statistics \cite{ishwaran2005spike}. For the spike and slab model, one also assumes $P(\,\cdot\,|\,0\,) = \delta_0$.   Informally, the sigma variables encode the support of the signal vector $\bm \beta_0$ in this case, and one wishes to recover the latent indicators $\sigma_i$ from the data. The spike and slab model and its relatives have emerged as the canonical choice for sparse regression models in high-dimensional Bayesian statistics in the past two decades (we refer the interested reader to \cite{tadesse2021handbook} and the references therein for a detailed survey of the progress in this area). While this is not necessary for our model, we will explore this case in depth in our subsequent discussion. 
    
    \item[(ii)] On the other hand, if we focus on the marginal distribution of the graph $\bm{G}$, it corresponds to a graph with a hidden community \cite{hajek2018recovering,arias2014community,hajek2015computational}. In this case, one typically assumes that $a_p \geq b_p$, so that the vertices with $\sigma_{0i}=1$ have a higher density of connecting edges. The recovery of the hidden community from the graph data has been studied in depth in the recent past \cite{montanari2015finding}. 
\end{itemize}

Thus, the \graphreg model ties together two distinct threads of inquiry in statistics and machine learning using a natural generative model. We note that the one hidden community assumption is a convenient simplification ---- the model and our subsequent results can be naturally extended to a setting with multiple latent communities.

In this paper, we study the \graphreg model and make the following contributions: 
\begin{itemize}
    \item[(i)] We study the \graphreg model under an additional i.i.d. gaussian assumption on the features \revsag{$\bm \phi_{\mu}$ for $1 \le \mu \le n$}. In addition, we assume a proportional asymptotic setting, where the number of observations $n$ and the feature dimension $p$ are both large and comparable. Formally, we assume that $n/p \to \kappa \in (0,\infty)$. Note that we allow $\kappa \in (0,1)$, and thus can cover settings where the feature dimension $p$ is larger than the sample size $n$. Under these assumptions, we introduce an algorithm based on Approximate Message Passing (AMP) \cite{BM11journal,feng2022unifying} for estimation (of \revsag{$\bm \beta_0$}) and \revsag{recovery of $\bm \sigma_0$ which corresponds to support recovery if $P(\,\cdot\,|\,0\,) = \delta_0$}. We characterize the precise $L^2$-estimation error and the limiting False Discovery Proportion (FDP) \revsag{in the context of support recovery} under this algorithm. 
    
    \item[(ii)]  We characterize the mutual information between the data and the latent parameters under the \graphreg model. In particular, this allows us to derive the Bayes optimal estimation error in this setting. We establish that under a wide class of priors, the AMP algorithm introduced in this paper is Bayes optimal.  

To derive the limiting mutual information in this model, we use the \emph{adaptive interpolation} method developed by  \citet{Adaptive_Interpolation}. This approach has been used in several past works to characterize the limiting mutual information in planted models (see e.g. \cite{ barbier2019optimal} and references therein), and we build directly on these seminal works. We note that the mutual information in high-dimensional models can also be derived using other techniques (see e.g. \cite{reeves2019replica,barbier2020mutual})---the mutual information in our setting can also be potentially characterized by adopting these alternative approaches.  
    
    \item[(iii)] Finally, using numerical simulations, we compare the statistical performance of the proposed AMP algorithm with existing penalization based approaches for estimation and support recovery. In our numerical experiments, the AMP algorithm significantly outperforms the benchmark algorithm. 
\end{itemize}

\subsection{Main Results} 
We highlight our main results in this section.

\subsubsection{Algorithm based on Approximate Message Passing} 
As a first step, we introduce a class of iterative algorithms based on Approximate Message Passing (AMP) for parameter estimation and \revsag{recovery of $\bm \sigma_0$} in the \graphreg model. This algorithm naturally incorporates the supervised data with the auxiliary graph information for statistical inference. To this end, we set $\bm S=\bm \Phi/\sqrt{\kappa}$, \revsag{where $\kappa$ is the limit of the undersampling ratio $n/p$}, and 
\revmark{\[\widebar{\bm A}=\frac{\bm A-(b_p/p)}{\sqrt{(b_p/p)(1-(b_p/p))}},\]}
where $\bm{A}$ denotes the adjacency matrix of the graph $\bm{G}$. \revsag{We also set} $\bm y^\circ=\bm y/\sqrt{\kappa}$. 

Formally, Approximate Message Passing is not a single algorithm, but rather a class of iterative algorithms, which are specified in terms of a sequence of non-linearities used in each step. We refer the interested reader to \cite{BM11journal,feng2022unifying} for a discussion on the origins of AMP and its applications to high-dimensional statistics and signal processing. To describe the specific instance of AMP we use in our setting, consider two sequences of Lipschitz functions $\zeta_t:\mathbb{R}^2\rightarrow\mathbb{R}$, $f_t:\mathbb{R}^2\rightarrow\mathbb{R}$, with Lipschitz derivatives and the synchronized \emph{approximate message passing} orbits $\{\bm \sigma^{t+1}\}_{t \ge 0}$ and $\{\bm z^t, \bm \beta^{t+1}\}_{t \ge 0}$ defined as follows.
\begin{align}
\label{eq:amp_iterates_graph}
\bm \sigma^{t+1} &= \frac{\widebar{\bm A}}{\sqrt{p}}\bm f_t(\bm \sigma^t,\bm S^\top\bm z^{t-1}+\bm \beta^{t-1})\\
&\quad\quad-(\mathcal{A}\bm f_t)(\bm \sigma^t,\bm S^\top\bm z^{t-1}+\bm \beta^{t-1})\bm f_{t-1}(\bm \sigma^{t-1},\bm S^\top\bm z^{t-2}+\bm \beta^{t-2}), \label{eq:amp_iterates}\\
\bm z^t & = \bm y^\circ - \bm S\bm \beta^t+\frac{1}{\kappa}\bm z^{t-1}(\mathcal{A}\bm \zeta_{t-1})(\bm S^\top\bm z^{t-1}+\bm \beta^{t-1},\bm \sigma^{t})\\
\bm \beta^{t+1} &= \bm\zeta_t(\bm S^\top\bm z^t+\bm \beta^t,\bm \sigma^{t+1}) \label{eq:amp_iterates_1}
\end{align}
where the functions $\bm \zeta_t:(\mathbb{R}^{2})^p \rightarrow \mathbb{R}^{p}$, $\bm f_t:(\mathbb{R}^{2})^p \rightarrow \mathbb{R}^{p}$  are defined as \revsag{$\bm \zeta_t(\bm x):=(\zeta_t(x_1),\cdots,\zeta_t(x_p))$ and $\bm f_t(\bm x):=(f_t(x_1),\cdots,f_t(x_p))$. 
Further, for a function \revsag{$g: \mathbb{R}^{2} \rightarrow \mathbb{R}$, and $\bm t_1, \bm t_2 \in \R^p$}, the correction term $(\mathcal A\bm g)(\bm t_1, \bm t_2)$ is defined as
$
 (\mathcal{A}\bm g)(\bm t_1, \bm t_2)=\frac{1}{p}\sum\limits_{i=1}^{p}\frac{\partial}{\partial t_1}g(t_{1i}, t_{2i}).
$}
\revsag{Finally, we choose $\bm \beta^0$ and $\bm \sigma^0$ in such a way so that they are positively correlated with the true signals $\bm \beta_0$ and $\bm \sigma_0$, and we set $\bm \beta^{-1}=\bm z^{-1}=\bm 0$.}

The performance of the AMP algorithms described above will be characterized in terms of some low-dimensional scalar parameters \cite{BM11journal}. In turn, these scalar parameters are defined using an iteration referred to as \emph{state evolution}. Formally, define the parameters $\tau^2_t$, $\nu^2_t$ and $\eta_t$ by the following iteration.
\revmark{\begin{align}
\label{eq:state_evol}
\nu^2_{t+1} &= \mathbb{E}[f^2_{t}(\eta_{t}\Sigma+\nu_{t}Z_3, B+\tau_{t-1}Z_4)],\\
\eta_{t+1} & = \sqrt{\lambda}\nu^2_{t+1},\\
\tau^2_{t+1} &=\frac{1}{\kappa}\left(\Delta+\,\mathbb{E}[(\zeta_{t-1}(B+\tau_{t}Z_1,\eta_{t+1}\Sigma+\nu_{t+1}Z_2)-B)^2]\right),
\end{align}}
where $Z_1,Z_2,Z_3,Z_4$ are i.i.d $\mathsf{N}(0,1)$, $\Sigma \sim \mathrm{Bernoulli}(\rho)$ and $B|\Sigma \sim P(\cdot|\Sigma)$.
The initial conditions are given by
$\tau^2_0=\frac{1}{\kappa}\left(\Delta+\,\mathbb{E}[B^2]\right)$ and $\tau_{-1}=\eta_0=\nu_0=0$. \revsag{Further, $\lambda=O(1)$ is implicitly defined through the following equation,
\begin{equation}
\label{eq:def_snr}
    \frac{a_p-b_p}{p}=\sqrt{\frac{\lambda\widebar{d}_p(1-\widebar{d}_p)}{p}},
\end{equation}
where $\widebar{d}_p=b_p/p$.}
We will consider a specific sequence of update functions, given as $f_{-1}=\zeta_{-1}=0$ and for $t \geq 1$, 

\begin{align}
\label{eq:non_linearities}
f_{t-1}(x,y) &:=\mathbb{E}[\Sigma|\eta_{t-1}\Sigma+\nu_{t-1}Z_2=x, B+\tau_{t-2}Z_1=y], \\
\zeta_{t-1}(x,y) & :=\mathbb{E}[B|B+\tau_{t}Z_1=x,\eta_{t+1}\Sigma+\nu_{t+1}Z_2=y].
\end{align}
If we denote the collection of state evolution parameters at time point $t$ for $t \ge 0$ by $\mathcal T_t=\{(\eta_s,\nu_s,\tau_s): 0 \le s \le t\}$, then the set $\mathcal T_t$ evolves according to the following mechanism:

\vskip 2em
\begin{tikzpicture}[node distance=2cm, every node/.style={draw, text centered}]
    \node (start)[diamond] {$\mathcal T_t$};
    \node (process1) [right of=start, xshift=1.2cm, circle] {$f_t$};
    \node (decision) [right of=process1, xshift=0.5cm, rectangle] {$(\nu_{t+1},\eta_{t+1})$};
    \node (process2a) [right of=decision, xshift=0.3cm, circle] {$\zeta_{t-1}$};
    \node (process2b) [right of=process2a, xshift=0.1cm, rectangle] {$\tau_{t+1}$};
    \node (end) [right of=process2b, xshift=0.3cm, diamond, aspect=1] {$\mathcal T_{t+1}$};

    \draw[->] (start) -- (process1) node[midway, above, draw=none, fill=none] {$(\eta_t,\nu_t,\tau_{t-1})$};
    \draw[->] (process1) -- (decision);
    \draw[->] (decision) -- (process2a);
    \draw[->] (start.south) to[out=-20, in=-160] node[midway, above, draw=none, fill=none] {$\tau_t$} (process2a.south);
    \draw[->] (decision.north) to[out=50, in=160] node[midway, above, draw=none, fill=none] {$(\nu_{t+1},\eta_{t+1})$} (end.north);
    \draw[->] (process2a) -- (process2b);
    \draw[->] (process2b) -- (end);

    \draw[->] ([xshift=-0.5cm]start.west) -- (start); 
    \draw[->]  (end) -- ([xshift=0.5cm]end.east); 
\end{tikzpicture}

\begin{rem}
    One practical way to choose $\bm \beta^0$ and $\bm \sigma^0$ so that they are positively correlated with $\bm \beta_0$ and $\bm \sigma_0$ is as follows: We can fit a LASSO to the pair of observations given by $\{(y_\mu,\bm \phi_\mu): 1 \le \mu \le n\}$ and take the estimate of the regression coefficient $\widehat{\bm \beta}$ as $\bm \beta^0$. Similarly, we can compute the leading eigenvector of the matrix $\widebar{\bm A}$ and take it as $\bm \sigma^0$. However, these initializers are dependent on the data, and consequently, analyzing the behavior of the iterates $(\bm \sigma^t,\bm z^t, \bm \beta^t)$ will be significantly difficult. But we anticipate that following a technique similar to \cite{montanari2021} this can be done. However, we do not pursue it in the current manuscript.
\end{rem}

\begin{lem}
    \label{lem:continuity_diff_denoise}
    \revsagr{The denoisers $f_t,\zeta_t$ defined in \eqref{eq:non_linearities} are Lipschitz continuous in both arguments with Lipschitz derivatives.}
\end{lem}

\begin{proof}
Let $M_1 = B + \tau_t Z_1$ and $M_2 = \eta_{t+1}\Sigma+\nu_{t+1}Z_2$, where $(Z_1,Z_2)$ are i.i.d. standard gaussian random variables. The posterior distribution of $(\Sigma,B)$ given $(M_1,M_2)$ is given as 
    \begin{align}
        &dP_t(\Sigma,B | M_1 = x, M_2 = y) = \frac{1}{Z_t} \exp\Big( \frac{xB}{\tau_t^2} - \frac{B^2}{2 \tau_t^2} + \frac{y\eta_{t+1} \Sigma}{ \nu_{t+1}^2} - \frac{\eta^2_{t+1}\Sigma^2}{2 \nu_{t+1}^2}  \Big) dp(\Sigma, B), \nonumber
         \end{align}
where $p$ is the joint density of $(\Sigma, B)$ and
        \[Z_t = \int \exp\Big( \frac{xB}{\tau_t^2} - \frac{B^2}{2 \tau_t^2} + \frac{y\eta_{t+1} \Sigma}{ \nu_{t+1}^2} - \frac{\eta^2_{t+1}\Sigma^2}{2 \nu_{t+1}^2}  \Big) dp(\Sigma, B). \]
    This implies 
    \begin{align}
        f_t(x,y) &= \int \Sigma\, dP_t(\Sigma,B|M_1=x, M_2=y) \nonumber \\
        &= \frac{1}{Z_t} \int \Sigma \,\exp\Big( \frac{xB}{\tau_t^2} - \frac{B^2}{2 \tau_t^2} + \frac{y\eta_{t+1} \Sigma}{ \nu_{t+1}^2} - \frac{\eta^2_{t+1}\Sigma^2}{2 \nu_{t+1}^2}  \Big) dp(\Sigma, B). \nonumber 
    \end{align}
    In turn, we have, 
    \begin{align}
        \frac{\partial}{\partial x} f_t(x,y) = \frac{1}{\tau_t^2} \mathrm{Cov}_{P_t}(\Sigma,B), \,\,\,\,\, \frac{\partial}{\partial y} f_t(x,y) = \frac{\eta_{t+1}}{\nu_{t+1}^2} \mathrm{Var}_{P_t} (\Sigma). \nonumber 
    \end{align}
    In our setting, $\Sigma \in \{0,1\}$ and thus $|\frac{\partial f_t(x,y)}{\partial y}| \leq C_t$ for some constant $C_t>0$. On the other hand,  
    \begin{align}
        \Big| \frac{\partial}{\partial x} f_t(x,y)  \Big| = \frac{1}{\tau_t^2} \Big| \mathrm{Cov}_{P_t}(\Sigma, B) \Big| = \frac{1}{\tau_t^2} \sqrt{\mathrm{Var}_{P_t}(\Sigma) \mathrm{Var}_{P_t}(B)}. \nonumber    
    \end{align}
    As above, we note that $\Sigma \in \{0,1\}$ and thus $\mathrm{Var}_{\mu_t}(\Sigma) \leq C$ for some universal constant $C>0$. Finally, $\mathrm{Var}_{P_t}(B) \leq C'$
 for some universal constant $C'>0$ if $B$ is universally bounded. On the other hand, if $P(\cdot|0), P(\cdot|1)$ have log-concave densities, there exists $\varepsilon>0$ such  that $B|\Sigma$ is $\varepsilon$-strongly log-concave. In this case, $\mathrm{Var}_{P_t}(B)\leq C_t/\varepsilon$ for some universal constant $C_t>0$. \revsagr{This shows that $f_t(\cdot,\cdot)$ is Lipschitz continuous}. The proof for $\zeta_t$ is similar, and thus omitted. Using similar techniques one can also show that the denoisers $f_t$ and $\zeta_t$ have Lipschitz derivatives. 
\end{proof}

Using this specific sequence of non-linearities in \eqref{eq:amp_iterates},\eqref{eq:amp_iterates_1}  we obtain the sequence of estimates given by \revsag{$\widehat{\bm \sigma}^t=\bm f_{t}(\bm \sigma^t,\bm S^\top\bm z^{t-1}+\bm \beta^{t-1})$} and $\widehat{\bm \beta}^t=\bm \beta^{t}$. Typically, the statistical performance of $(\widehat{\bm \sigma}^t,\widehat{\bm \beta}^t)$ in estimating $(\bm \sigma_0, \bm \beta_0)$ improves with increasing number of iterations. To quantify the limiting statistical performance of the estimators obtained (in the limit of a  large number of iterations), we introduce 
\begin{align*}
\label{eq:mse_amp_1}
\mathsf{MSE^{AMP}_\sigma}:=\lim\limits_{t \rightarrow \infty}\lim\limits_{p \rightarrow \infty}\frac{1}{p^2}\;\mathbb E[\|\widehat{\bm \sigma}^t(\widehat{\bm \sigma}^t)^\top-\bm \sigma_0\bm \sigma^\top_0\|^2_F], \,\,
\mathsf{MSE^{AMP}_\beta}:=\lim\limits_{t \rightarrow \infty}\lim\limits_{n \rightarrow \infty}\frac{1}{n}\;\mathbb E[\|\bm\Phi(\widehat{\bm \beta}^t-\bm \beta_0)\|^2_2].
\end{align*}
Note the specific order of the iterated limits---we let the dimension $p\to \infty$ before we let the number of iterations diverge. This order of iterated limits is typical in the analysis of AMP algorithms. Using recent progress in the analysis of AMP algorithms, it might be possible to analyze the AMP algorithms after a growing number of steps \cite{li2022non,rush2018finite,cademartori2023non}, but we refrain from examining this direction in this paper. 

To characterize the limiting behavior of $\mathsf{MSE^{AMP}_\sigma}$ and $\mathsf{MSE^{AMP}_\beta}$, we need to introduce one final set of scalar functionals. Define
\revsag{
\begin{align}
\mathsf{mmse}_1(\mu,\xi)=\mathbb E\left[\left\{\Sigma-\mathbb E\Bigg[\Sigma\,\bigg|\,\sqrt{\mu}\Sigma+Z_2,B+\sqrt{(\Delta(1+\xi)/\kappa)}Z_1\right]\Bigg\}^2\right], \\
\mathsf{mmse}_2(\mu,\xi):=\mathbb E\left[\left\{B-\mathbb E\Bigg[B\,\bigg|\,B+\sqrt{(\Delta(1+\xi)/\kappa)}Z_3,\sqrt{\mu}\Sigma+Z_4\right]\Bigg\}^2\right],
\end{align}
where $Z_1,Z_2,Z_3,Z_4$ are i.i.d $\mathsf{N}(0,1)$.}
Armed with this definition, we can write
\revsag{
\[
\mathbb E\left[\mathbb E\left[\Sigma\,\bigg|\,\sqrt{\mu}\Sigma+Z_2,B+\sqrt{(\Delta(1+\xi)/\kappa)}Z_1\right]^2\right]=\rho-\mathsf{mmse}_1(\mu,\xi).
\]
Now if we set 
\[
\xi_t = \frac{\kappa\tau^2_t-\Delta}{\Delta} \quad \mbox{and} \quad \mu_t = \lambda\nu^2_t,
\]
we can rewrite the state evolution equations \eqref{eq:state_evol} as follows:
\begin{align}
\label{eq:fixed_point_iteration}
    \mu_{t+1} = \lambda(\rho-\mathsf{mmse}_1(\mu_t,\xi_{t-1})),\quad \mbox{and} \quad
\xi_{t+1} = \frac{1}{\Delta}\mathsf{mmse}_2(\mu_{t+1},\xi_t).
\end{align}
The following lemma shows that the recursion is contractive in nature:
\begin{lem}
    \label{lem:contractive}
 The sequence of parameters $(\mu_t,\xi_t)$ converges to a point $(\mu_*,\xi_*)$ satisfying:
 \begin{align}
 \label{eq:def_mu_fix_point}
\mu^*=\lambda(\rho-\mathsf{mmse}_1(\mu^*,\xi^*)),\,\,\,\,
\xi^* = \frac{1}{\Delta}\mathsf{mmse}_2(\mu^*,\xi^*),
\end{align}
as $t \rightarrow \infty$.
\end{lem}
}
The next result characterizes the limiting statistical behavior of the AMP algorithm \revsag{in terms of $(\mu^*,\xi^*)$}. 
\revsagr{
\begin{thm}
\label{thm:mse_amp}
Assume that $\frac{b_p}{p}(1-\frac{b_p}{p})\ge C\,\frac{\log p}{p}$ for any constant $C>0$. Then the reconstruction errors of the estimators $\{\widehat{\bm \sigma}^t, \widehat{\bm \beta}^t: t \ge 0\}$ satisfy
    \begin{align}
    \lim\limits_{t \rightarrow \infty}\lim\limits_{p \rightarrow \infty}\frac{1}{p^2}\;\|\widehat{\bm \sigma}^t(\widehat{\bm \sigma}^t)^\top-\bm \sigma_0\bm \sigma^\top_0\|^2_F= \rho^2-\frac{(\mu^*)^2}{\lambda^2}, \,\,\,\,\, \lim\limits_{t \rightarrow \infty}\lim\limits_{n \rightarrow \infty}\frac{1}{n}\;\|\bm\Phi(\widehat{\bm \beta}^t-\bm \beta_0)\|^2_2= \frac{\Delta\;\xi^*}{(1+\xi^*)},
\end{align}
almost surely. Furthermore, if $P(\,\cdot\,|\,0\,)$ and $P(\,\cdot\,|\,1\,)$  are compactly supported, we have, 
\begin{align}
   \mathsf{MSE^{AMP}_\sigma} = \rho^2-\frac{(\mu^*)^2}{\lambda^2}, \,\,\,\,\, \mathsf{MSE^{AMP}_\beta} = \frac{\Delta\;\xi^*}{(1+\xi^*)},
\end{align}
where $(\mu^*,\xi^*)$ are defined by \eqref{eq:def_mu_fix_point}. 
\end{thm}}
\begin{rem}
Let us observe that using these specific sets of non-linearities given by \eqref{eq:non_linearities} precludes the occurrence of any trivial fixed points (i.e., $(0,0)$ is not allowed as a fixed point) for the recursive equations \eqref{eq:state_evol}. Consequently, starting the AMP iterations from uninformative starting values also results in non-trivial solutions. Further, we note here that the system of equations \eqref{eq:fixed points} could have multiple solutions in general. \revsagc{The values of $\mathsf{MSE^{AMP}_\sigma}$ and $\mathsf{MSE^{AMP}_\beta}$ are determined by the specific fixed points attained by the AMP algorithm. Finally, we emphasize that the uniqueness of the fixed points depends on the priors.}
\end{rem}

\begin{rem}
\revsagc{The condition $\frac{b_p}{p}(1-\frac{b_p}{p})\ge C\,\frac{\log p}{p}$ ensures that the underlying network $\bm G$ is dense enough for the correctness of the state evolution formalism. One expects the SE formalism to break down if the graph is too sparse (e.g., if the average degrees are constant).}

\end{rem}

\begin{rem}
    We analyze the AMP algorithm introduced above using the results of \cite{ma_nandy}. However, we remark that the algorithm can be equivalently analyzed using the powerful general framework introduced in \cite{gerbelot2021graph}. To the best of our knowledge, AMP algorithms with such interacting sets of variables arose originally in \cite{manoel2017multi} in the analysis of multi-layer generalized linear estimation problems.  
\end{rem}

\subsubsection{Statistical Optimality of the Algorithm} 
Having introduced an algorithm based on Approximate Message Passing, we turn to the question of optimal statistical estimation in this setting. In this section, we identify a broad class of settings where the AMP based algorithm introduced above yields optimal statistical performance. 

As a first step, we characterize the limiting mutual information between the underlying signal $(\bm \beta_0, \bm \sigma_0)$ and the observed data. In addition to being a fundamental information theoretic object in its own right, the limiting mutual information will help us characterize the estimation performance of the Bayes optimal estimator in this setup.


To this end, we assume that the conditional distributions $P(\,\cdot\,|\,0\,)$ and $P(\,\cdot\,|\,1\,)$ have finite supports contained in some compact interval $[-s_{\max}, s_{\max}]$. This finite support assumption is merely for technical convenience---we expect that the results can be extended to unbounded, light-tailed (e.g. subgaussian) settings with additional work. Note that the special case $P(\cdot|0)=\delta_0$ corresponds to a discrete case of the classical spike and slab prior. 
Recall that the mutual information between $(\bm \beta_0, \bm \sigma_0)$ and the data represented by $(\bm A,\bm \Phi,\bm y)$ is defined as follows:
\[
\revsag{I(\bm \beta_0,\bm \sigma_0;\bm A,\bm y):= \mathbb{E}_{(\bm \beta_0, \bm \sigma_0, \bm \varepsilon, \bm \Phi, \bm A)}\left[\log\frac{P(\bm A,\bm y|\bm \beta_0, \bm \sigma_0)}{P(\bm A,\bm y)}\right].}
\]
Define 
\[
a:=\sqrt{\mu}\Sigma+\widebar{Z} \quad \mbox{and} \quad y:=B+\sqrt{\frac{\Delta(1+\xi)}{\kappa}}\widebar{\varepsilon},
\]
where $\Sigma \sim \mathrm{Bernoulli}(\rho)$, $B|\Sigma \sim P(\,\cdot\,|\,\Sigma)$, $\widebar{Z},\widebar{\varepsilon} \sim \mathsf{N}(0,1)$ and $\Sigma, B,\widebar{Z}$ and $\widebar{\varepsilon}$ are mutually independent.
We further set 
\[
\mathsf{I}(\mu,\xi;\Delta):=\mathbb{E}\left[\log\frac{P(a,y|\Sigma,B)}{P(a,y)}\right].
\]
Then we have the following theorem characterizing the limiting mutual information.
\begin{thm}
\label{thm:mut_inf_graph}
If $b_p(1-\frac{b_p}{p}) \rightarrow \infty$ as $n,p \rightarrow \infty$, then we have
\begin{align}
\label{eq:mut_inf_graph}
&\mathcal{I}:=\lim\limits_{p \rightarrow \infty}\frac{1}{p}I(\bm \beta_0,\bm \sigma_0;\bm A, \bm y) \nonumber \\
&=  \min\limits_{\mu,\xi\ge 0}\Bigg\{\frac{\lambda\rho^2}{4}+\frac{\mu^2}{4\lambda}+\frac{\kappa}{2}\left[\log\left(1+\xi\right)-\frac{\xi}{1+\xi}\right]- \frac{\mu\rho}{2}+\mathsf{I}(\mu,\xi;\Delta)\Bigg\}.
\end{align}
\end{thm}
\begin{rem}
\revsag{The condition $b_p(1-\frac{b_p}{p}) \rightarrow \infty$ assumed in this theorem ensures that the typical degree in the observed network diverges
as the problem size grows to infinity. We expect this condition to be necessary for this result. If the average graph degrees are bounded, the limiting mutual information should be given by the Bethe prediction on infinite trees \cite{10.1093/acprof:oso/9780198570837.001.0001}. }
\end{rem}

Observe that if $(\widebar{\mu},\widebar{\xi})$ is the global optimizer of the RHS of \eqref{eq:mut_inf_graph}, the first order stationary point conditions imply that $(\widebar{\mu},\widebar{\xi})$ satisfy the fixed point equations 
\begin{align}
\label{eq:first_order_cond}
\widebar{\mu}&=\lambda(\rho-\mathsf{mmse}_1(\widebar{\mu},\widebar{\xi})),\\
\widebar{\xi} &= \frac{1}{\Delta}\mathsf{mmse}_2(\widebar{\mu},\widebar{\xi}).
\end{align}
Recalling \eqref{eq:def_mu_fix_point}, we see that $(\mu^*, \xi^*)$ satisfy the same fixed point system. Of course, the fixed point system can have multiple solutions, and thus these two solutions are not equal in general. 

Our next result establishes that if these fixed points actually coincide, then the AMP based algorithm has Bayes optimal reconstruction performance. 
%
%
%
%
%
%
To this end, define 
\begin{align}
\label{eq:amp_performance} 
    &\mathsf{MMSE}:=\frac{1}{p^2}\mathbb E[\|\bm \sigma_0\bm \sigma^\top_0-\mathbb{E}[\bm \sigma_0\bm\sigma^\top_0|\bm A,\bm y,\bm \Phi]\|^2_F],\\  &\mathsf{y}_{\smmse}:=\frac{1}{n}\mathbb{E}[\|\bm \Phi(\mathbb E[\bm \beta_0|\bm A,\bm y,\bm \Phi]-\bm \beta_0)\|^2_2].
\end{align}
Note that $\mathsf{MMSE}$ and $\mathsf{y}_{\smmse}$ captures the reconstruction performance of the Bayes optimal estimators in this setting. 

\begin{thm}
\label{thm:opt_amp}
Assume that $(\mu^*, \xi^*) = (\widebar{\mu}, \widebar{\xi})$. Then we have 
\[
\lim\limits_{p \rightarrow \infty}\mathsf{MMSE}=\mathsf{MSE^{AMP}_\sigma} \quad \mbox{and} \quad \lim\limits_{p \rightarrow \infty}\mathsf{y_{mmse}}=\mathsf{MSE^{AMP}_\beta}.
\]
\end{thm}
\noindent 
In general, it is hard to check the Assumption in Theorem ~\ref{thm:opt_amp}. If \eqref{eq:def_mu_fix_point} has a unique root, the assumption is trivially satisfied---indeed, this is one prominent example when the above condition can be verified in practice. The difference between $(\mu^*, \xi^*)$ and $(\widebar{\mu}, \widebar{\xi})$ is naturally associated with possible statistical/computational gaps in this problem. We defer a discussion of this point to Section \ref{sec:discussion}.

\subsubsection{Variable discovery and applications to multiple testing} 
\label{variable_discovery}


We turn to the recovery of non-null variables in this section. We will work under the \graphreg model, and assume that $P(\cdot|\sigma_{0i}) = \delta_0\,\mathbb{I}(\sigma_{0i}=0)+\mathrm{Q}\,\mathbb{I}(\sigma_{0i}=1)$, $\mathrm{Q}((-\varepsilon, \varepsilon))=0$ for some $\varepsilon>0$. In this case, discovering the non-null variables is equivalent to recovering the non-zero $\sigma_{0i}$ variables. Specifically, we consider the hypothesis tests
\[
H_{0i}: \sigma_{0i} = 0 \quad \mbox{vs} \quad H_{1i}: \sigma_{0i} \neq 0.
\]

To develop a multiple testing procedure for this setting, we \revsag{shall} employ the Bayes optimal AMP algorithm given by \eqref{eq:amp_iterates}, \revsag{and} \eqref{eq:amp_iterates_1}. Our algorithm is inspired by a similar strategy developed in \cite{montanari2021} in the context of a low-rank matrix recovery problem. 
We use the $\revsag{\bm \sigma^t}$ iterates to devise a Benjamini-Hochberg type multiple testing procedure for variable discovery. 
%
%
Recall the state evolution parameters $\tau^2_t,\nu^2_t$ and $\eta_t$'s defined by \eqref{eq:state_evol}.
Consider the following p-values to test the hypotheses stated above. Define 
\revsag{
\begin{equation}
\label{eq:p_val}
    p_i := 2\,\left(1-\Phi\left(\frac{|\sigma^t_i|}{\nu_t}\right)\right),
\end{equation}
where $\Phi(x):=\int_{-\infty}^{x}\frac{1}{\sqrt{2\pi}}e^{-t^2/2}\,dt$.} Our next result establishes that the p-values introduced above are actually asymptotically valid. 

\begin{thm}
    \label{thm:P_val_0}
    Consider the $P$-values defined by \eqref{eq:p_val}. Then, if $\sigma_{0i_0}=0$, for any $\alpha \in [0,1]$,
    \[
       \lim\limits_{n \rightarrow \infty}\mathbb{P}\left[p_{i_0}(t) \le \alpha\right] = \alpha.
    \]
\end{thm}

Using these p-values we shall design a Benjamini-Hochberg type procedure (\citet{10.2307/2346101}) for variable selection, which controls the false discovery proportion at level $\alpha$. To this end, we define the following estimator of false discovery proportion:

\[\revsag{
\widehat{\mathrm{FDP}}(s;t):=\frac{p(1-\rho)s}{1 \vee \left\{\sum_{i=1}^{p}\mathbb{I}\,(p_i(t) \le s)\right\}}.}
\]

Define the threshold $s_*(\alpha;t)$ given by

\[
s_*(\alpha;t):= \inf\left\{s \in [0,1]:\widehat{\mathrm{FDP}}(s;t) \ge \alpha\right\}.
\]

We reject the hypothesis $H_{0i}$ if $p_i(t)<s_*(\alpha;t)$. Let the set of rejected hypotheses be defined as $\hat{S}(\alpha;t)$. Consider the false discovery rate given by

\[
\mathrm{FDR}(\alpha,t;n):=\mathbb{E}\left[\frac{|\hat{S}(\alpha;t) \cap \{i:\sigma_{0i}=0\}|}{1 \vee |\hat{S}(\alpha;t)|}\right].
\]

The following theorem establishes that the $\mathrm{FDR}(\alpha,t;n)$ is asymptotically $\alpha$.

\begin{thm}
    \label{thm:P_val}
    For any $t \geq 0$, 
    \[
       \lim\limits_{n \rightarrow \infty}\mathrm{FDR}(\alpha,t;n) = \alpha.
    \]
\end{thm}


In addition, it is often helpful to have credible intervals to quantify the uncertainty involved in recovering the $\sigma_{0i}$ variables. To this end, we construct the following credible sets for $\sigma_{0i}$'s based on $\nu_t,\eta_t$ and $\sigma^t_i$.

\begin{equation}
\label{eq:cred_set}
    \hat{J}_i(\alpha,t) = \Big[\frac{1}{\eta_t}\sigma^t_i-\frac{\nu_t}{\eta_t}\Phi^{-1}(1-\alpha/2),\frac{1}{\eta_t}\sigma^t_i+\frac{\nu_t}{\eta_t}\Phi^{-1}(1-\alpha/2)\Big].
\end{equation}
Our next result establishes that $(1-\alpha)$ fraction of the credible intervals contain the true $\sigma_i$ variables. 
\begin{thm}
    \label{thm:cred_set}
    Consider the credible sets defined by \eqref{eq:cred_set}. Then, almost surely
    \[
       \lim\limits_{p \rightarrow \infty}\frac{1}{p}\sum_{i=1}^{p}\mathbb{I}(\sigma_{0i} \in \hat{J}_i(\alpha,t)) = 1-\alpha.
    \]
\end{thm}
Observe that, coupled with the Dominated Convergence Theorem, Theorem \ref{thm:cred_set} implies that on average the coverage probability of the credible sets defined by \eqref{eq:cred_set} is $1-\alpha$. In other words,
\[
\lim\limits_{p \rightarrow \infty}\frac{1}{p}\sum_{i=1}^{p}\mathbb{P}(\sigma_{0i} \in \hat{J}_i(\alpha,t)) = 1-\alpha.
\]
The proof of Theorems \ref{thm:P_val}-\ref{thm:cred_set} follows using similar techniques as \citet{montanari2021}. Hence we omit these proofs for the sake of avoiding repetition.







\subsection{Prior Art} 
The problem studied in this paper and our approach overlap with two distinct lines of research in high-dimensional statistics. On the one hand, the meta question of incorporating network side information into supervised learning has been explored in the prior literature. In stark contrast to our approach, the prior works do not use a joint model for the whole data. Instead, the network information is usually incorporated directly into the estimation scheme via an intuitive penalization procedure. We note that while this strategy is intuitive, it is quite ad hoc, and it is difficult to formulate the question of optimal estimation and recovery without a joint generative model. We review the relevant results in this direction in Section \ref{sec:network_side}. 

Our AMP based approach is also related to the broader theme of incorporation of side-information into AMP algorithms. We review the progression in this direction in Section \ref{sec:amp_side}.

\subsubsection{Regression with network side information} 
\label{sec:network_side} 

Regression problems with network side information have been investigated in high-dimensional statistics and bioinformatics, often with the goal of incorporating relevant biological information into the inference procedure. For example, in a genomics setting, the network often represents pairwise relations between genes that are commonly co-expressed. It is natural to believe that successful incorporation of this side information should yield biologically interpretable models. We discuss below some prominent approaches that have been proposed in the literature, and compare them to the approach adopted in this paper.

\begin{itemize}
    \item[(i)] Penalization-based approaches: Network side information can be naturally incorporated into the estimation procedure through a suitably designed penalty parameter in an empirical risk minimization framework. In this spirit, \citet{li_and_li}, \cite{li2010variable} employ a penalty based on the graph laplacian. This penalty promotes smoothness of the estimation coefficients along the edges of the observed network. The performance of related penalization-based methods was rigorously explored in follow-up work in settings with sparsity in the underlying coefficients \cite{huang2011sparse}. The idea of penalization using the graph heat kernel has been revisited recently in  \cite{ghosh2022learning}.  
    In a similar vein, an $\ell_1+ \ell_2$ penalty based on the graph structure was also examined recently in  \cite{tran2022generalized}.

    We note that similar graph-based penalization approaches have also been studied extensively in the context of total variation denoising  (see e.g. \cite{rudin1992nonlinear,tibshirani2005sparsity,sharpnack2012sparsistency,wang2015trend,smola2003kernels,hutter2016optimal,padilla2017dfs} and references therein). However, one typically assumes that the data arises from a Gaussian sequence model (rather than a regression model) in this line of work. 
    

    \item[(ii)] In the bioinformatics literature, the graph laplacian  has also been employed, although via a different approach. The HotNet2 method \cite{leiserson2015pan}, a canonical procedure in this context, uses the graph laplacian to propagate univariate association measures (between the response $\bm{y}$ and individual features \revsag{$\bm \phi_i$}) along the network. These propagated scores are subsequently used to detect biologically interpretable subnetworks. This method has been successfully employed in several applications with network side information (see e.g. \cite{hoadley2014multiplatform,cancer2014comprehensive,agrawal2014integrated}). 
\end{itemize}

The existing approaches suffer from some specific shortcomings: 
\begin{itemize}
    \item[(i)] Existing methods assume that the network information is observed without measurement error. 
    Unfortunately, in many settings, the network observation itself is noisy and incomplete---the stochastic model assumed on the observed network might be more relevant in such settings.  
    
    \item[(ii)] While \revsag{the} existing strategies naturally incorporate the network information via appropriate penalties, these methods cannot directly incorporate additional structural information about the regression coefficients. For example, if it is known that individual regression coefficients are $\{\pm 1\}$ valued, one would naturally try to incorporate this constraint into the associated optimization problem. However, the resulting discrete optimization problem is typically intractable. While this issue can be potentially tackled using a subsequent convex relaxation, the statistical properties of the resulting estimator are not obvious. In sharp contrast, the Bayesian framework introduced in this paper directly incorporates structural information regarding the regression coefficients through the prior.

    
    \item[(iii)] While several strategies have been proposed in the prior literature for incorporating the additional network information, it is difficult to determine an optimal strategy for estimation and variable discovery. The question of optimal recovery is particularly important in the context of the biological applications outlined above---a non-trivial gain in the statistical power could create a critical difference in terms of a practical impact. Our framework is particularly useful in this context; one can rigorously study the optimality of statistical algorithms under the Bayesian framework introduced in this paper. 
    
    \item[(iv)] The optimization step in penalization based approaches can be quite slow when $n$ and $p$ are both large. On the other hand, the AMP based iterative algorithms discussed in this paper scale efficiently to large problem dimensions. 
\end{itemize}

\subsubsection{Approximate Message Passing with side information}
\label{sec:amp_side} 
Approximate Message Passing (AMP) algorithms were introduced originally in the context of compressed sensing (\citet{BM11journal,6566160,10.5555/3454287.3455127,https://doi.org/10.48550/arxiv.2110.09502}), but have found broad applications in many high-dimensional inference problems (e.g. linear and generalized linear models, low-rank matrix estimation, sparse codes, etc.). AMP algorithms are simple iterative algorithms that employ a set of specific non-linearities at each step. The statistical performance of these algorithms can be tracked using low-dimensional scalar recursions referred to as \emph{state evolution}. This makes the algorithms theoretically tractable, and further encourages practical deployment. We refer the interested reader to \cite{feng2022unifying} for a recent survey of these algorithms and their applications. 

Given some side information about the variables of interest, it is natural to incorporate this side information into the estimation procedure---this typically improves the estimation performance and associated downstream statistical performance. The incorporation of side information in AMP algorithms was  considered in \citet{Rush}. These results were later extended to a more general setup by the first author in \citet{ma_nandy}. We use the results of \cite{ma_nandy} to analyze the AMP algorithm, but emphasize that this analysis could be equivalently performed using the powerful framework introduced in \cite{gerbelot2021graph}. Related results also appear in the recent manuscript \citet{wang_zhong_fan}. 

We note that the use of side information is also critical in the evaluation of the limiting free energy in these models. Once a vanishing amount of side information is added to the model, one can establish concentration of the \emph{overlap}. This is critical for an application of interpolation methods employed for the evaluation of the limiting mutual information. This idea has been introduced, and exploited  in fundamental past works in the area (see e.g. \cite{barbier2019optimal, Adaptive_Interpolation} and references therein) and is also used crucially in our analysis. 

\paragraph{Organization} 
The rest of the paper is organized as follows. We present the proof of Theorem \ref{thm:mse_amp} in Section \ref{pf_amp}. We explore the finite sample performance of our algorithm in Section \ref{numerical}. In addition, we also compare the algorithm to existing penalization based approaches, and provide evidence for the robustness of our method to distributional assumptions. The proofs of the remaining technical results are deferred to the Appendix.

\section{Proof of Theorem \ref{thm:mse_amp}}
\label{pf_amp}
To prove Theorem \ref{thm:mse_amp} we shall first characterize the state evolution of the AMP iterates defined by \eqref{eq:amp_iterates} and \eqref{eq:amp_iterates_1}. In fact, we shall characterize the state evolution of a related sequence of AMP iterates with a Gaussian sensing matrix instead of a sensing matrix which is the adjacency matrix of a SBM.
\paragraph{State Evolution of AMP Iterates with Gaussian Sensing Matrix}
\revsag{Let us consider $\lambda$ as defined in \eqref{eq:def_snr} and the following matrix defined as:}
\begin{equation}
\label{eq:Gaussian_model}
    \widetilde{\bm A}=\sqrt{\frac{\lambda}{p}}\bm \sigma_0 \bm \sigma^\top_0 + \bm Z,
\end{equation}
where \revsag{$Z_{ij}=Z_{ji}\sim \mathsf{N}(0,1)$ if $i \neq j$ and $Z_{ii} \sim \mathsf{N}(0,2)$}. Let us consider the same AMP iterates as \eqref{eq:amp_iterates} and \eqref{eq:amp_iterates_1}, except we replace the matrix $\widebar{\bm A}$ by $\widetilde{\bm A}$. In other words, we consider the sequence of iterates $\{\widebar{\bm \sigma}^{t+1}\}_{t \ge 0}$ and $\{\widebar{\bm z}^t, \widebar{\bm \beta}^{t+1}\}_{t \ge 0}$ defined as follows.
\begin{align}
\label{eq:amp_iterates_2}
\widebar{\bm \sigma}^{t+1} &= \frac{\widetilde{\bm A}}{\sqrt{p}}\bm f_t(\widebar{\bm \sigma}^t,\bm S^\top\widebar{\bm z}^{t-1}+\widebar{\bm \beta}^{t-1})\\
&\quad\quad-(\mathcal A\bm f_t)(\widebar{\bm \sigma}^t,\bm S^\top\widebar{\bm z}^{t-1}+\widebar{\bm \beta}^{t-1})\bm f_{t-1}(\widebar{\bm \sigma}^{t-1},\bm S^\top\widebar{\bm z}^{t-2}+\widebar{\bm \beta}^{t-2}),\\
\end{align}
and,
\begin{align}
\label{eq:amp_iterates_3}
\widebar{\bm z}^t & = \bm y^\circ - \bm S\widebar{\bm \beta}^t+\frac{1}{\kappa}\widebar{\bm z}^{t-1}(\mathcal A\bm \zeta_{t-1})(\bm S^\top\widebar{\bm z}^{t-1}+\widebar{\bm \beta}^{t-1},\widebar{\bm \sigma}^{t})\\
\widebar{\bm \beta}^{t+1} &= \bm\zeta_t(\bm S^\top\widebar{\bm z}^t+\widebar{\bm \beta}^t,\widebar{\bm \sigma}^{t+1}).\\
\end{align}
Here, the definitions of all other terms are the same as that in \eqref{eq:amp_iterates} and \eqref{eq:amp_iterates_1}. Now, let us define a \emph{pseudo-Lipschitz functions} $\bm f$ as in (1.5) of \citet{BM11journal}.
\begin{defn}
Consider $\bm{a}=(a_1,\cdots,a_{k})^\top$ and $\bm{b}=(b_1,\cdots,b_{k})^\top$. 
A function $f:\mathbb{R}^{k} \rightarrow \mathbb{R}$ is called 
\emph{pseudo-Lipschitz} if there is an absolute constant $C > 0$ such that for all $\bm a, \bm b \in \mathbb{R}^k$
\begin{equation}
|f(\bm a)-f(\bm b)| \le C(1+\|\bm a\|+\|\bm b\|)\|\bm a -\bm b\|.
\end{equation}
\end{defn}
With this definition, we have the following theorem that describes the state evolution of the AMP iterates given by \eqref{eq:amp_iterates_2} and \eqref{eq:amp_iterates_3}.
\begin{thm}
\label{thm:state_evol}
Let the functions $f_t, \eta_t$ in \eqref{eq:amp_iterates_2} and \eqref{eq:amp_iterates_3} be Lipschitz with Lipschitz derivatives. Then, for pseudo-Lipschitz functions $\tilde{\psi}: \mathbb{R}^4 \rightarrow \mathbb{R}$, $\tilde{\phi}:\R^2 \rightarrow \R$ and the AMP iterates given by $\{\widebar{\bm \sigma}^{t+1}\}_{t \ge 0}$ and $\{\widebar{\bm z}^t, \widebar{\bm \beta}^{t}\}_{t \ge 0}$, we get the following relations.
\begin{align}
\lim\limits_{p \rightarrow \infty}\frac{1}{p}\sum\limits_{i=1}^{p}\widetilde{\psi}([\bm S^\top\widebar{\bm z}^t+\widebar{\bm \beta}^t]_i,\widebar{\sigma}^{t+1}_i,\beta_{0i},\sigma_{0i})&\overset{a.s.}{=}\mathbb{E}[\widetilde{\psi}(B+\tau_t Z_1,\eta_{t+1}\Sigma+\nu_{t+1}Z_2,B,\Sigma)],\\
\lim\limits_{n \rightarrow \infty}\frac{1}{n}\sum\limits_{i=1}^{n}\widetilde{\phi}(\widebar{z}^t_i,\bar{\varepsilon}_i)&\overset{a.s.}{=}\mathbb{E}[\widetilde{\phi}(W+\widetilde{\tau}_tZ_3,W)],
\end{align}
where $\tau_t,\eta_t,\nu_t$ are defined before in \eqref{eq:state_evol} and $\widetilde{\tau}^2_t=\tau^2_t-\Delta/\kappa$. Additionally, the random variables $Z_1,Z_2,Z_3 \sim \mathsf{N}(0,1),\, W \sim \mathsf{N}(0,\kappa^{-1}),\, B \sim P\,(\,\cdot\,|\,\Sigma),\, \Sigma \sim \mathrm{Bernoulli}(\rho)$ and all the random variables are mutually independent. Further, $\bar{\bm \varepsilon}=\bm \varepsilon/\sqrt{\kappa}$, where $\bm \varepsilon$ is defined in \eqref{eq:lin_model_1}.
\end{thm}
\begin{proof}
Let us begin by defining 
\begin{align}
\bm h^{t+1} &= \bm \beta_0-(\bm S^\top \bar{\bm z}^t+\widebar{\bm \beta}^t)\\
\bm q^t &= \widebar{\bm \beta}^t -\bm \beta_0\\
\bm e^t&=\bm y^\circ-\bm S\bm\beta_0-\widebar{\bm z}^t\\
\bm m^t&=-\widebar{\bm z}^t\\
\bm r^t&=\bm f_t(\widebar{\bm \sigma}^t,\bm \beta_0-\bm h^t).
\end{align}
Further, let us define
\begin{align}
\ell_t(s,r,x_0)&=\zeta_{t-1}(x_0-s,r)-x_0\\
g_t(s,w)&=s-w.
\end{align}
Let $\bm \ell_t:\R^{3p} \rightarrow \R^p$ and $\bm g_t:\R^{2n} \rightarrow \R^n$ be defined by applying $\ell_t$ and $g_t$ componentwise to the arguments of the functions. Then, from \eqref{eq:amp_iterates_3} we can observe that
\[
\bm q^t = \bm \ell_t(\bm h^t,\widebar{\bm \sigma}^{t+1}, \bm \beta_0), \quad \mbox{and} \quad \bm m^t=\bm g_t(\bm e^t,\widebar{\bm \varepsilon}).
\]
Now, we can rewrite the AMP iterates as,
\begin{align}
\label{eq:gaussian_AMP}
\bm q^t&=\bm \ell_t(\bm h^t,\widebar{\bm \sigma}^{t+1},\bm \beta_0),\\
\bm e^t&= \bm S \bm q^t-\lambda_t\bm m^{t-1},\\
\bm m^t&=\bm g_t(\bm e^t,\widebar{\bm \varepsilon}),\\
\bm h^{t+1} & = \bm S^\top\bm m^t-c_t\bm q^t,
\end{align}
and
\begin{align}
\label{eq:gaussian_AMP_1}
\bm r^{t} & =\bm f_t(\widebar{\bm \sigma}^{t},\bm \beta_0-\bm h^{t-1}),\\
\widebar{\bm \sigma}^{t+1}&= \frac{\widetilde{\bm A}}{\sqrt{p}}\bm r^{t}-\upsilon_t\bm r^{t-1},
\end{align}
where
\begin{equation}
\label{eq:onsager_terms}
    c_t=\frac{1}{n}\sum_{i=1}^{n}g^\prime_t(b^t_i,\varepsilon_i), \quad  \lambda_t=\frac{1}{p\kappa}\sum_{i=1}^p \ell^\prime_t(h^t_i,\widebar{\sigma}^{t+1}_i,\beta_{0i}), \quad \upsilon_t=\frac{1}{p}\sum_{i=1}^p f^{\prime}_t(\widebar{\sigma}^t_i,\beta_{0i}-h^t_i).
\end{equation}
In \eqref{eq:onsager_terms}, $g^\prime_t,\ell^\prime_t$ and $f^{\prime}_t$ are derivatives of $g_t,\ell_t$ and $f_t$ respectively with respect to their first coordinates and the initialization of the iterations is as in \eqref{eq:amp_iterates_graph}. Now using the techniques of \cite[Theorem 7.2]{ma_nandy}, which was in turn derived using the ideas developed in \cite[Theorem 1]{BM11journal}, we can conclude that for any pseudo-Lipschitz functions $\psi:\R^4 \rightarrow \R$ and $\phi: \R^2 \rightarrow \R$,
\begin{align}
    \lim_{p \rightarrow \infty}\frac{1}{p}\sum_{i=1}^{p}\psi(h^t_i,\widebar{\sigma}^{t+1}_i,\beta_{0,i},\sigma_{0i}) &\overset{a.s}{=} \mathbb E[\psi(\tau_tZ_1,\eta_{t+1}\Sigma+\nu_{t+1}Z_2,B,\Sigma)],\\
    \lim_{n \rightarrow \infty}\frac{1}{n}\sum_{i=1}^{n}\phi(e^t_i,\widebar{\varepsilon}_i) &\overset{a.s}{=} \mathbb E[\phi(\widetilde{\tau}_tZ_3,W)],
\end{align}
where $Z_1,Z_2,Z_3 \sim \mathsf{N}(0,1)$, $W \sim \mathsf{N}(0,\kappa^{-1})$, $B \sim P(\,\cdot\,|\,\Sigma)$ and $\Sigma \sim \mathrm{Bernoulli}(0,\rho)$ with all the random variables being mutually independent. Further, the parameters $\{\eta_t,\nu^2_t,\tau^2_t,\widetilde{\tau}^2_t\}_{t \ge 0}$ are defined as follows:
\begin{align}
\label{eq:proof_state_evol}
\nu^2_{t+1} &= \mathbb{E}[f^2_{t}(\eta_{t}\Sigma+\nu_{t}\widetilde{Z}_3,B+\tau_{t-1}\widetilde{Z}_4)],\\
\eta_{t+1} & = \sqrt{\lambda}\nu^2_{t+1},\\
\widetilde{\tau}^2_t &=\mathbb E[\ell^2_t(\tau_{t}\widetilde{Z}_1,\eta_{t+1}\Sigma+\nu_{t+1}\widetilde{Z}_2, B)]=\frac{1}{\kappa}\mathbb{E}[(\zeta_{t-1}(B+\tau_{t}\widetilde{Z}_1,\eta_{t+1}\Sigma+\nu_{t+1}\widetilde{Z}_2)-B)^2],\\
\tau^2_{t+1} & = \mathbb E[(\widetilde{Z}_5\tau^2_t-\widetilde{W})^2],
\end{align}
where $\widetilde{Z}_1,\widetilde{Z}_2,\widetilde{Z}_3,\widetilde{Z}_4,\widetilde{Z}_5 \sim \mathsf{N}(0,1)$, $\widetilde{W} \sim \mathsf{N}(0,\kappa^{-1})$, $B \sim P(\,\cdot\,|\,\Sigma)$ and $\Sigma \sim \mathrm{Bernoulli}(\rho)$, with the random variables being independent. Finally, if we take 
\begin{align}
\psi(h^t_i,\widebar{\sigma}^{t+1}_i,\beta_{0i},\sigma_{0i})&=\widetilde{\psi}(\beta_{0i}-h^t_i,\widebar{\sigma}^{t+1}_i,\beta_{0i},\sigma_{0i}),\\
\phi(e^t_i,\widebar{\varepsilon}_i) &= \widetilde{\phi}(\widebar{\varepsilon}_i-e^t_i,\widebar{\varepsilon}_i),
\end{align}
then the theorem follows.

\end{proof}
\paragraph{State Evolution of Graph Based AMP Iterates}
To obtain similar results for graph-based AMP iterates, let us recognize that
\begin{equation}
\label{eq:rank_one_graph_deformation}
   \widebar{\bm A}=\sqrt{\frac{\lambda}{p}}\bm\sigma_0\bm\sigma^\top_0+\bm{W}, 
\end{equation}
where $\bm W$ is a symmetric matrix with entries satisfying 
\[
\revsag{\mathbb E(W_{ij})=0, \quad \mathbb E(W^2_{ij})\in\left\{\frac{a_p(1-a_p/p)}{p\widebar{d}_p(1-\widebar{d}_p)},\frac{b_p(1-b_p/p)}{p\widebar{d}_p(1-\widebar{d}_p)}\right\} \quad \mbox{and} \quad |W_{ij}| \le \frac{1}{\sqrt{\widebar{d}_p(1-\widebar{d}_p)}}.}
\]
\revsag{Further, under the assumption $\widebar{d}_p(1-\widebar{d}_p) \rightarrow \infty$, $S_{ij}=\mathbb E(W^2_{ij})$ converges to $1$ uniformly on $[p]\times[p]$.} This implies $\bm W/\sqrt{p}$ is a generalized Wigner matrix in the sense of Definition 2.3 of \citet{wang_zhong_fan}. \revsag{Moreover under the assumption $\widebar{d}_p(1-\widebar{d}_p)\ge C\,\log p/p$ we can show using Theorem 2.7 of \citet{10.1214/19-AIHP1033} and (2.4) of \citet{wang_zhong_fan} that $\|\bm W\|_{\mathsf{op}} \le C_1\sqrt{p}$ for an absolute constant $C_1>0$ almost surely for large $n,p$.} This implies we satisfy the assumptions of Theorem 2.4 of \citet{wang_zhong_fan}. Hence, we can combine the proof techniques of Theorem 2.4 of \citet{wang_zhong_fan} and those of Theorem \ref{thm:state_evol} to get the following result.
\begin{thm}
\label{thm:state_evol_graph}
Let us assume $\widebar{d}_p(1-\widebar{d}_p)\ge C\,\log p/p$ for some constant $C>0$. Further, let the functions $f_t, \zeta_t$ in \eqref{eq:amp_iterates_2} and \eqref{eq:amp_iterates_3} be Lipschitz with Lipschitz derivatives. Then, for the pseudo-Lipschitz functions $\psi: \mathbb{R}^4 \rightarrow \mathbb{R}$, $\phi: \mathbb{R}^2 \rightarrow \mathbb{R}$ and the AMP iterates $\{\bm \sigma^{t+1}\}_{t \ge 0}$ and $\{\bm z^t, \bm \beta^{t}\}_{t \ge 0}$ defined in \eqref{eq:amp_iterates} and \eqref{eq:amp_iterates_1}, we get the following relations.
\begin{align}
\lim\limits_{p \rightarrow \infty}\frac{1}{p}\sum\limits_{i=1}^{p}\psi([\bm S^\top\bm z^t+\bm \beta^t]_i,\sigma^{t+1}_i,\beta_{0i},\sigma_{0i})&\overset{a.s.}{=}\mathbb{E}[\psi(B+\tau_t Z_1,\eta_{t+1}\Sigma+\nu_{t+1}Z_2,B,\Sigma)],\\
\lim\limits_{n \rightarrow \infty}\frac{1}{n}\sum\limits_{i=1}^{n}\phi(z^t_i,\widebar\varepsilon_i)&\overset{a.s.}{=}\mathbb{E}[\phi(W+\widetilde{\tau}_tZ_3,W)],
\end{align}
where $Z_1,Z_2,Z_3,W,\Sigma,B$ are defined before and $\tau^2_t,\widetilde \tau^2_t,\nu^2_t$ and $\eta_t$ are defined in \eqref{eq:proof_state_evol}.
\end{thm}
Since the proof of this theorem repeats the arguments of the proofs of Theorem \ref{thm:state_evol} and Theorem 2.4 of \citet{wang_zhong_fan}, we relegate its details to Section \ref{proof_amp_graph} of the appendix.
\paragraph{Asymptotics of \revsag{the reconstruction errors}}
To complete the proof of Theorem \ref{thm:mse_amp}, let us begin by observing the following equation.
\begin{align}
\frac{1}{p^2}\|\bm \sigma_0\bm\sigma^\top_0-\widehat{\bm \sigma}^t(\widehat{\bm \sigma}^t)^\top\|^2_F &= \frac{1}{p^2}\left(\sum\limits_{i=1}^{p}\sigma^2_{0i}\right)^2-2\,\frac{1}{p^2}\left(\sum\limits_{i=1}^{n}\sigma_i\widehat{\sigma}^t_i\right)^2+\frac{1}{p^2}\left(\sum\limits_{i=1}^{n}(\widehat{\sigma}^t_i)^2\right)^2.
\end{align}
By Theorem \ref{thm:state_evol_graph} we have the following.
\begin{align}
\lim\limits_{p \rightarrow \infty}\frac{1}{p}\sum\limits_{i=1}^{p}\sigma_{0i}\widehat{\sigma}^t_i &\overset{a.s}{=} \mathbb E\Bigg[\Sigma\;\mathbb E\{\Sigma|B+\sqrt{\Delta(1+\xi_{t-1})/\kappa}\,Z_1,\mu_{t}\Sigma+\nu_{t}Z_2\}\Bigg]\\
&= \mathbb E\Bigg\{\mathbb E[\Sigma|B+\sqrt{\Delta(1+\xi_{t-1})/\kappa}\,Z_1,\mu_{t}\Sigma+\nu_{t}Z_2]\Bigg\}^2\\
&\overset{a.s}{=}\lim\limits_{p \rightarrow \infty}\frac{1}{p}\sum\limits_{i=1}^{n}(\widehat{\sigma}^t_i)^2.
\end{align}
Now using the Strong law of Large numbers and the Dominated Convergence Theorem and the definitions of $\mu_t$ and \revsag{$\xi_t$ from \eqref{eq:fixed_point_iteration}}, we have
\[
\lim\limits_{p \rightarrow \infty}\frac{1}{p^2}\mathbb E[\|\bm \sigma_0\bm\sigma^\top_0-\widehat{\bm \sigma}^t(\widehat{\bm \sigma}^t)^\top\|^2_F] = (\mathbb{E}[\Sigma^2])^2-\frac{\revsag{\mu^2_{t+1}}}{\lambda^2}.
\]
Taking the limit as $t \rightarrow \infty$ and using \eqref{eq:fixed points}, we get the theorem.

Next, let us define
\[
\revsag{\omega^{t-1}=\frac{1}{\kappa p} \sum\limits_{i=1}^{p}\zeta^\prime_{t-1}([\bm S\bm z^{t-1}+\bm \beta^{t-1}]_i,\sigma^{t}_i),}
\]
where the derivative is with respect to the first co-ordinate of the $\zeta_{t-1}$ and
\[
\omega^{*}=\lim\limits_{t \rightarrow \infty}\lim\limits_{p \rightarrow \infty}\omega^{t-1}.
\]
This limit exists inductively using Theorem \ref{thm:state_evol_graph} and the property of the AMP state evolution.
We also have
\[
\lim\limits_{t \rightarrow \infty}\lim\limits_{n \rightarrow \infty}\frac{1}{n}\|\bm z^t-\bm z^{t-1}\|^2_2 \overset{a.s.}{=}0.
\]
Let us observe that by Theorem \ref{thm:state_evol_graph} and the Strong Law of Large Numbers
\begin{align}
\lim\limits_{t \rightarrow \infty}\lim\limits_{n \rightarrow \infty}\frac{1}{\kappa n}\|\bm \Phi(\bm \beta^t-\bm \beta_0)\|^2_2&=\lim\limits_{t \rightarrow \infty}\lim\limits_{n \rightarrow \infty}\frac{1}{n}\left\|\frac{1}{\sqrt{\kappa}}\bm\varepsilon-\bm z^t+\omega^{t-1}\bm z^{t-1}\right\|^2_2\\
&=\lim\limits_{n \rightarrow \infty}\frac{1}{\kappa n}\|\bm\varepsilon\|^2_2+(\omega^{*}-1)^2\lim\limits_{t \rightarrow \infty}\lim\limits_{n \rightarrow \infty}\frac{1}{n}\|\bm z^{t-1}\|^2_2\\
&\quad\quad+2(\omega^*-1)\lim\limits_{t \rightarrow \infty}\lim\limits_{n \rightarrow \infty}\frac{1}{n}\langle\bm \widebar{\bm \varepsilon},\bm z^{t-1}\rangle\\
&\overset{a.s.}{=}\frac{\Delta}{\kappa}+(\omega^{*}-1)^2\lim\limits_{t \rightarrow \infty}\mathbb E[(W+\widetilde\tau_{t-1}Z_3)^2]\\
&\hskip 2em+2(\omega^*-1)\lim\limits_{t \rightarrow \infty}\mathbb E[W(W+\widetilde\tau_{t-1}Z_3)]\\
&=\frac{\Delta}{\kappa}+(\omega^*-1)^2(\frac{\Delta}{\kappa}+\lim\limits_{t \rightarrow \infty}\widetilde \tau^2_{t-1})+2(\omega^*-1)\frac{\Delta}{\kappa}\\
&=\frac{\Delta}{\kappa}+(\omega^*-1)^2\frac{\Delta}{\kappa}(1+\xi^*)+2(\omega^*-1)\frac{\Delta}{\kappa}.
\end{align}
It can be verified that
\begin{align}
\omega^*&=\lim\limits_{t\rightarrow\infty}\lim\limits_{n \rightarrow \infty}\frac{1}{\kappa p} \sum\limits_{i=1}^{p}\zeta^\prime_{t-1}([\bm S^\top\bm z^{t-1}+\bm \beta^{t-1}]_i,\sigma^{t}_i)\\
&=\lim\limits_{t\rightarrow\infty}\frac{\widetilde\tau^2_{t-1}}{\tau^2_{t-1}}=\frac{\xi^*}{1+\xi^*}.
\end{align}
\revsagr{Plugging in \eqref{eq:connect_amp} and using the definition of $\widehat{\bm \beta}^t$, we get
\[
\lim\limits_{t \rightarrow \infty}\lim\limits_{n \rightarrow \infty}\frac{1}{n}\;\|\bm\Phi(\widehat{\bm \beta}^t-\bm \beta_0)\|^2_2= \frac{\Delta\;\xi^*}{(1+\xi^*)},
\]
almost surely.} 
\revsagr{
Furthermore, under the assumption that $P(\cdot|0)$ and $P(\cdot|1)$ have finite and bounded supports, the random variable,
\[
\lim\limits_{n \rightarrow \infty}\frac{1}{\kappa n}\;\|\bm\Phi(\widehat{\bm \beta}^t-\bm \beta_0)\|^2_2
\]
is uniformly bounded by an integrable random variable. Therefore, using the Dominated Convergence Theorem
\begin{align}
\label{eq:connect_amp}
\lim\limits_{t \rightarrow \infty}\lim\limits_{n \rightarrow \infty}\frac{1}{\kappa n}\mathbb E\|\bm \Phi(\widehat{\bm \beta}^t-\bm \beta_0)\|^2_2=\frac{\Delta\;\xi^*}{(1+\xi^*)}.
\end{align}
}

\section{Numerical Experiments}
\label{numerical}
In this section, we explore the finite sample performance of the proposed AMP-based methodology and further explore the consequences of this supporting theory. 

\begin{itemize}
    \item[(i)] In Section \ref{subsec_effect_graph}, we first explore the information-theoretic consequence of having the graph side information. To this end, we compare the limiting mutual information in the model with graph side information to one with no additional side information. This characterizes the information-theoretic gain in incorporating the graph side information. 
    
    \item[(ii)] In Section \ref{sec:comparison}, we compare the performance of our AMP-based algorithm to Laplacian penalized estimators, proposed in \cite{li_and_li}. Our findings indicate that the AMP-based algorithm significantly outperforms the Laplacian penalization-based method. We also explore the robustness of our findings to the distribution of the design in this section. 
    
    \item[(iii)] Finally, we explore the variable discovery performance of the AMP-based method. In particular, we compare the method to the Knockoff filter of \cite{barber2015controlling,candes2018panning}. We note that the Knockoff filter has emerged as the canonical method for variable discovery in the linear model. Our results indicate that the AMP-based method outperforms the Knockoff filter by incorporating the graph side information. Note that the Knockoff filter completely ignores the graph side information, and thus it should not be a surprise that our method outperforms the Knockoff filter. However, we believe it illustrates the power of leveraging network-side information for variable discovery. 
\end{itemize}



\subsection{Information theoretic effect of the graph side information} 
\label{subsec_effect_graph} 
For our experiments, let \revsag{$\bm \sigma_0=(\sigma_{01},\ldots,\sigma_{0p})^\top$ satisfy $\sigma_{0i} \overset{i.i.d}{\sim} \mathrm{Bernoulli}(0.4)$}. Given \revsag{$\sigma_{0i}=0$ we set $\beta_{0i}=0$} with probability $1$ and if \revsag{$\sigma_{0i}=1$, then $\beta_{0i}$} is generated from the discrete distribution that puts mass $1/5$ on each of $\{-2,-1,0,1,2\}$. We assume that $\kappa=1.5$ and  
plot the asymptotic per mutual information between \revsag{$(\bm \beta_0,\bm \sigma_0)$} and $(\bm y, \bm G)$ when the graph $\bm G$ is observed versus when it is not observed. Note that not observing the graph is equivalent to $\lambda=0$ in terms of the mutual information between \revsag{$(\bm \beta_0,\bm \sigma_0)$} and $(\bm y, \bm G)$. So we fix $\lambda=0, 1, 2$ and $3$ and plot the asymptotic per mutual information between \revsag{$(\bm \beta_0,\bm \sigma_0)$} and $(\bm y, \bm G)$ in Figure \ref{fig:inf_del}.

\begin{figure}[tb]
\begin{subfigure}{.5\textwidth}
  \centering
  \includegraphics[width=\linewidth]{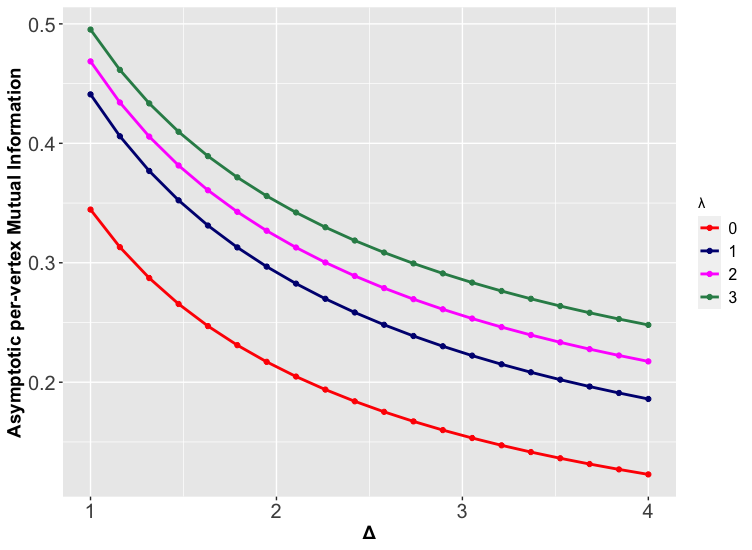}
  \caption{Mutual information vs $\Delta$}
  \label{fig:inf_del}
\end{subfigure}%
\begin{subfigure}{.5\textwidth}
  \centering
  \includegraphics[width=\linewidth]{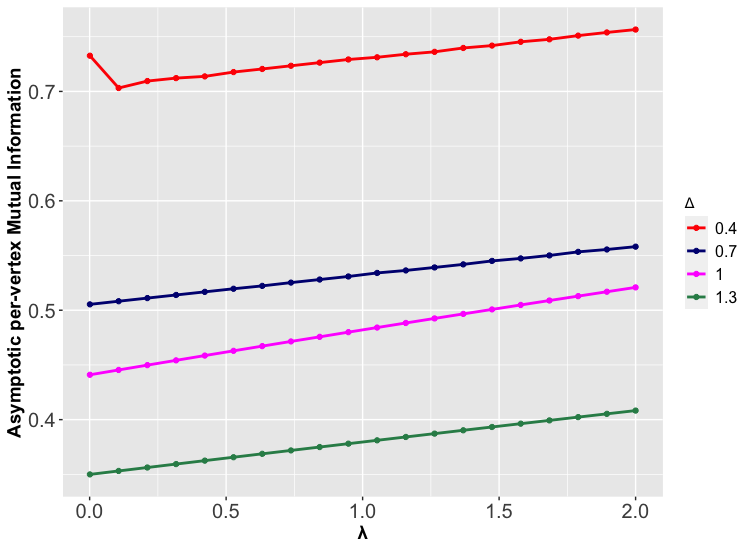}
  \caption{Mutual information vs $\lambda$}
  \label{fig:inf_lam}
\end{subfigure}%
\caption{Plot of asymptotic per vertex mutual information between \revsag{$(\bm \beta_0,\bm \sigma_0)$} and $(\bm y,\bm G)$. For Figure (a), we plot the mutual information (MI) as a function of $\Delta$. Note that as the noise strength $\Delta$ increases, the limiting MI decreases. The case $\lambda=0$ corresponds to the setting with no graph side information. We see that the graph information increases the limiting MI in the model. In Figure (b), we plot the limiting MI as a function of $\lambda$. We again see that as the signal strength $\lambda$ increases, the asymptotic MI increases in the model.}
\label{fig:inf}
\end{figure}

We observe in Figure \ref{fig:inf_del} that if $\Delta$ is large, that is, we have significant noise in $\bm y$, the observed graph $\bm G$ adds significant information to our estimation procedure which increases with increasing $\lambda$. But in Figure \ref{fig:inf_lam}, it is clear that the advantage in observing network information decreases with an increase in $\Delta$.

\subsection{Comparison of reconstruction error between AMP-based estimator and estimator based on Graph Constrained Regression}
\label{sec:comparison} 
\subsubsection{Gaussian Design Matrix}
\label{comp_method}
In this subsection, we compare our AMP-based estimation of $\bm \beta$ with an estimator based on Laplacian penalization \cite{li_and_li}. 
We set $n=p=3000$. 
For $1\leq i \leq p$, let $\sigma_{0i}$ be  independent samples from $\mathrm{Bernoulli}(0.7)$. Given $\sigma_{0i}=0$, set \revsag{$\beta_{0i}=0$} with probability $1$. If $\sigma_{0i}=1$, we generate \revsag{$\beta_{0i}$} from the uniform distribution on $\{\pm 1\}$. The graph $\bm G$ is generated according to \eqref{eq:graph}. The entries of the feature matrix $\bm \Phi$ are generated i.i.d from $\mathsf{N}(0,p^{-0.5})$, and the observation vector $\bm y$ is generated according to \eqref{eq:lin_model_1}. We fix $\lambda$ to be in the set $\{3,5\}$ and $b_p=0.7$. Using the relation \eqref{eq:def_snr}, the parameter $a_p$ is set to be equal to
\[
 \revsag{a_p = b_p+\sqrt{\lambda b_p\left(1-\frac{b_p}{p}\right)}.}
\]
We vary $\Delta$ in an equispaced grid with $20$ points in $[0.2,4]$. 

Next, we compute the estimates $\widehat{\bm \beta}_{\mathsf{Lap}}$ and $\widehat{\bm \beta}_{\mathsf{AMP}}$ where $\widehat{\bm \beta}_{\mathsf{Lap}}$ is the estimator computed using the method described in \cite{li_and_li} and $\widehat{\bm \beta}_{\mathsf{AMP}}$ is computed by the AMP iterations described by \eqref{eq:amp_iterates} and \eqref{eq:amp_iterates_1}. We run the AMP algorithm for 25 iterations to generate our estimates. For each combination of $(\lambda,\Delta)$ we repeat the experiment $20$ times independently and for each iteration we compute the empirical reconstruction errors
\[
\revsag{\mathcal E_{\mathsf{Lap}}=\frac{1}{p}\|\bm \Phi(\widehat{\bm \beta}_{\mathsf{Lap}}-\bm \beta_0)\|^2_2},
\]
and
\[
\revsag{\mathcal E_{\mathsf{AMP}}=\frac{1}{p}\|\bm \Phi(\widehat{\bm \beta}_{\mathsf{AMP}}-\bm \beta_0)\|^2_2}.
\]
We approximate the mean reconstruction errors by the sample average of the estimates across the independent replications. 
The plots of the estimated reconstruction errors are shown in Figure \ref{fig:amp}.

\begin{figure}[tb]
\begin{subfigure}{.5\textwidth}
  \centering
  \includegraphics[width=\linewidth]{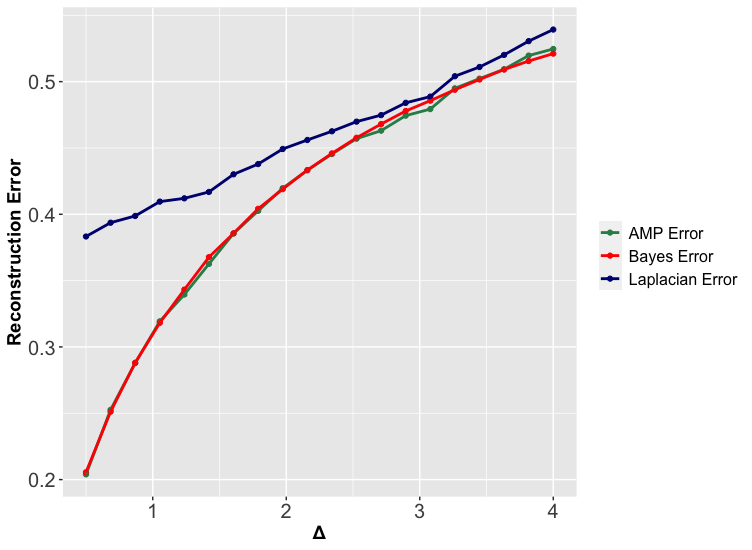}
  \caption{$\lambda=3$}
  \label{fig:amp_2}
\end{subfigure}%
\begin{subfigure}{.5\textwidth}
  \centering
  \includegraphics[width=\linewidth]{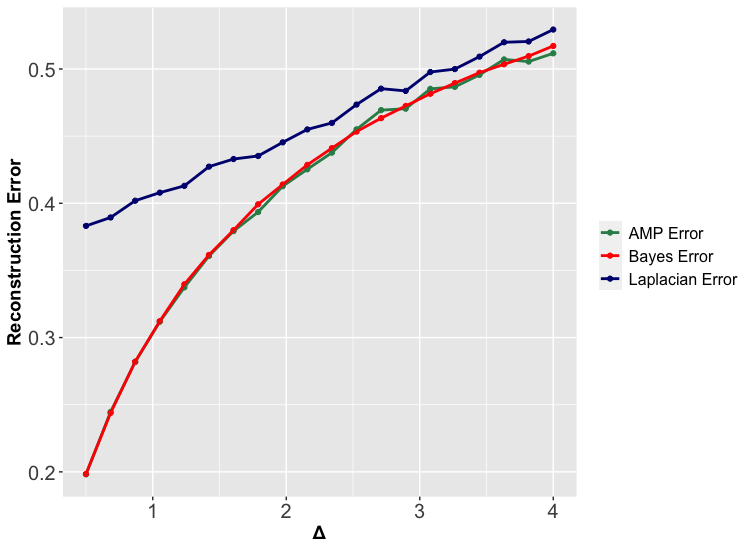}
  \caption{$\lambda=5$}
  \label{fig:amp_3}
\end{subfigure}%
\caption{Plot of the average reconstruction error of $\widehat{\bm \beta}_{\mathsf{Lap}}$ and $\widehat{\bm \beta}_{\mathsf{AMP}}$ for Gaussian design matrix}
\label{fig:amp}
\end{figure}

We observe that the AMP-based estimator $\widehat{\bm \beta}_{\mathsf{AMP}}$ performs consistently better than $\widehat{\bm \beta}_{\mathsf{Lap}}$ for both values of $\lambda$. In fact, the performance of $\widehat{\bm \beta}_{\mathsf{AMP}}$ in terms of the reconstruction error is approximately equal to the Bayes error in estimating $\bm \beta_0$ with the specified priors. The observed gaps are due to finite sample effects. The simulations demonstrate that if the prior is known, it is much more efficient to incorporate the prior information through AMP-based algorithms. 

\subsubsection{Robustness to design distribution}
Our theoretical results are derived under an iid Gaussian assumption on the entries of the design matrix. AMP style algorithms are known to be quite robust to the design distribution \cite{bayati2015universality,10.1214/21-EJP604,dudeja2022universality,wang_zhong_fan}, and we expect the algorithms introduced in this paper to enjoy similar universality properties. 

To this end, we provide initial numerical evidence in support of the universality properties of this algorithm. We work under the same setup reported in the previous subsection, but construct the design $\bm\Phi$ using iid centered and normalized $\mathrm{Bernoulli}(0.3)$ entries. We plot the reconstruction error in Figure \ref{fig:amp_u_3}. We note that the main takeaways remain the same---the AMP algorithm still outperforms the Laplacian penalization-based algorithm. We believe that using the ideas introduced in \cite{dudeja2022universality,wang_zhong_fan}, it should not be too difficult to establish the universality property for this class of AMP algorithms. 
%
%
%
%
%

\begin{figure}[tb]
\begin{subfigure}{.5\textwidth}
  \centering
  \includegraphics[width=\linewidth]{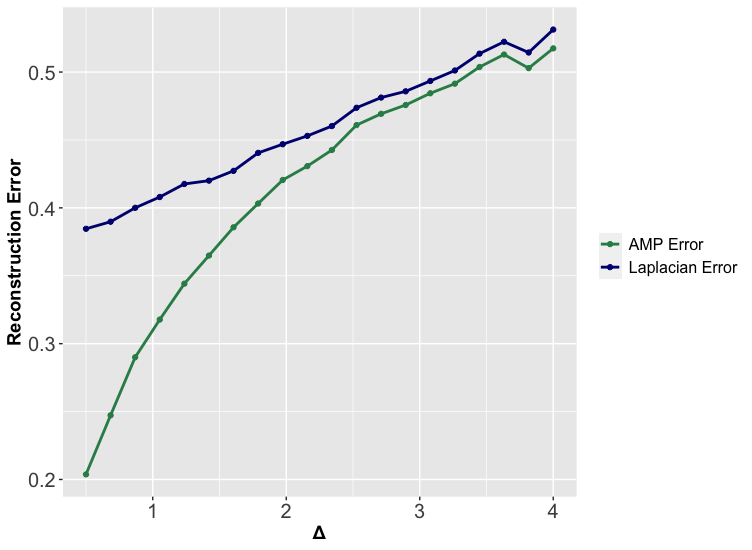}
  \caption{$\lambda=3$}
  \label{fig:amp_u_2}
\end{subfigure}%
\begin{subfigure}{.5\textwidth}
  \centering
  \includegraphics[width=\linewidth]{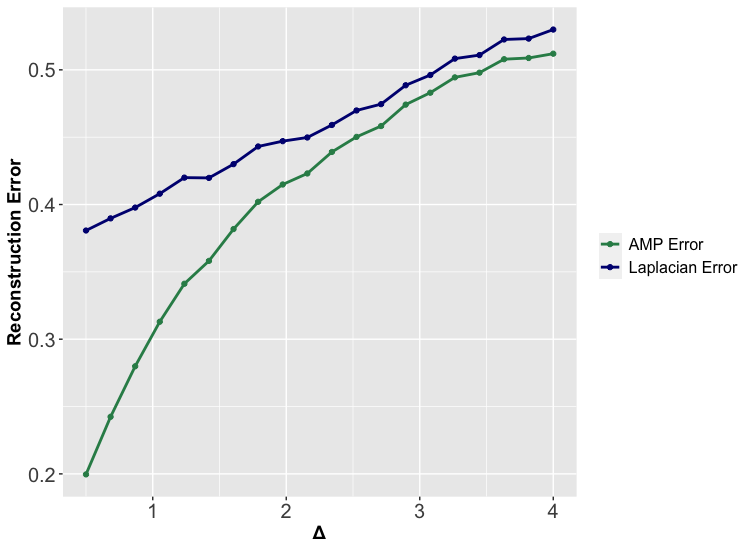}
  \caption{$\lambda=5$}
  \label{fig:amp_u_3}
\end{subfigure}%
\caption{Plot of the average reconstruction error of $\widehat{\bm \beta}_{\mathsf{Lap}}$ and $\widehat{\bm \beta}_{\mathsf{AMP}}$ for Bernoulli Design Matrix}
\label{fig:amp_u}
\end{figure}

\subsection{Variable discovery: AMP vs. Model-X Knockoff} 

In this section, we explore the variable discovery properties of the AMP based algorithm introduced in Section \ref{variable_discovery}. 

Knockoffs \cite{barber2015controlling,candes2018panning} have emerged as the canonical choice for variable discovery in supervised learning problems. This methodology is attractive due to the finite sample guarantees on the FDR, along with high power in most practical settings. 

In this section, we compare the performance of the AMP based algorithm to the statistical performance of the Model-X framework. Note that the Model-X framework ignores the graph information, and thus this is not a fair comparison in principle. Here, we invoke the Model-X methodology as the canonical algorithm for variable discovery in the absence of graph-side information---this helps us explore the practical gains obtained from the graph-side information.

\subsubsection{Power comparison}
In this section, we compare the performance of the variable discovery mechanism stated in Section \ref{variable_discovery} with Model-X knockoff in terms of the True Discovery Ratio given by
\[
\mathrm{TDR}(\alpha,t;n):=\mathbb{E}\left[\frac{|\hat{S}(\alpha;t) \cap \{i:\sigma_{0i}\neq0\}|}{1 \vee |\hat{S}(\alpha;t)|}\right].
\]
Here $\hat{S}(\alpha;t)$ is as defined in Section \ref{variable_discovery}. We take the generative model with the same $n, p$ and the prior distribution for \revsag{$\bm \beta_0$} given $\bm \sigma_0$ as Section \ref{comp_method}, but now we generate $\sigma_{0i}$ for $i=1,\cdots,p$ independently from $\mathrm{Bernoulli}(0.07)$ to induce sparsity in the model. We compare the performances of our methods versus Model-X knockoff for $\lambda \in \{5, 10\}$ and 5 equispaced $\Delta$ lying between $0.5$ and $4$. We tabulate the Monte Carlo estimates of the TDR over  20 independent runs of the experiment in Table \ref{tab:amp_v_knockoff}. We observe that irrespective of the values of $\Delta$ and $\lambda$, our variable selection procedure performs uniformly better than Model-X knockoff that ignores the graph side information.
\begin{table}[h!]
\centering
\begin{tabular}{|c|c|c|c|c|c|c|} 
 \hline
 $\Delta$ & \multicolumn{2}{c|}{$\lambda = 5$} & \multicolumn{2}{c|}{$\lambda = 10$} \\
 \hline
 {}   & AMP   & Knockoff  & AMP   & Knockoff  \\ 
 \hline
 0.5 & 0.805 & 0.641  & 0.856 & 0.787 \\
 \hline
 1.05 & 0.651 & 0.433 & 0.775 & 0.658\\
\hline
1.79 & 0.605 & 0.456 & 0.542 & 0.475\\
\hline
2.52 & 0.438 & 0.133 & 0.470 & 0.404\\
\hline
3.26 & 0.338 & 0.231 & 0.217 & 0.199\\
\hline
4.00 & 0.17 & 0.09 & 0.3 & 0.2\\
\hline
\end{tabular}
\caption{Performance Comparison between the TDR of AMP-based variable selection and Model-X Knockoff-based variable selection}
\label{tab:amp_v_knockoff}
\end{table}

\section{Discussion}
\label{sec:discussion} 
In this paper, we formulated the problem of supervised learning with graph-side information in terms of a simple generative model and introduced an AMP-based algorithm to combine the information from the two sources. We also derived the asymptotic mutual information in this model and established that in many settings, this algorithm is, in fact, Bayes optimal. Finally, our numerical experiments establish the improvements obtained by this aggregation scheme. 

In this section, we discuss some current limitations of these results, and opportunities to go beyond these barriers. In turn, this automatically suggests some natural directions for future inquiry. 

\begin{itemize}
    \item[(i)] Generalization to account for more than 1 planted community---In this paper, we study a simple setting with one planted community. Of course, in the applications discussed in the introduction, it is probably more natural to assume the presence of multiple planted communities, with different connection probabilities for variables in distinct communities. We expect that the technical framework introduced in this paper can be extended to this setting in a straightforward manner, and present the 1 community case for ease of exposition. 
    
    \item[(ii)] Incorporating correlation among the features---We assume independent Gaussian features in our regression model. In many practical settings, it might be more realistic to have correlated features. For example, one could study a setting where the rows $\phi_\mu \in \mathbb{R}^p$ are i.i.d. samples from $N(0,\Sigma)$. It should be possible to design AMP-based algorithms for estimation and variable discovery even in this setting, using the ideas in \cite{huang2022lasso,loureiro2021learning,clarte2023double}. However, it is particularly challenging to characterize the Bayes optimal performance in the correlated setting. Specifically, evaluation of the limiting mutual information will require new ideas. We believe this will be a very interesting direction for follow up investigations. 
    

    
    \item[(iii)] The need for empirical Bayes approaches---The AMP algorithm introduced in this paper explicitly uses knowledge of the problem parameters $\rho$, $\Delta$, $a_p$, $b_p$ and the Markov kernel $P(\cdot|\cdot)$. In our discussions, we assume that these problem parameters are known. In practice, some or all of these parameters might be unknown. To make the algorithms practicable in this case, the unknown parameters need to be estimated from the given data. We note that the estimation of the graph connectivity parameters $a_p$, $b_p$ has been explored in the previous literature \cite{mossel2012stochastic}. In a similar vein, the estimation of the noise variance $\Delta$ and the underlying sparsity $\rho$ have been explored in the statistical literature (see e.g. \cite{janson2017eigenprism} and references therein). In this context, the estimation of the kernel $P(\cdot|\cdot)$ is expected to be the most challenging. One natural idea to estimate the conditional distribution would be to use empirical Bayesian methods \cite{robbins1992empirical}. We refer the interested reader to a recent application of this idea to the PCA problem in \cite{zhong2022empirical}. While this would be extremely interesting to explore, we feel that this is substantially beyond the scope of this paper, and we defer this to follow-up investigations. We note here that even if the Markov kernel $P(\cdot|\cdot)$ is unknown, the AMP algorithm can be implemented with any arbitrary kernel $Q(\cdot|\cdot)$, and the performance of the algorithm can be tracked using state evolution. Of course, if $P$ and $Q$ are quite different, the AMP performance is expected to be sub-optimal compared to the Bayes optimal performance. 
    
    \item[(iv)] Statistical/Computational gaps in this problem---In Theorem \ref{thm:opt_amp}, we noted that the AMP algorithm attains Bayes optimal performance if $(\mu^*, \xi^*) = (\bar{\mu}, \bar{\xi})$. Of course, this equality could be violated for certain parameters $(a_p,b_p, \rho, \Delta, P)$. In this case, we conjecture the existence of a statistical-computational gap in this problem, and that the performance of the AMP algorithm represents the best statistical performance among computationally feasible algorithms. Statistical/Computational gaps have been conjectured in similar problems in the recent literature (see e.g. \cite{brennan2018reducibility,abbe2017community,kunisky2022notes,celentano2022fundamental,decelle2011asymptotic} and references therein for a very incomplete list), and there has been significant recent progress in favor of this conjecture by analyzing specific sub-classes of algorithms (e.g. based on convex penalized estimators \cite{celentano2022fundamental}, first-order methods \cite{celentano2020estimation}, low degree algorithms \cite{montanari2022equivalence,schramm2022computational}, query lower bounds \cite{diakonikolas2017statistical} etc.). We believe that it should be possible to adapt the existing techniques to establish statistical/computational gaps in this problem. We refrain from exploring this direction in this paper to keep our discussion focused.

    \item[(v)] Recent applications of AMP to genomics---Our work is motivated by problems arising in biology, genomics, neuroscience etc. We note that AMP ideas have been used recently to integrate EHR data with genomic information \cite{mondelli}. We believe that these works collectively demonstrate the usefulness of AMP-based algorithms for big-data applications, and hope that this will stimulate further investigations into the use of these ideas for other scientific applications.   
\end{itemize}

\paragraph{Acknowledgments:} SS thanks Jishnu Das for introducing him to the HotNet2 procedure, which motivated this work. SS gratefully acknowledges support from NSF (DMS CAREER 2239234), ONR (N00014-23-1-2489), AFOSR (FA9950-23-1-0429) and a Harvard FAS Dean’s competitive
fund award. SN thanks Zongming Ma for many helpful discussions. We thank Lenka Zdeborov\'{a} and Jean Barbier for pointing us to several relevant references, and helpful discussions. Finally, the authors thank the anonymous referees for their insightful comments which improved the quality and the presentation of the paper.

\bibliographystyle{apalike}
\bibliography{Regression_with_Graph_Side_Information.bib}
\newpage
\appendix

\section{Proof of Lemma \ref{lem:contractive}}
\begin{proof}[Proof of Lemma \ref{lem:contractive}]
By definition, the functions $\mathsf{mmse}_1$ and $\mathsf{mmse}_2$ are continuous bounded functions of their arguments with the additional property that for $a \le b$ and $x \ge y$, we have
\[
\mathsf{mmse}_1(a,x) \ge \mathsf{mmse}_1(b,y), \quad \mbox{and} \quad \mathsf{mmse}_2(a,x) \ge \mathsf{mmse}_2(b,y).
\]
Let us also observe that $\frac{1}{\kappa}\mathbb E[B^2] = \xi_0 \ge \xi_1$ and $0 = \mu_0 \le \mu_1$. Hence, we can inductively conclude that $\{\xi_t\}_{t \ge 0}$ and $\{\mu_t\}_{t \ge 0}$ are decreasing and increasing in $t$, respectively. Therefore, using the boundedness of the functions $\mathsf{mmse}_1$ and $\mathsf{mmse}_2$ we can conclude that there exists $\mu^*, \xi^* \ge 0$ satisfying
\begin{align}
\label{eq:fixed points}
\mu^*=\lambda(\rho-\mathsf{mmse}_1(\mu^*,\xi^*)),\,\,\,\,
\xi^* = \frac{1}{\Delta}\mathsf{mmse}_2(\mu^*,\xi^*),
\end{align}
such that $\mu_t \rightarrow \mu^*$ and $\xi_t \rightarrow \xi^*$ as $t \rightarrow \infty$.
\end{proof}
\section{Proof of Theorem \ref{thm:mut_inf_graph}}
\label{mut_inf}

We employ the \emph{adaptive interpolation} approach of \citet{Adaptive_Interpolation} (see also \cite{barbier2019optimal,barbier2020mutual} and references therein) to characterize the mutual information in this setting.

We first connect the Regression plus Graph model to an equivalent Regression plus Gaussian Orthogonal Ensemble model. The model is described as follows. Given $\sigma_{01},\cdots,\sigma_{0p}$ generated from $\mbox{Ber}(\rho)$ distribution, we still observe the pair $(\bm y, \bm \Phi)$ \revsag{satisfying} \eqref{eq:lin_model_1}, but the random network $\bm G$ replaced by a Gaussian random matrix described as follows:
\[
\widetilde{\bm A}=\sqrt{\frac{\lambda}{p}}\bm \sigma_0 \bm \sigma^\top_0 + \bm Z,
\]         
where $Z_{ij}=Z_{ji}\sim \mathsf{N}(0,1)$ if $i \neq j$ and $Z_{ii} \sim N(0,2)$ \revsag{for all $1 \le i, j \le p$}. From the definition of mutual information, we get the following.
\begin{align}
\label{eq:lim_mut_inform}
\lim\limits_{p \rightarrow \infty}\frac{1}{p}I(\bm \beta_0,\bm \sigma_0;\widetilde{\bm A},\bm y)&=\lim\limits_{p \rightarrow \infty}\frac{1}{p}\mathbb{E}_{(\bm \beta_0, \bm \sigma_0, \bm \Phi,\bm Z, \bm \varepsilon)}\left[\log\frac{P(\widetilde{\bm A},\bm y\,|\,\bm \beta_0, \bm \sigma_0)}{P(\widetilde{\bm A},\bm y)}\right]\\
&=\mathbb{E}_{(\Sigma,\revsag{B})}[\log P(\revsag{B}\,|\,\Sigma)Q(\Sigma)]+\frac{\lambda}{4}(\mathbb E[\Sigma^2])^2-\lim\limits_{p \rightarrow \infty}\frac{1}{p}\mathbb{E}[\log \mathcal{Z}(\revsagr{\bm \sigma_0,\bm \beta_0})],
\end{align}
where, $Q(\Sigma)=\rho^\Sigma(1-\rho)^{1-\Sigma}\mathbb I\{\Sigma \in \{0,1\}\}$, $B|\Sigma \sim P(\,\cdot\,|\,\Sigma)$ and
\[
\mathcal{Z}(\bm \sigma_0,\bm \beta_0) = \int\limits_{\bm x,\bm \beta}\left\{\prod\limits_{i=1}^{p}dx_i\,d\beta_i\,Q(x_i)P(\beta_i|x_i)\right\}\exp\left(-\mathcal{H}_n(\bm x,\bm \beta; \bm \sigma_0,\bm \beta_0,\bm \Phi, \bm \varepsilon, \bm Z)\right).
\]
The term $\mathcal H_n$ referred to as the \emph{Hamiltonian} is statistical physics literature is given by
\begin{align}
\mathcal{H}_n(\bm x,\bm \beta; \bm \sigma_0,\bm \beta_0,\bm \Phi, \bm \varepsilon, \bm Z) & =\lambda \sum\limits_{i \le j=1}^p\left\{\frac{x^2_ix^2_j}{2p}-\frac{x_ix_j\sigma_{0i}\sigma_{0j}}{p}-\frac{x_ix_jZ_{ij}}{\sqrt{\lambda p}}\right\}\\
& + \frac{1}{\Delta}\sum\limits_{\mu=1}^{n}\left\{\frac{1}{2}[\bm \Phi(\bm \beta-\bm \beta_0)]^2_\mu-[\bm \Phi(\bm \beta-\bm \beta_0)]_\mu\varepsilon_\mu\sqrt{\Delta}\right\}.
\end{align}
Therefore, in order to understand the asymptotics of the per-vertex mutual information, we focus on studying the limiting behavior of
\[
f_p:=-\frac{1}{p}\revsagr{\mathbb{E}_{(\bm \beta_0, \bm \sigma_0, \bm \Phi,\bm Z, \bm \varepsilon)}[\log \mathcal{Z}(\bm \sigma_0,\bm \beta_0)]},
\]
referred to as \emph{Bethe Free Energy} in statistical physics. In that direction, let us define the \emph{replica symmetric potential} function $f_{\mathrm{RS}}$ as follows: 
\[
f_{\rs}(\lambda,q;E,\Delta):=\frac{\lambda q^2}{4}+\psi(E;\Delta)+\wt{f}\left(\frac{1}{\lambda q},\Sigs(E,\Delta)^{2}\right),
\]
where 
\[
\revsagr{\Sigs(E,\Delta)^{-2}}=\frac{\kappa}{\Delta+E}; \quad \quad \psi(E,\Delta)=\frac{\kappa}{2}\left[\log\left(1+\frac{E}{\Delta}\right)-\frac{E}{E+\Delta}\right],
\] 
and
\begin{align}
\wt{f}(\widetilde{\sigma}^2_1,\widetilde{\sigma}^2_2)&:=-\mathbb{E}_\Theta\Bigg[\revsagr{\log\int_{x,\beta} P(\beta|x)Q(x)\exp\Bigg(\hspace{-0.04in}-\frac{1}{\widetilde{\sigma}^2_1}\left(\frac{x^2}{2}-\Sigma x-\widetilde{\sigma}_1x \widebar{Z}\right)}\\
&\hskip 1in -\frac{1}{\widetilde{\sigma}^2_2}\left(\frac{(\beta-B)^2}{2}-\widetilde{\sigma}_2(\beta-B)\widebar{\varepsilon}\right)\hspace{-0.04in}\Bigg)\Bigg],
\end{align}
for \revsagr{$\Theta = (B,\Sigma,\widebar{Z},\widebar{\varepsilon})$}, $B \sim P(\,\cdot\,|\,\Sigma)$, $\Sigma \sim \mathrm{Bernoulli}(\rho)$, and $\widebar{Z},\widebar{\varepsilon}\overset{i.i.d}{\sim} \mathsf{N}(0,1)$. Then, we have the following theorem characterizing the asymptotic behavior of $f_p$.
\begin{thm}
    \label{thm:lim_free_ener}
    If \revsagr{$P(\,\cdot\,|\,\Sigma\,)$ has discrete and bounded support for $\Sigma \in \{0,1\}$}, then we have 
    \[
    \revsagr{\lim_{p \rightarrow \infty}f_p = \inf\bigg\{f_{\rs}(\lambda,q;E,\Delta):q \ge 0, E \ge 0\bigg\}.}
    \]
\end{thm}
\subsection{Proof of Theorem \ref{thm:lim_free_ener}}
In order to prove Theorem \ref{thm:lim_free_ener}, let us consider a sequence of functions interpolating between the true Hamiltonian given by
\begin{align}
\label{eq:h_on_top}
h(\bm x,\bm \beta;\bm \sigma_0,\bm \beta_0,\bm \Phi,\bm \varepsilon,\bm Z,\theta^2_1,\theta^2_2)&:=\frac{1}{\theta^2_1}\sum\limits_{\substack{i \le j\\i,j=1}}^{p}\left\{\frac{x^2_i x^2_j}{2p}-\frac{x_ix_j\sigma_{0i}\sigma_{0j}}{p}-\theta_1\frac{x_ix_jZ_{ij}}{\sqrt{p}}\right\}\\
&+\frac{1}{\theta^2_2}\sum\limits_{\mu=1}^{n}\left\{\frac{[\bm \Phi(\bm \beta-\bm \beta_0)]^2_\mu}{2}-\theta_2[\bm \Phi(\bm \beta-\bm \beta_0)]_\mu\varepsilon_\mu\right\},
\end{align}
and the \emph{Mean Field Hamiltonian} given by
\begin{align}
\label{eq:h_mf_on_top}
h_{mf}(\bm x,\bm \beta;\bm \sigma_0,\bm \beta_0,\bm \Phi,\widebar{\bm \varepsilon},\widebar{\bm Z}, \theta^2_1,\theta^2_2)&:=\frac{1}{\theta^2_1}\sum\limits_{i=1}^{p}\left\{\frac{x^2_i}{2}-x_i\sigma_{0i}-\theta_1x_i\widebar{Z_{i}}\right\}\\
&+\frac{1}{\theta^2_2}\sum\limits_{j=1}^{p}\left\{\frac{(\beta_j-\beta_{0j})^2}{2}-\theta_2(\beta_j-\beta_{0j})\widebar{\varepsilon}_j\right\}.
\end{align}
Next, let us introduce the sequence of functions $\{\mathcal H_{k,t;\bm \eta}: t \in [0,1], k \in [K]\}$ interpolating between the true and the mean-field Hamiltonians characterized by a sequence of parameters $\{(q_k,E_k): k \in [K]\}$ and \revsagr{$\bm \eta =(\eta_1,\eta_2)$ with $\eta_1,\eta_2>0$} as follows:
\begin{align}
\mathcal{H}_{k,t;\bm \eta}\left(\bm x,\bm \beta;\bm \sigma_0,\bm \beta_0,\Theta\right)&:=\sum\limits_{\ell=k+1}^{K}h(\bm x,\bm \beta;\bm \sigma_0,\bm \beta_0,\bm \Phi,\bm \varepsilon_\ell,\bm Z_\ell,K/\lambda,K\Delta)\\
&\hspace{0.2in}+\sum\limits_{\ell=1}^{k-1}h_{mf}(\bm x,\bm \beta;\bm \sigma_0,\bm \beta_0,\bm \Phi,\widebar{\bm \varepsilon}_\ell,\widebar{\bm Z}_\ell,K/\lambda q_\ell,K\revsagr{\Psi^2_\ell})\\
&\hspace{0.2in}+h(\bm x,\bm \beta;\bm \sigma_0,\bm \beta_0,\bm \Phi,\bm \varepsilon_k,\bm Z_k,K/\lambda(1-t),K/\gamma_k(t))\\
&\hspace{0.2in}+h_{mf}(\bm x,\bm \beta;\bm \sigma_0,\bm \beta_0,\bm \Phi,\widebar{\bm \varepsilon}_k,\widebar{\bm Z}_k,K/\lambda tq_k,K/\lambda_k(t))\\
&\hspace{0.2in}+\eta_1\sum\limits_{i=1}^{p}\left\{\frac{x^2_i}{2}-\sigma_{0i}x_i-\frac{x_i\widehat{Z}_i}{\sqrt{\eta_1}}\right\}+\eta_2\sum\limits_{i=1}^{p}\left\{\frac{(\beta_i-\beta_{0i})^2}{2}-\frac{(\beta_i-\beta_{0i})\widehat{\widehat{Z}}_i}{\sqrt{\eta_2}}\right\}.
\end{align}
Here 
\begin{align}
\label{eq:def_big_bold_theta}
\Theta:=(\bm \Phi,\{\bm \varepsilon_\ell\}_{\ell=1}^{K},\{\bm Z_\ell\}_{\ell=1}^{K},\{\widebar{\bm \varepsilon}_\ell\}_{\ell=1}^{K},\{\widebar{\bm Z}_\ell\}_{\ell=1}^{K},\{\widehat{Z}_i\}_{i=1}^{p},\{\widehat{\widehat{Z}}_i\}_{i=1}^{p},\{q_k\}_{k=1}^{K},\{E_k\}_{k=1}^{K}),
\end{align}
with $\{\bm \varepsilon_\ell\}_{\ell=1}^{K} \overset{i.i.d}{\sim} \mathsf{N}_n(\bm 0, \bm I_n)$, $\{\widebar{\bm \varepsilon}_\ell\}_{\ell=1}^{K}\overset{i.i.d}{\sim} \mathsf{N}_p(\bm 0, \bm I_p)$, $\{\bm Z_\ell\}_{\ell=1}^{K} \overset{i.i.d}{\sim}\sqrt{p}\,\mathrm{GOE}(p)$, $\{\widebar{\bm Z}_\ell\}_{\ell=1}^{K}\overset{i.i.d}{\sim}\mathsf{N}_p(\bm 0, \bm I_p)$, $\{\widehat{Z}_\ell\}_{\ell=1}^{K} \overset{i.i.d}{\sim}\mathsf{N}(0,1)$ and $\{\widehat{\widehat{Z}}_\ell\}_{\ell=1}^{K}\overset{i.i.d}{\sim}\mathsf{N}(0,1)$. A matrix $\bm Z \in \R^{p \times p}$ is said to follow $\mathrm{GOE}(p)$ if for all $1 \le i < j \le p$, $\bm Z_{ij} \sim \mathsf{N}(0,p^{-1})$, $\bm Z_{ii} \sim \mathsf{N}(0,2p^{-1})$ and \revsagr{$\bm Z_{ij}=\bm Z_{ji}$ for all $i \neq j$}. Furthermore, the parameters 
\[
\revsagr{\Psi^{-2}_\ell:=\frac{\kappa}{\Delta+E_\ell},}
\]
and the functions $\gamma_k,\lambda_k:[0,1]\rightarrow \mathbb{R}$ for $k \in [K]$ are defined by the differential equation
\begin{equation}
\label{eq:param_lin_mod}
\frac{d\lambda_k(t)}{dt}=-\frac{d\gamma_k(t)}{dt}\frac{\kappa}{(1+\gamma_k(t)E_k)^2}.
\end{equation}
with boundary conditions
\[
\gamma_k(0)=\Delta^{-1}; \quad \lambda_k(0)=0; \quad \gamma_k(1)=0, \quad \mbox{and} \quad \lambda_k(1)=\revsagr{\Psi^{-2}_k}.
\]
\revsagr{Additionally, we require $\frac{d\gamma_k(t)}{dt} \le 0$ for all $t$.} Finally, we define the interpolating \emph{free energy} as follows:
\begin{equation}
\label{eq:free_energy_perturbed}
f_{k,t;\bm \eta}=-\frac{1}{p}\mathbb{E}_\Theta\left[\log\int\left\{\prod\limits_{i=1}^{p}P(\beta_i|x_i)Q(x_i)d\beta_idx_i\right\}\exp\left(-\mathcal{H}_{k,t;\bm \eta}(\bm x, \bm \beta;\bm \sigma_0,\bm \beta_0,\Theta)\right)\right].
\end{equation}
Next, let us consider the boundary cases \revsagr{$\mathcal{H}_{1,0;\bm 0}$ and $\mathcal{H}_{K,1;\bm 0}$.}

\begin{align}
\revsagr{\mathcal{H}_{1,0;\bm 0}}(\bm x,\bm \beta;\bm \sigma_0,\bm \beta_0,\Theta)&=\sum\limits_{\ell =1}^{K}h(\bm x,\bm \beta;\bm \sigma_0,\bm \beta_0,\bm \Phi,\bm \varepsilon_\ell,\bm Z_\ell,K/\lambda,K\Delta)\\
&=\frac{\lambda}{K}\sum\limits_{\ell=1}^{K}\left\{\sum\limits_{\substack{i\le j\\i,j=1}}^{p}\left[\frac{x^2_ix^2_j}{2p}-\frac{\sigma_{0i}\sigma_{0j}x_ix_j}{p}\right]\right\}-\sum\limits_{\substack{i\le j\\i,j=1}}^px_ix_j\sqrt{\frac{\lambda}{K}}\frac{1}{\sqrt{p}}\sum\limits_{\ell=1}^{K}(\bm Z_{\ell})_{ij}\\
& \hspace{0.3in}+\frac{1}{\Delta}\sum\limits_{\mu=1}^{n}\frac{1}{2}[\bm \Phi(\bm \beta-\bm \beta_0)]^2_\mu-\frac{1}{\Delta}\sum\limits_{\mu=1}^{n}[\bm \Phi(\bm \beta-\bm \beta_0)]_\mu\left(\frac{\sqrt{\Delta}}{\sqrt{K}}\sum\limits_{\ell=1}^{K}(\bm \varepsilon_{\ell})_\mu\right)\\
&\overset{d}{=}\lambda\left\{\sum\limits_{\substack{i\le j\\i,j=1}}^{p}\left[\frac{x^2_ix^2_j}{2p}-\frac{\sigma_{0i}\sigma_{0j}x_ix_j}{p}\right]\right\}-\sqrt{\frac{\lambda}{p}}\sum\limits_{\substack{i\le j\\i,j=1}}^{p}x_ix_j\bm Z_{ij} \\
& \hspace{0.3in}+\frac{1}{\Delta}\sum\limits_{\mu=1}^{n}\frac{1}{2}[\bm \Phi(\bm \beta-\bm \beta_0)]^2_\mu-\frac{1}{\Delta}\sum\limits_{\mu=1}^{n}[\bm \Phi(\bm \beta-\bm \beta_0)]_\mu\sqrt{\Delta}\varepsilon_\mu\\
&\hskip 12em \mbox{[where $\bm Z \sim \sqrt{p}\;\mathrm{GOE}(p)$ and  $\varepsilon_\mu \sim \mathsf{N}(0,1)$]}.
\end{align}
Furthermore, we also have
\begin{align}
\revsagr{\mathcal{H}_{K,1;\bm 0}}(\bm x,\bm \beta;\bm \sigma_0,\bm \beta_0,\Theta)&=\sum_{\ell=1}^{K}h_{mf}(\bm x,\bm \beta;\bm \sigma_0,\bm \beta_0,\bm \Phi,\widebar{\bm \varepsilon}_\ell,\widebar{\bm Z}_\ell,K/\lambda q_\ell,K\Psi^2_\ell)\\
& = \lambda \widebar{q}\left\{\sum\limits_{i=1}^{p}\left[\frac{x^2_i}{2}-\sigma_{0i}x_i-\sqrt{\frac{1}{\lambda\widebar{q}}}x_i\sum\limits_{\ell=1}^{K}(\widebar{\bm Z}_{\ell})_i\sqrt{\frac{q_\ell}{K\widebar{q}}}\right]\right\}\\
& \quad + \frac{1}{\widebar{\Psi}^2}\left\{\sum\limits_{i=1}^{p}\left[\frac{(\beta_i-\beta_{0i})^2}{2}-\widebar{\Psi}(\beta_i-\beta_{0i})\sum\limits_{\ell=1}^{K}(\widebar{\bm \varepsilon}_{\ell})_i\frac{\widebar\Psi}{\Psi_\ell}\frac{1}{\sqrt{K}}\right]\right\}\\
&=h_{mf}(\bm x,\bm \beta;\bm \sigma_0,\bm \beta_0,\bm \Phi,\bm \varepsilon,(\lambda\bar{q})^{-1},\widebar{\Psi}^2),
\end{align}
where 
\begin{equation}
\label{eq:effec_param}
\widebar{\Psi}^{-2}=\frac{1}{K}\sum_{\ell=1}^{K}\Psi^{-2}_\ell; \quad \quad \mbox{and} \quad \quad \widebar{q}=\frac{1}{K}\sum\limits_{\ell=1}^{K}q_\ell.
\end{equation}
Here it is worthwhile to observe that $\revsagr{f_{1,0,\bm 0}}=f_p$ and $\revsagr{f_{K,1,\bm 0}}=\wt f((\lambda\bar{q})^{-1},\widebar{\Psi}^2)$. Also, for all $k \in [K-1]$, we have $f_{k+1,0;\bm \eta}=f_{k,1;\bm \eta}$.

Let us consider the Gibbs measure with respect to the interpolating Hamiltonian $\mathcal H_{k,t;\bm \eta}$.
\begin{align}
\label{eq:gibbs_mesr}
P_{k,t;\bm \eta}(\bm x,\bm \beta|\theta):=\frac{\prod_{i=1}^{p}Q(x_i)P(\beta_i|x_i)\exp(-\revsagr{\mathcal{H}_{k,t;\bm \eta}}(\bm x,\bm \beta;\bm \sigma_0,\bm \beta_0,\theta))}{\int\{\prod_{i=1}^{p}Q(x_i)P(\beta_i|x_i)dx_id\beta_i\}\exp(-\revsagr{\mathcal{H}_{k,t;\bm \eta}}(\bm x,\bm \beta;\bm \sigma_0,\bm \beta_0,\theta)}.
\end{align}
Furthermore, the expectation of a function $A(\bm x,\bm \beta)$ with respect to $P_{k,t;\bm \eta}(\bm x,\bm \beta|\theta)$ will be denoted by $\langle A(\bm x, \bm \beta)\rangle_{\mathcal H_{k,t;\bm \eta}}$. We shall compute the limit $\lim_{p \rightarrow \infty}f_p$ by characterizing the change in free energy along the interpolation path specified by the functions $\{f_{k,t;\bm \eta}: t \in [0,1], k \in [K_p]\}$ for a sequence of numbers $K_p \rightarrow +\infty$ and appropriately chosen parameters $\{E_k,q_k: 1 \le k \le K_p\}$. In that direction, let us consider $\frac{df_{k,t;\bm \eta}}{dt}$.
From the definition
\[
\frac{df_{k,t;\bm \eta}}{dt}=\frac{1}{p}\mathbb{E}_\Theta\left[\left\langle\frac{d\mathcal{H}_{k,t;\bm \eta}}{dt}\right\rangle_{\mathcal{H}_{k,t;\bm \eta}}\right].
\]
Using the definition of $\mathcal{H}_{k,t;\bm \eta}$, we have the following identity.
\begin{align}
\label{eq:diff_hamiltonian}
\frac{d\mathcal{H}_{k,t;\bm \eta}}{dt}&=\frac{d}{dt}h_{mf}(\bm x,\bm \beta;\bm \sigma_0,\bm\beta_0,\bm \Phi,\widebar{\bm{\varepsilon}}_k,\widebar{\bm Z}_k,K/\lambda tq_k,K/\lambda_k(t))\\
&\quad\quad+\frac{d}{dt}h(\bm x,\bm \beta;\bm \sigma_0,\bm \beta_0,\bm \Phi,\bm \varepsilon_k,\bm Z_k,K/(1-t)\lambda,K/\gamma_k(t))\\
& = \frac{\lambda q_k}{K}\sum\limits_{i=1}^{p}\left(\frac{x^2_i}{2}-x_i\sigma_{0i}-\frac{x_i(\widebar{\bm Z}_{k})_{i}}{2}\sqrt{\frac{K}{tq_k\lambda}}\right)\\
& \quad + \frac{d\lambda_k(t)}{dt}\frac{1}{2K}\sum\limits_{j=1}^{p}\left((\beta_j-\beta_{0j})^2-\sqrt{\frac{K}{\lambda_k(t)}}(\beta_j-\beta_{0j})(\widebar{\bm\varepsilon}_{k})_{j}\right)\\
& -\frac{\lambda}{K}\sum\limits_{\substack{i \le j\\i,j=1}}^{p}\left\{\frac{x^2_ix^2_j}{2p}-\frac{x_ix_j\sigma_{0i}\sigma_{0j}}{p}-\frac{x_ix_j(\bm Z_{k})_{ij}}{2\sqrt{p}}\sqrt{\frac{K}{(1-t)\lambda}}\right\}\\
& \quad + \frac{d\gamma_k(t)}{dt}\frac{1}{2K}\sum\limits_{\mu=1}^{n}\left\{[\bm \Phi(\bm \beta-\bm \beta_0)]^2_\mu-\sqrt{\frac{K}{\gamma_k(t)}}[\bm \Phi(\bm \beta-\bm \beta_0)]_\mu(\bm \varepsilon_k)_\mu\right\}.
\end{align}
We consider the \emph{Nishimori Identity} given below.
\begin{equation}
\label{eq:nishimori}
\mathbb{E}\left[\langle g(\bm x,\bm \beta, \bm \sigma_0, \bm \beta_0)\rangle_{\mathcal{H}_{k,t;\bm \eta}}\right]=\mathbb{E}\left[\langle g(\bm x,\bm \beta, \bm x^\prime, \bm \beta^\prime)\rangle_{\mathcal{H}_{k,t;\bm \eta}}\right],
\end{equation}
where $(\bm x,\bm \beta)$ and $(\bm x^\prime,\bm \beta^\prime)$ are drawn i.i.d from the Gibbs measure $P_{k,t;\bm \eta}(\bm x,\bm \beta|\Theta)$. \revsagr{From \cite{Stein1981}, for $Z \sim \mathsf{N}(0,1)$ and $f$ differentiable, we have
\begin{align}
\label{eq:stein}
  \mathbb E[Zf(Z)]=\mathbb E[f^\prime(Z)].  
\end{align}
Consequently, we can verify that
\[
\sqrt{\frac{K}{tq_k\lambda}}\mathbb E\left[\left\langle\frac{x_i(\widebar{\bm Z}_{k})_{i}}{2}\right\rangle_{\mathcal{H}_{k,t;\bm \eta}}\right]=\mathbb E\left[\left\langle\frac{x_i^2}{2}-\frac{x_ix^\prime_i}{2}\right\rangle_{\mathcal{H}_{k,t;\bm \eta}}\right],
\]
for all $i \in [p]$ and $(\bm x,\bm \beta)$, $(\bm x^\prime,\bm \beta^\prime)$ drawn i.i.d from the Gibbs measure $P_{k,t;\bm \eta}(\bm x,\bm \beta|\Theta)$. Similarly, we can also verify that
\[
\sqrt{\frac{K}{(1-t)\lambda}}\mathbb E\left[\left\langle\frac{x_ix_j(\bm Z_{k})_{ij}}{2\sqrt{p}}\right\rangle_{\mathcal{H}_{k,t;\bm \eta}}\right]=\mathbb E\left[\left\langle\frac{x_i^2x^2_j}{2p}-\frac{x_ix_jx^\prime_ix^\prime_j}{2p}\right\rangle_{\mathcal{H}_{k,t;\bm \eta}}\right], \quad \mbox{for all $i,j \in [p]$,}
\]
\[
\sqrt{\frac{K}{\lambda_k(t)}}\mathbb E\left[\left\langle(\beta_j-\beta_{0j})(\widebar{\bm\varepsilon}_{k})_{j}\right\rangle_{\mathcal{H}_{k,t;\bm \eta}}\right]=\mathbb E\left[\left\langle(\beta_j-\beta_{0j})^2\right\rangle_{\mathcal{H}_{k,t;\bm \eta}}-\left\langle\beta_i-\beta_{0i}\right\rangle^2_{\mathcal{H}_{k,t;\bm \eta}}\right], \quad \mbox{for all $j \in [p]$,}
\]
and
\[
\sqrt{\frac{K}{\gamma_k(t)}}\mathbb E\left[\left\langle[\bm \Phi(\bm \beta-\bm \beta_0)]_\mu(\bm \varepsilon_k)_\mu\right\rangle_{\mathcal{H}_{k,t;\bm \eta}}\right]=\mathbb E\left[\left\langle[\bm \Phi(\bm \beta-\bm \beta_0)]^2_\mu\right\rangle_{\mathcal{H}_{k,t;\bm \eta}}-\left\langle[\bm\Phi(\bm \beta-\bm\beta_0)]_\mu\right\rangle^2_{\mathcal H_{k,t;\bm \eta}}\right]
\]
for all $\mu \in [n]$. Plugging the above equations in  \eqref{eq:diff_hamiltonian}, we get the following.}
\begin{align}
\frac{d\,f_{k,t;\bm \eta}}{dt}&=\frac{\lambda}{pK}\mathbb{E}\left[\left\langle q_k\sum\limits_{i=1}^{p}\left(\frac{x_ix^\prime_i}{2}-x_i\sigma_{0i}\right)-\sum\limits_{i \le j=1}^{p}\left(\frac{x_ix_jx^\prime_ix^\prime_j}{2p}-\frac{x_ix_j\sigma_{0i}\sigma_{0j}}{p}\right)\right\rangle_{\mathcal H_{k,t;\bm \eta}}\right]\\
& \quad + \frac{d\gamma_k(t)}{dt}\frac{1}{2pK}\sum\limits_{\mu=1}^{n}\mathbb{E}\left[\left\langle[\bm\Phi(\bm \beta-\bm\beta_0)]_\mu\right\rangle^2_{\mathcal H_{k,t;\bm \eta}}\right]+\frac{d\lambda_k(t)}{dt}\frac{1}{2pK}\sum\limits_{i=1}^{p}\mathbb{E}\left[\left\langle\beta_i-\beta_{0i}\right\rangle^2_{\mathcal H_{k,t;\bm \eta}}\right].
\end{align}
Let us define \revsagr{the following quantities}.
\begin{equation}
\label{eq:single_letter}
\widebar{\mathsf{y}}_{k,t;\bm \eta}:=\frac{1}{n}\mathbb{E}[\|\bm \Phi(\langle\bm \beta\rangle_{\mathcal{H}_{k,t;\bm \eta}}-\bm \beta_0)\|^2_2] \quad \mbox{and} \quad \mathsf{mmse}_{k,t;\bm \eta}:=\frac{1}{p}\mathbb{E}[\|\langle \bm \beta\rangle_{\mathcal{H}_{k,t;\bm \eta}}-\bm \beta_0\|^2_2].
\end{equation}
Therefore
\begin{align}
\frac{df_{k,t;\bm \eta}}{dt} & = \frac{\lambda}{pK}\mathbb{E}\left[\left\langle q_k\sum\limits_{i=1}^{p}\left(\frac{x_ix^\prime_i}{2}-x_i\sigma_{0i}\right)-\sum\limits_{\substack{i\le j\\i,j=1}}^{p}\left(\frac{x_ix_jx^\prime_ix^\prime_j}{2p}-\frac{x_ix_j\sigma_{0i}\sigma_{0j}}{p}\right)\right\rangle_{\mathcal{H}_{k,t;\bm \eta}}\right]\\
&\quad+\frac{d\gamma_k(t)}{dt}\frac{\kappa}{2K}\left[\widebar{\mathsf{y}}_{k,t;\bm \eta}-\frac{1}{(1+\gamma_k(t)E_k)^2}\mathsf{mmse}_{k,t;\bm \eta}\right],
\end{align}
where the last equality follows from \eqref{eq:param_lin_mod}. Now using \eqref{eq:nishimori}; we get
\begin{align}
\frac{df_{k,t;\bm \eta}}{dt}&=\mathbb{E}\left[\left\langle-\frac{\lambda}{pK}\frac{q_k}{2}\sum\limits_{i=1}^{p}\sigma_{0i}x_i+\frac{\lambda}{2Kp^2}\sum_{\substack{i \le j\\i,j=1}}^px_ix_j\sigma_{0i}\sigma_{0j}\right\rangle_{\revsag{\mathcal{H}_{k,t;\bm \eta}}}\right]\\
& \hskip 5em +\revsagr{\frac{d\gamma_k(t)}{dt}}\frac{\kappa}{2K}\left[\widebar{\mathsf{y}}_{\revsag{k,t;\bm \eta}}-\frac{1}{(1+\gamma_k(t)E_k)^2}\mathsf{mmse}_{\revsag{k,t;\bm \eta}}\right].
\end{align}
Furthermore, using the Cauchy Schwartz Inequality and the Nishimori Identity, we get
\[
\mathbb{E}\left[\left\langle\frac{1}{n}\sum\limits_{i=1}^{n}x^2_i\sigma^2_{0i}\right\rangle_{\revsag{\mathcal{H}_{k,t;\bm \eta}}}\right] \le \left(\mathbb{E}\left[\left\langle\frac{1}{n}\sum\limits_{i=1}^{p}x^4_i\right\rangle_{\revsag{\mathcal{H}_{k,t;\bm \eta}}}\right]\right)^{1/2} \left(\mathbb{E}\left[\left\langle\frac{1}{n}\sum\limits_{i=1}^{p}\sigma^4_{0i}\right\rangle_{\revsag{\mathcal{H}_{k,t;\bm \eta}}}\right]\right)^{1/2}=\mathbb{E}[\Sigma^4].
\]
Therefore, we have
\begin{align}
\label{eq:diff_free_energy}
\frac{df_{k,t;\bm \eta}}{dt}&=\mathbb{E}\left[\left\langle\frac{\lambda}{4K}s^2_{\bm x,\bm \sigma_0}-\frac{\lambda q_k}{\revsagr{2}K}s_{\bm x,\bm \sigma_0}\right\rangle_{\revsag{\mathcal{H}_{k,t;\bm \eta}}}\right]\\
&+\frac{\kappa}{2K}\revsagr{\frac{d\gamma_k(t)}{dt}}\left[\widebar{\mathsf{y}}_{k,t;\bm \eta}-\frac{1}{(1+\gamma_k(t)E_k)^2}\mathsf{mmse}_{k,t;\bm \eta}\right]+O\left(\frac{1}{nK}\right),
\end{align}
where $s_{\bm x,\bm \sigma_0}=(1/p)\sum\limits_{i=1}^{p}\sigma_{0i}x_i$ is the empirical overlap.
\revsag{Next, observing that $f_{k+1,0;\bm \eta}=f_{k,1;\bm \eta}$ for all $k \in \{1,\ldots,K-1\}$, we have} 
\begin{align}
\label{eq:34}
f_{1,0;\bm \eta}&=f_{K,1;\bm \eta}+\sum\limits_{k=1}^{K}(f_{k,0;\bm \eta}-f_{k,1;\bm \eta})\\
&=f_{K,1;\bm \eta}-\sum\limits_{k=1}^{K}\int\limits_{0}^{1}\frac{df_{k,t;\bm \eta}}{dt}\,dt.
\end{align}
We consider the following lemma characterizing the concentration of the empirical overlap $s_{\bm x,\bm \sigma_0}$ around its mean \revsag{under the Gibbs measure $P_{k,t;\bm \eta}$.}
\begin{lem}
\label{lem:x_overlap}
\revsagr{For $K_p \rightarrow \infty$, $0<a^{(1)}_p<b^{(1)}_p<1$, consider the sequence of parameters $\{q_k\}_{k=1}^{K_p}$, possibly depending on $\bm \eta$ such that as a functions of $\bm \eta$, $\{q_k(\bm \eta)\}$ are differentiable and bounded. Furthermore, for all $\eta_2$, the functions $\{q_k(\cdot,\eta_2)\}$ are increasing in $\eta_1$.} Then, we have an absolute constant $C_1>0$ and $0<\alpha<1/4$ such that,
\[
\int\limits_{a^{(1)}_p}^{b^{(1)}_p}d\eta_1\left(\frac{1}{K_p}\sum\limits_{k=1}^{K_p}\int\limits_{0}^{1}dt\,\mathbb{E}\left[\left\langle(s_{\bm x,\bm \sigma_0}-\mathbb{E}[\langle s_{\bm x,\bm \sigma_0}\rangle_{\revsag{\mathcal{H}_{k,t;\bm \eta}}}])^2\right\rangle_{\revsag{\mathcal{H}_{k,t;\bm \eta}}}\right]\right) \le \frac{C}{(a^{(1)}_p)^2p^\alpha},
\]
where the Hamiltonian is computed with respect to the Gibbs measure characterized by the parameters $\{q_k\}_{k=1}^{K_p}$.
\end{lem}
Similarly, also consider the following lemma characterizing the concentration of $\widebar{\mathsf{y}}_{\revsag{k,t;\bm \eta}}$.
\begin{lem}
\label{lem:mmse_overlap}
For $K_p \rightarrow \infty$, $0<a^{(2)}_p<b^{(2)}_p<1$, \revsagr{consider the sequence of parameters $\{E_k\}_{k=1}^{K_p}$, possibly depending on $\bm \eta$ such that the functions $\{E_k(\bm \eta)\}$ are differentiable and bounded. Furthermore, for all fixed $\eta_1$, the functions $\{E_k(\eta_1,\cdot)\}$ are non-increasing with respect to $\eta_2$. Then we have an absolute constant $C_1>0$ and the same $\alpha$ as Lemma \ref{lem:x_overlap} such that,}
\begin{align}
\label{eq:mmse_overlap}
\int\limits_{a^{(2)}_p}^{b^{(2)}_p}d\eta_2\left(\frac{1}{K_p}\sum\limits_{k=1}^{K_p}\int\limits_{0}^{1}dt\frac{d\gamma_k(t)}{dt}\left\{\widebar{\mathsf{y}}_{\revsag{k,t;\bm \eta}}-\frac{\mathsf{mmse}_{\revsag{k,t;\bm \eta}}}{1+\gamma_k(t)\mathsf{mmse}_{\revsag{k,t;\bm \eta}}}\right\}\right)=O\left(\frac{1}{(a^{(2)}_p)^2p^\alpha}\right).
\end{align}
\end{lem}
\revsagr{The proof of this lemma follows using the definition of $\widebar{\mathsf{y}}_{\revsag{k,t;\bm \eta}}$, \eqref{eq:nishimori} and \eqref{eq:stein}. The details of the proof closely follow the techniques outlined in the proof of Lemma 4.6 of \cite{barbier2020mutual}.}
Next, let us consider the sequence of parameters,
\[
V_{K}=\frac{1}{K}\sum\limits_{k=1}^{K}q^2_k-\left(\frac{1}{K}\sum\limits_{k=1}^{K}q_k\right)^2.
\]
\revsagr{Let us define 
\begin{align}
\label{eq:combined_free_energy}
f_{\rs}\left(\{q_k\}_{k=1}^{K_p},\{E_k\}_{k=1}^{K_p},\Delta,\lambda\right)=\frac{\lambda}{4}\left(\frac{1}{K_p}\sum\limits_{k=1}^{K_p}q_k\right)^2+\frac{1}{K_p}\sum\limits_{k=1}^{K_p}\psi(\Delta;E_k)+\wt{f}\left(\frac{1}{\lambda\widebar{q}},\widebar{\Psi}^{2}\right),\nonumber\\
\end{align}
and $\widebar{q},\widebar \Psi^2$ are defined in \eqref{eq:effec_param}. Observe that
\[
\psi(E_k;\Delta)=\frac{\kappa}{2}\int\limits_{0}^{1}\frac{d\gamma_k(t)}{dt}\left(\frac{E_k}{(1+\gamma_k(t)E_k)^2}-\frac{E_k}{1+\gamma_k(t)E_k}\right)dt.
\]
Therefore, combining \eqref{eq:diff_free_energy},\eqref{eq:34},\eqref{eq:mmse_overlap}, we get}
\begin{align}
\label{eq:calc_1}
&\int\limits_{a^{(1)}_p}^{b^{(1)}_p}\int\limits_{a^{(2)}_p}^{b^{(2)}_p}f_{1,0;\bm \eta}d\eta_2d\eta_1\\
&=\int\limits_{a^{(1)}_p}^{b^{(1)}_p}\int\limits_{a^{(2)}_p}^{b^{(2)}_p}\Bigg\{(f_{K_p,1;\bm\eta}-f_{K_p,1;\bm 0})+\frac{\lambda}{4}V_{K_p}\Bigg\}d\eta_2d\eta_1\\
&\quad\quad-\frac{\lambda}{4}\int\limits_{a^{(1)}_p}^{b^{(1)}_p}\int\limits_{a^{(2)}_p}^{b^{(2)}_p}\left[\frac{1}{K_p}\sum\limits_{k=1}^{K_p}\int\limits_{0}^{1}\mathbb{E}\left[\langle(s_{\bm x,\bm \sigma_0}-q_k)^2\rangle_{\mathcal H_{k,t;\bm \eta}}\right]dt\right]d\eta_2d\eta_1\\
&\quad-\frac{\kappa}{2}\int\limits_{a^{(1)}_p}^{b^{(1)}_p}\int\limits_{a^{(2)}_p}^{b^{(2)}_p}\bigg\{\frac{1}{K_p}\sum\limits_{k=1}^{K_p}\int\limits_{0}^{1}\frac{d\gamma_k(t)}{dt}\Bigg[\frac{\mathsf{mmse}_{k,t;\bm \eta}}{1+\gamma_k(t)\mathsf{mmse}_{k,t;\bm \eta}}-\frac{\mathsf{mmse}_{k,t;\bm \eta}}{\revsagr{(1+\gamma_k(t)E_k)^2}}\\
& \quad\quad\quad +\frac{E_k}{(1+\gamma_k(t)E_k)^2}-\frac{E_k}{(1+\gamma_k(t)E_k)}\Bigg]\,dt\bigg\}d\eta_2d\eta_1\\
&\quad+\int\limits_{a^{(1)}_p}^{b^{(1)}_p}\int\limits_{a^{(2)}_p}^{b^{(2)}_p}f_{\rs}\left(\{q_k\}_{k=1}^{K_p},\{E_k\}_{k=1}^{K_p},\Delta,\lambda\right)d\eta_2d\eta_1+O\left(\frac{1}{(a^{(1)}_p\wedge a^{(2)}_p)^2p^\alpha}\right),
\end{align}

Now by using \eqref{eq:calc_1} and Lemmas \ref{lem:x_overlap} and \ref{lem:mmse_overlap}, we get the following equation.
\begin{align}
\label{eq:calc_3}
\int\limits_{a^{(1)}_p}^{b^{(1)}_p}\int\limits_{a^{(2)}_p}^{b^{(2)}_p}f_{1,0;\bm \eta}\,d\eta_2d\eta_1&=\int\limits_{a^{(1)}_p}^{b^{(1)}_p}\int\limits_{a^{(2)}_p}^{b^{(2)}_p}\Bigg\{(f_{K_p,1;\bm\eta}-f_{K_p,1;\bm 0})+\frac{\lambda}{4}V_{K_p}\Bigg\}d\eta_2d\eta_1\\
&\quad\quad-\frac{\lambda}{4}\int\limits_{a^{(1)}_p}^{b^{(1)}_p}\int\limits_{a^{(2)}_p}^{b^{(2)}_p}\left[\frac{1}{K_p}\sum\limits_{k=1}^{K_p}\int\limits_{0}^{1}\revsag{\left\{\mathbb{E}[\langle s_{\bm x,\bm \sigma_0}\rangle_{\mathcal H_{k,t;\bm \eta}}]-q_k\right\}^2}dt\right]d\eta_2d\eta_1\\
&\quad\revsagr{+}\frac{\kappa}{2}\int\limits_{a^{(1)}_p}^{b^{(1)}_p}\int\limits_{a^{(2)}_p}^{b^{(2)}_p}\left[\frac{1}{K_p}\sum\limits_{k=1}^{K_p}\int\limits_{0}^{1}\frac{d\gamma_k(t)}{dt}\frac{\gamma_k(t)(E_k-\mathsf{mmse}_{k,t;\bm \eta})^2}{(1+\gamma_k(t)E_k)^2(1+\gamma_k(t)\mathsf{mmse}_{k,t;\bm \eta})}\,dt\right]d\eta_2d\eta_1\\
&\quad+\int\limits_{a^{(1)}_p}^{b^{(1)}_p}\int\limits_{a^{(2)}_p}^{b^{(2)}_p}f_{\rs}\left(\{q_k\}_{k=1}^{K_p},\{E_k\}_{k=1}^{K_p},\Delta,\lambda\right)d\eta_2d\eta_1+O\left(\frac{1}{(a^{(1)}_p\wedge a^{(2)}_p)^2p^\alpha}\right),
\end{align}
\paragraph{Upper Bound to the limit of the Free Energy} To get the upper bound to the asymptotic limit of the free energy let us consider the following lemma.
\begin{lem}
\label{lem:terminal_cond}
For $\bm \eta=(\eta_1,\eta_2)$ and $f_{k,t;\bm \eta}$ defined in \eqref{eq:free_energy_perturbed}, we have constants $C_1,C_2>0$,
\[
\left|f_{1,0;\bm \eta}-f_{1,0;\bm 0}\right| \le C_1\|\bm \eta\| \quad \mbox{and} \quad \left|f_{K,1;\bm \eta}-f_{K,1;\bm 0}\right| \le C_2\|\bm \eta\|.
\]
\end{lem} 
Let us choose the interpolation parameters as 
\begin{align}
\label{eq:choice_interp_up_bound}
(q_k,E_k):=\argmin\bigg\{f_{\rs}(\lambda,q;E,\Delta):q \ge 0, E \ge 0\bigg\} \quad \mbox{for all $1 \le k \le K_p.$}
\end{align}
Then using the Mean Value Theorem, we get the following:
\begin{equation}
\label{eq:mvt_argument}
\int\limits_{a^{(1)}_p}^{b^{(1)}_p}\int\limits_{a^{(2)}_p}^{b^{(2)}_p}f_{1,0;\bm \eta}d\eta_2d\eta_1 = (b^{(1)}_p-a^{(1)}_p)(b^{(2)}_p-a^{(2)}_p)f_{1,0;\bm \eta^*},
\end{equation}
where $\bm \eta^*=(\eta^*_1,\eta^*_2)$ for $\eta^*_1 \in (a^{(1)}_p,b^{(1)}_p)$ and $\eta^*_2 \in (a^{(2)}_p,b^{(2)}_p)$. \revsagr{Similarly, we can get $\check{\bm \eta}=(\check{\eta}_1,\check{\eta}_2)$ for $\check{\eta}_1 \in (a^{(1)}_p,b^{(1)}_p)$ and $\check{\eta}_2 \in (a^{(2)}_p,b^{(2)}_p)$, such that
\begin{equation}
\label{eq:mvt_argument_2}
\int\limits_{a^{(1)}_p}^{b^{(1)}_p}\int\limits_{a^{(2)}_p}^{b^{(2)}_p}(f_{K_p,1;\bm\eta}-f_{K_p,1;\bm 0})d\eta_2d\eta_1 = (b^{(1)}_p-a^{(1)}_p)(b^{(2)}_p-a^{(2)}_p)(f_{K_p,1;\check{\bm\eta}}-f_{K_p,1;\bm 0})
\end{equation}
By \eqref{eq:choice_interp_up_bound}, we have $V_{K_p}=0$, and
\[
f_{\rs}\left(\{q_k\}_{k=1}^{K_p},\{E_k\}_{k=1}^{K_p},\Delta,\lambda\right)=\inf\bigg\{f_{\rs}(\lambda,q;E,\Delta):q \ge 0, E \ge 0\bigg\}.
\]
Since $\frac{d\gamma_k(t)}{dt} \le 0$ for all $t$, plugging in the above equalities along with \eqref{eq:mvt_argument}, \eqref{eq:mvt_argument_2} in \eqref{eq:calc_3}, along with Lemma \ref{lem:terminal_cond}, we get
\begin{align}
&\limsup_{p \rightarrow \infty}\,(b^{(1)}_p-a^{(1)}_p)(b^{(2)}_p-a^{(2)}_p)f_{1,0;\bm \eta^*}\\ &\le \limsup_{p \rightarrow \infty}\,(b^{(1)}_p-a^{(1)}_p)(b^{(2)}_p-a^{(2)}_p)(f_{K_p,1;\check{\bm\eta}}-f_{K_p,1;\bm 0}) \\
&+ \limsup_{p \rightarrow \infty}\,(b^{(1)}_p-a^{(1)}_p)(b^{(2)}_p-a^{(2)}_p)\inf\bigg\{f_{\rs}(\lambda,q;E,\Delta):q \ge 0, E \ge 0\bigg\}.
\end{align}
} 
Now, taking $b^{(i)}_p=2a^{(i)}_p$ for $i=1,2$, and $a^{(i)}_p = p^{-\alpha/4}$ in \eqref{eq:calc_3}, we get the following upper bound using the continuity of $f_{1,0;\bm \eta}$ in $\bm \eta$ as $\bm \eta \rightarrow \bm 0$:
\[
\limsup\limits_{p \rightarrow \infty}f_p \le \min\limits_{(q,E)\ge 0}f_{\rs}(\lambda,q;E,\Delta).
\] 
\paragraph{Lower Bound to the limit of the Free Energy}
To show the lower bound, we consider the following lemma.
\begin{lem}
\label{lem:overlap_approx}
Let us fix $K,\bm \eta$ and $\{q_k\}_{k=1}^{K}$. Furthermore, let \revsagr{$P(\,\cdot\,|\,\Sigma)$ satisfy
\[\mathbb E[|B|^4|\Sigma=i] \le C, \quad \mbox{for $i \in \{0,1\}$ and $C>0$ (independent of $i$).}\]
Then, for all $k \in \{1,\cdots,K\}$ and $t \in (0,1)$, we have}
\[
\left|\mathbb{E}\left[\left\langle s_{\bm x, \bm \sigma_0}\right\rangle_{\revsag{\mathcal{H}_{k,t;\bm \eta}}}\right]-\mathbb{E}\left[\left\langle s_{\bm x, \bm \sigma_0}\right\rangle_{\revsag{\mathcal{H}_{k,0;\bm \eta}}}\right]\right|=O\left(\frac{p}{K}\right).
\]
\end{lem}
\revsagr{Using Gaussian integration by parts and Cauchy Schwartz inequality as in the  proof of Lemma~\ref{lem:overlap_approx} and utilizing the fact that $P(\,\cdot\,|\,\Sigma)$ has discrete and bounded support for $\Sigma \in \{0,1\}$, we can show that for $k \in \{1,\cdots,K\}$ and $t \in (0,1)$,
\[
\left|\mathsf{mmse}_{k,t;\bm \eta}-\mathsf{mmse}_{k,0;\bm \eta}\right|=O\left(\frac{p}{K}\right).
\]}
Now let us take $K_p=\Omega(n^b)$ with $b=2$ \revsag{using} Lemma \ref{lem:overlap_approx}; we can re-write \eqref{eq:calc_3} as
\begin{align}
\label{eq:lower_bound}
\int\limits_{a^{(1)}_p}^{b^{(1)}_p}\int\limits_{a^{(2)}_p}^{b^{(2)}_p}f_{1,0;\bm \eta}d\eta_2d\eta_1&=\int\limits_{a^{(1)}_p}^{b^{(1)}_p}\int\limits_{a^{(2)}_p}^{b^{(2)}_p}\Bigg\{(f_{K_p,1;\bm\eta}-f_{K_p,1;\bm 0})+\frac{\lambda}{4}V_{K_p}\Bigg\}d\eta_2d\eta_1\\
&\quad\quad-\frac{\lambda}{4}\int\limits_{a^{(1)}_p}^{b^{(1)}_p}\int\limits_{a^{(2)}_p}^{b^{(2)}_p}\left[\frac{1}{K_p}\sum\limits_{k=1}^{K_p}\int\limits_{0}^{1}\left\{\mathbb{E}\left[\langle s_{\bm x,\bm \sigma_0}\rangle_{\revsag{\mathcal{H}_{k,0;\bm \eta}}}\right]-q_k\right\}^2dt\right]d\eta_2d\eta_1\\
&\revsagr{+}\frac{\kappa}{2}\int\limits_{a^{(1)}_p}^{b^{(1)}_p}\int\limits_{a^{(2)}_p}^{b^{(2)}_p}\left[\frac{1}{K_p}\sum\limits_{k=1}^{K_p}\int\limits_{0}^{1}dt\,\frac{d\gamma_k(t)}{dt}\frac{\gamma_k(t)(E_k-\mathsf{mmse}_{k,0;\bm \eta})^2}{(1+\gamma_k(t)E_k)^2(1+\gamma_k(t)\mathsf{mmse}_{k,0;\bm \eta})}\right]d\eta_2d\eta_1\\
&\quad+\int\limits_{a^{(1)}_p}^{b^{(1)}_p}\int\limits_{a^{(2)}_p}^{b^{(2)}_p}f_{\rs}\left(\{q_k\}_{k=1}^{K_p},\{E_k\}_{k=1}^{K_p},\Delta,\lambda\right)d\eta_2d\eta_1+O\left(\frac{1}{(a^{(1)}_p\wedge a^{(2)}_p)^2p^\alpha}\right),
\end{align}
Let us take $E_k=\mathsf{mmse}_{k,0;\bm \eta}$ and $q_k=\mathbb{E}[\langle s_{\bm x,\bm \sigma_0}\rangle_{\revsag{\mathcal{H}_{k,0;\bm \eta}}}]$. Now we need to show that $E_k$ and $q_k$ are valid parameters, \revsag{that satisfies the assumptions of Lemmas \ref{lem:x_overlap} and \ref{lem:mmse_overlap}}.
\begin{lem}
\label{lem:valid_param}
\revsagr{For each $n,p \in \mathbb N$, the parameters $E_k$'s and $q_k$ defined as
\[E_k=\mathsf{mmse}_{k,0;\bm \eta}, \quad \mbox{and,} \quad q_k=\mathbb{E}[\langle s_{\bm x,\bm \sigma_0}\rangle_{\revsag{\mathcal{H}_{k,0;\bm \eta}}}]\]
 are bounded and differentiable as functions of $\bm \eta$. Furthermore, for all $1 \le k \le K_p$, $E_k(\eta_1,\cdot)$ is non-increasing in $\eta_2$ for all $\eta_1>0$ and $q_k(\cdot,\eta_2)$ is non-decreasing in $\eta_1$ for all $\eta_2>0$. Furthermore, the parameters $\{E_k:k \in [K_p]\}$ and $\{q_k:k \in [K_p]\}$ are well defined.}
\end{lem}
\begin{proof}
The proof follows inductively as in the proof of Lemma 4 of \citet{Adaptive_Interpolation}.
\end{proof}
Using Lemma \ref{lem:valid_param}, \eqref{eq:lower_bound}, and using the \revsagr{Mean Value Theorem}, we get
\[
\liminf\limits_{p \rightarrow \infty} f_p \ge \min\limits_{(q,E) \ge 0}f_{\rs}(\lambda,q;E,\Delta).
\]
Combining the upper and the lower bound, we have the following limit.
\begin{equation}
\label{eq:lim_free_energy}
\lim\limits_{p \rightarrow \infty} f_p = \min\limits_{(q,E) \ge 0}f_{\rs}(\lambda,q;E,\Delta).
\end{equation}
\subsection{Limit of the Mutual Information}
\paragraph{Limit under the Gaussian Wigner Model.} From \eqref{eq:lim_mut_inform} and \eqref{eq:lim_free_energy}, we get
\[
\lim\limits_{p \rightarrow \infty}\frac{1}{p}I(\bm \beta_0,\bm \sigma_0;\widetilde{\bm A},\bm y)= \mathbb{E}_{(\Sigma,B)}[\log P(B\,|\,\Sigma)Q(\Sigma)]+\frac{\lambda}{4}(\mathbb E[\Sigma^2])^2+\lim\limits_{p \rightarrow \infty} f_p.
\]
From the definition this implies \revsagr{
\begin{align}
\label{eq:final_lim}
\lim\limits_{p \rightarrow \infty}\frac{1}{p}I(\bm \beta_0,\bm \sigma_0;\widetilde{\bm A},\bm y)&=\min\limits_{q,E \ge 0}\Bigg\{\mathbb{E}_{(\Sigma,B)}[\log P(B\,|\,\Sigma)Q(\Sigma)]+\frac{\lambda}{4}(\mathbb E[\Sigma^2])^2+\frac{\lambda q^2}{4}\\
&\quad\quad +\frac{\kappa}{2}\left[\log\left(1+\frac{E}{\Delta}\right)-\frac{E}{\Delta+E}\right]+\wt{f}\left(\frac{1}{\lambda q},\Psi(E,\Delta)^2\right)\Bigg\}\\
&= \min\limits_{q,E\ge 0}\Bigg\{\mathbb{E}_{(\Sigma,B)}[\log P(B\,|\,\Sigma)Q(\Sigma)]+\frac{\lambda}{4}(\mathbb E[\Sigma^2])^2+\frac{\lambda q^2}{4}\\
&\quad+\frac{\kappa}{2}\left[\log\left(1+\frac{E}{\Delta}\right)-\frac{E}{\Delta+E}\right]\\
&-\mathbb{E}_{(\Sigma,B,\widebar{Z},\widebar{\varepsilon})}\Bigg[\log\int d\beta dxP(\beta|x)Q(x)\exp\Bigg(-\lambda q\Bigg(\frac{x^2}{2}-\Sigma x-\frac{1}{\sqrt{\lambda q}}x\widebar{Z}\Bigg)\\
&\quad\quad-\frac{\kappa}{\Delta+E}\Bigg(\frac{(\beta-B)^2}{2}-\sqrt{\frac{\Delta+E}{\kappa}}(\beta-B)\widebar{\varepsilon}\Bigg)\Bigg)\Bigg]\Bigg\}.
\end{align}
where $\Sigma \sim\mathrm{Bernoulli}(\rho)$, $B|\Sigma \sim P(\,\cdot\,|\,\cdot)$ and $\widebar{\varepsilon},\widebar{Z} \sim \mathsf{N}(0,1)$.}
Let us define the following parameters
\[
\mu:= \lambda q \quad \quad \mbox{and} \quad \quad \xi:=E/\Delta.
\]
Rewriting the above equation, we get
\begin{align}
\label{eq:final_lim_1}
\lim\limits_{p \rightarrow \infty}\frac{1}{p}I(\bm \beta_0,\bm \sigma_0;\widetilde{\bm A},\bm y)&= \min\limits_{\mu,\xi\ge 0}\Bigg\{\mathbb{E}_{(\Sigma,B)}[\log P(B\,|\,\Sigma)Q(\Sigma)]+\frac{\lambda}{4}(\mathbb E[\Sigma^2])^2+\frac{\mu^2}{4\lambda}\\
&\quad+\frac{\kappa}{2}\left[\log\left(1+\xi\right)-\frac{\xi}{1+\xi}\right]\\
&\revsagr{-}\mathbb{E}_{(\Sigma,B,\widebar{Z},\widebar{\varepsilon})}\Bigg[\log\int d\beta dxP(\beta|x)Q(x)\exp\Bigg(-\mu\Bigg(\frac{x^2}{2}-\Sigma x-\frac{1}{\sqrt{\mu}}x\widebar{Z}\Bigg)\\
&\quad\quad-\frac{\kappa}{\Delta(1+\xi)}\Bigg(\frac{(\beta-B)^2}{2}-\sqrt{\frac{\Delta(1+\xi)}{\kappa}}(\beta-B)\widebar{\varepsilon}\Bigg)\Bigg)\Bigg]\Bigg\}.
\end{align}
Let us define the variables
\[
a:=\sqrt{\mu}\Sigma+\widebar{Z} \quad \mbox{and} \quad y:=B+\sqrt{\frac{\Delta(1+\xi)}{\kappa}}\widebar{\varepsilon}.
\]
Further, define
\[
\mathsf{I}(\mu,\xi;\Delta):=\mathbb{E}\left[\log\frac{P(a,y|\Sigma,B)}{P(a,y)}\right].
\]
By simple calculations, we can show that
\begin{align}
\lim\limits_{p \rightarrow \infty}\frac{1}{p}I(\bm \beta_0,\bm \sigma_0;\widetilde{\bm A},\bm y)&= \min\limits_{\mu,\xi\ge 0}\Bigg\{\frac{\lambda}{4}(\mathbb E[\Sigma^2])^2+\frac{\mu^2}{4\lambda}+\frac{\kappa}{2}\left[\log\left(1+\xi\right)-\frac{\xi}{1+\xi}\right]\\
&\hspace{1.3in} - \frac{\mu}{2}\mathbb{E}[\Sigma^2]+\mathsf{I}(\mu,\xi;\Delta)\Bigg\}.
\end{align}
If we plug in the values of $\mathbb E[\Sigma^2]$, we get
\begin{equation}
\label{eq:lim_inf}
\lim\limits_{p \rightarrow \infty}\frac{1}{p}I(\bm \beta_0,\bm \sigma_0;\widetilde{\bm A},\bm y)= \min\limits_{(\mu,\xi)\ge 0}\Bigg\{\frac{\lambda\rho^2}{4}+\frac{\mu^2}{4\lambda}+\frac{\kappa}{2}\left[\log\left(1+\xi\right)-\frac{\xi}{1+\xi}\right]- \frac{\mu\rho}{2}+\mathsf{I}(\mu,\xi;\Delta)\Bigg\}.
\end{equation}

\paragraph{Limit under the Graphical Model.} Now we consider the graph constrained regression model given by \eqref{eq:graph} and \eqref{eq:lin_model_1}. Let us observe that
\[
\mathbb{P}(G_{ij}=1)=\widebar{d}_p+\sqrt{\frac{\lambda\widebar{d}_p(1-\widebar{d}_p)}{p}} \sigma_{0i}\sigma_{0j},
\]
where $\widebar{d}_p:=\frac{b_p}{p}$.
Now using the techniques used to prove Theorem 3.1 of \citet{ma_nandy} and Proposition 3.1 of \citet{AbbeMonYash} we can prove the following theorem.
\begin{thm}
If $p\,\widebar{d}_p(1-\widebar{d}_p)\rightarrow \infty$, then as $n,p \rightarrow \infty$ we have constant $C>0$ such that:
\[
\Bigg|\frac{1}{p}I(\bm \beta_0,\bm \sigma_0;\widetilde{\bm A},\bm y)-\frac{1}{p}I(\bm \beta_0,\bm \sigma_0;\bm G,\bm y)\Bigg| \le C\,\frac{\lambda^{3/2}}{\sqrt{p\,\widebar{d}_p(1-\widebar{d}_p)}}.
\]
\end{thm}
This implies as $p\,\widebar{d}_p(1-\widebar{d}_p)\rightarrow \infty$ we have
\begin{align}
\lim\limits_{p \rightarrow \infty}\frac{1}{p}I(\bm \beta_0,\bm \sigma_0;\bm A,\bm y)&=  \min\limits_{(\mu,\xi)\ge 0}\Bigg\{\frac{\lambda\rho^2}{4}+\frac{\mu^2}{4\lambda}+\frac{\kappa}{2}\left[\log\left(1+\xi\right)-\frac{\xi}{1+\xi}\right]- \frac{\mu\rho}{2}+\mathsf{I}(\mu,\xi;\Delta)\Bigg\}.
\end{align}

\section{Proof of the lemmas of Section \ref{mut_inf}}
\subsection{Proof of Lemma \ref{lem:x_overlap}}
Let us define,
\[
\mathcal{L}:= \frac{1}{p}\sum\limits_{i=1}^{p}\left\{\frac{x^2_i}{2}-x_i\sigma_{0i}-\frac{x_i\revsagr{\widehat{Z}_i}}{2\sqrt{\widetilde{\eta}}}\right\},
\]
where,
\[
\revsagr{\widetilde{\eta}(\bm \eta)=\eta_1+\frac{\lambda}{K}\left(\sum\limits_{\ell=1}^{k-1}q_\ell(\bm \eta)+tq_k(\bm \eta)\right), \quad \mbox{where $\bm \eta=(\eta_1,\eta_2)$.}}
\]
\revsagr{Observe that, since for all $\eta_2>0$ and $1 \le \ell \le K$, $q_k(\cdot,\eta_2)$ is monotone in $\eta_1$, therefore $\bm \eta$ can be re-parametrized as $\bm \eta=\widetilde g(\widetilde \eta,\eta_2)$, where $\widetilde g$ is a differentiable function of $\widetilde \eta$ and $\eta_2$. Consequently, $\mathcal{L}$ is also a differentiable function of $\widetilde \eta$ and $\eta_2$.}
We shall show the following proposition.
\begin{prop}
\label{eq:prob_mathcal_l}
\revsagr{Let us consider a sequence of interpolation parameters $\{q_k(\bm \eta)\}_{k=1}^{K}$, such that as a function of $\bm \eta$ it is bounded and differentiable. Furthermore, for all $\eta_2>0$, the function $q_k(\cdot,\eta_2)$ is non-decreasing in $\eta_1$. Then for any sequence $0<a^{(1)}_p<b^{(1)}_p<1$ and $0<a^{(2)}_p<b^{(2)}_p<1$, we have}
\[
\int\limits_{a^{(1)}_p}^{b^{(1)}_p}\int\limits_{a^{(2)}_p}^{b^{(2)}_p}d\eta_2d\eta_1\mathbb{E}\left[\left\langle(\mathcal{L}-\mathbb{E}[\langle\mathcal{L}\rangle_{\revsag{\mathcal H_{k,t;\bm \eta}}}])^2\right\rangle_{\revsag{\mathcal H_{k,t;\bm \eta}}}\right] \le \frac{C}{p^\alpha (a^{(1)}_p \wedge a^{(2)}_p)^\alpha},
\]
for any $0<\alpha<1/4$ with $C>0$.
\end{prop}
\revsagr{Following Section 6 of \citet{Adaptive_Interpolation}}, we observe that $\mathbb{E}\left[\left\langle(\mathcal{L}-\mathbb{E}[\langle\mathcal{L}\rangle_{\revsag{\mathcal H_{k,t;\bm \eta}}}])^2\right\rangle_{\revsag{\mathcal H_{k,t;\bm \eta}}}\right]$ can be equivalently expressed as follows.
\begin{align}
\label{eq:exact_formula}
\mathbb{E}\left[\left\langle(\mathcal{L}-\mathbb{E}\left[\langle\mathcal{L}\rangle_{\revsag{\mathcal H_{k,t;\bm \eta}}}\right])^2\right\rangle_{\revsag{\mathcal H_{k,t;\bm \eta}}}\right] & =\frac{1}{4p^2}\sum\limits_{i,j=1}^{p}\left\{\mathbb{E}\left[\langle x_ix_j\rangle^2_{\revsag{\mathcal H_{k,t;\bm \eta}}}\right]-\mathbb{E}\left[\langle x_i\rangle^2_{\revsag{\mathcal H_{k,t;\bm \eta}}}\right]\mathbb{E}\left[\langle x_j\rangle^2_{\revsag{\mathcal H_{k,t;\bm \eta}}}\right]\right\}\nonumber\\
&+\frac{1}{2p^2}\sum\limits_{i,j=1}^{p}\left\{\mathbb{E}\left[\langle x_ix_j\rangle^2_{\revsag{\mathcal H_{k,t;\bm \eta}}}\right]-\mathbb{E}\left[\langle x_ix_j\rangle_{\revsag{\mathcal H_{k,t;\bm \eta}}}\langle x_i\rangle_{\revsag{\mathcal H_{k,t;\bm \eta}}}\langle x_j\rangle_{\revsag{\mathcal H_{k,t;\bm \eta}}}\right]\right\}\nonumber\\
&+\frac{1}{4p^2\widetilde{\eta}}\sum\limits_{i=1}^{m}\mathbb{E}\left[\langle x^2_i\rangle_{\revsag{\mathcal H_{k,t;\bm \eta}}}\right],
\end{align}
\revsagr{where $x_i \overset{i.i.d}{\sim}P_{k,t;\bm \eta}$ defined in \eqref{eq:gibbs_mesr}.} Using the Nishimori Identity, we get
\[
\frac{1}{p^2}\sum\limits_{i,j=1}^{p}\mathbb{E}\left[\langle x_ix_j\rangle^2_{\revsag{\mathcal H_{k,t;\bm \eta}}}\right]=\frac{1}{p^2}\sum\limits_{i,j=1}^{p}\mathbb{E}\left[\sigma_{0i}\sigma_{0j}\langle x_ix_j\rangle_{\revsag{\mathcal H_{k,t;\bm \eta}}}\right]=\mathbb{E}\left[\langle s^2_{\bm x,\bm\sigma_0}\rangle_{\revsag{\mathcal H_{k,t;\bm \eta}}}\right].
\]
Similarly, $\mathbb{E}\left[\langle x_i\rangle^2_{\revsag{\mathcal H_{k,t;\bm \eta}}}\right]=\mathbb{E}\left[\sigma_{0i}\langle x_i\rangle_{\revsag{\mathcal H_{k,t;\bm \eta}}}\right]$, \revsag{implying}
\[
\frac{1}{p^2}\sum\limits_{i,j=1}^{p}\mathbb{E}\left[\langle x_i\rangle^2_{\revsag{\mathcal H_{k,t;\bm \eta}}}\right]\mathbb{E}\left[\langle x_j\rangle^2_{\revsag{\mathcal H_{k,t;\bm \eta}}}\right]=\mathbb{E}\left[\langle s_{\bm x,\bm \sigma_0}\rangle_{\revsag{\mathcal H_{k,t;\bm \eta}}}\right]^2,
\]
Further
\[
\mathbb{E}\left[\langle x_ix_j\rangle\langle x_i\rangle\langle x_j\rangle_{\revsag{\mathcal H_{k,t;\bm \eta}}}\right]=\mathbb{E}\left[\sigma_{0i}\sigma_{0j}\langle x_i\rangle_{\revsag{\mathcal H_{k,t;\bm \eta}}}\langle x_j\rangle_{\revsag{\mathcal H_{k,t;\bm \eta}}}\right],
\]
implies
\[
\frac{1}{p^2}\sum\limits_{i,j=1}^{p}\mathbb{E}\left[\langle x_ix_j\rangle_{\revsag{\mathcal H_{k,t;\bm \eta}}}\langle x_i\rangle_{\revsag{\mathcal H_{k,t;\bm \eta}}}\langle x_j\rangle_{\revsag{\mathcal H_{k,t;\bm \eta}}}\right]=\mathbb{E}\left[\langle s_{\bm x,\bm \sigma_0}\rangle^2_{\revsag{\mathcal H_{k,t;\bm \eta}}}\right].
\]
Therefore,
\begin{align}
\mathbb{E}\left[\langle (\mathcal{L}-\mathbb{E}[\mathcal{L}])^2\rangle_{\revsag{\mathcal H_{k,t;\bm \eta}}}\right]&=\frac{1}{4}\left(\mathbb{E}\left[\langle s^2_{\bm x, \bm \sigma_0}\rangle_{\revsag{\mathcal H_{k,t;\bm \eta}}}\right]-\mathbb{E}\left[\langle s_{\bm x, \bm \sigma_0}\rangle_{\revsag{\mathcal H_{k,t;\bm \eta}}}\right]^2\right)\\
&+\frac{1}{2}\left(\mathbb{E}\left[\langle s^2_{\bm x, \bm \sigma_0}\rangle_{\revsag{\mathcal H_{k,t;\bm \eta}}}\right]-\mathbb{E}[\langle s_{\bm x, \bm \sigma_0}\rangle^2_{\revsag{\mathcal H_{k,t;\bm \eta}}}]\right)+\frac{1}{4p\widetilde{\eta}}\mathbb{E}\left[\Sigma^2\right].
\end{align}
This implies,
\[
\mathbb{E}\left[\left\langle\left(s_{\bm x,\bm \sigma_0}-\mathbb{E}\left[\left\langle s_{\bm x,\bm \sigma_0}\right\rangle_{\revsag{\mathcal H_{k,t;\bm \eta}}}\right]\right)^2\right\rangle_{\revsag{\mathcal H_{k,t;\bm \eta}}}\right] \le 4 \mathbb{E}\left[\left\langle\left(\mathcal{L}-\mathbb{E}\left[\left\langle \mathcal{L}\right\rangle_{\revsag{\mathcal H_{k,t;\bm \eta}}}\right]\right)^2\right\rangle_{\revsag{\mathcal H_{k,t;\bm \eta}}}\right].
\]
By Fubini's Theorem and Proposition \ref{eq:prob_mathcal_l}, 
\begin{align}
&\int\limits_{a^{(1)}_p}^{b^{(1)}_p}\int\limits_{a^{(2)}_p}^{b^{(2)}_p}\lt\{\frac{1}{K_p}\sum\limits_{k=1}^{K_p}\int\limits_{0}^{1}dt\,\mathbb{E}\lt[\lt\langle\lt(s_{\bm x,\bm \sigma_0}-\mathbb{E}\lt[\lt\langle s_{\bm x, \bm \sigma_0}\rt\rangle_{\sr}\rt]\rt)^2\rt\rangle_\sr\rt]\rt\}d\eta_2d\eta_1 \\
& \le \frac{4}{K_p}\sum\limits_{k=1}^{K_p}\int\limits_{0}^{1}\lt\{\int\limits_{a^{(1)}_p}^{b^{(1)}_p}\int\limits_{a^{(2)}_p}^{b^{(2)}_p}\mathbb{E}\lt[\lt\langle \lt(\mathcal{L}-\mathbb{E}\lt[\lt\langle \mathcal{L}\rt\rangle_\sr\rt]\rt)^2\rt\rangle_\sr\rt]d\eta_2d\eta_1\rt\}\,dt \le \frac{4C}{\left(a^{(1)}_p\right)^\alpha p^\alpha},
\end{align}
for all $0<\alpha<1/4$. This completes the proof of Lemma \ref{lem:x_overlap}.
\qed

Now, we turn to the proof of Proposition \ref{eq:prob_mathcal_l}.
\paragraph{Proof of Proposition \ref{eq:prob_mathcal_l}.}
We divide the proof into two parts.
\begin{align}
    \mathbb{E}\lt[\lt\langle\lt(\mathcal{L}-\mathbb{E}\lt[\lt\langle\mathcal{L}\rt\rangle_\sr\rt]\rt)^2\rt\rangle_\sr\rt]&=\mathbb{E}\lt[\lt\langle\lt(\mathcal{L}-\lt\langle\mathcal{L}\rt\rangle_\sr\rt)^2\rt\rangle_\sr\rt]\\
    &\quad+\mathbb{E}\lt[\lt(\lt\langle\mathcal{L}\rt\rangle_\sr-\mathbb{E}\lt[\lt\langle\mathcal{L}\rt\rangle_\sr\rt]\rt)^2\rt].
\end{align}
We shall show the following two lemmas.
\begin{lem}
\label{lem:lem_5}
\revsagr{Let us consider a sequence of interpolation parameters $\{q_k(\bm \eta)\}_{k=1}^{K}$, such that as a function of $\bm \eta$ it is bounded and differentiable. Furthermore, for all $\eta_2>0$, the function $q_k(\cdot,\eta_2)$ is non-decreasing in $\eta_1$. Then for any sequence $0<a^{(1)}_p<b^{(1)}_p<1$ and $0<a^{(2)}_p<b^{(2)}_p<1$, we have}
\[
\int\limits_{a^{(1)}_p}^{b^{(1)}_p}\int\limits_{a^{(2)}_p}^{b^{(2)}_p} \mathbb{E}\lt[\lt\langle\lt(\mathcal{L}-\lt\langle\mathcal{L}\rt\rangle_\sr\rt)^2\rt\rangle_\sr\rt]d\eta_2d\eta_1 \le \frac{\mathbb{E}[\Sigma^2]}{p}\left(1+\frac{|\log a^{(1)}_p|}{4}\right).
\]
\end{lem}
\begin{lem}
\label{lem:lem_6}
\revsagr{Let us consider a sequence of interpolation parameters $\{q_k(\bm \eta)\}_{k=1}^{K}$, such that as a function of $\bm \eta$ it is bounded and differentiable. Furthermore, for all $\eta_2>0$, the function $q_k(\cdot,\eta_2)$ is non-decreasing in $\eta_1$. Then for any sequence $0<a^{(1)}_p<b^{(1)}_p<1$ and $0<a^{(2)}_p<b^{(2)}_p<1$, we have}
\[
\int\limits_{a^{(1)}_p}^{b^{(1)}_p}\int\limits_{a^{(2)}_p}^{b^{(2)}_p} \mathbb{E}\lt[\lt(\lt\langle\mathcal{L}\rt\rangle_\sr-\mathbb{E}\lt[\lt\langle\mathcal{L}\rt\rangle_\sr\rt]\rt)^2\rt]d\eta_2d\eta_1 \le \frac{C}{(a^{(1)}_p)^2p^{1/4}},
\]
for a constant $C>0$.
\end{lem}
From these two lemmas, Proposition \ref{eq:prob_mathcal_l} immediately follows. Next, we want to prove Lemmas \ref{lem:lem_5} and \ref{lem:lem_6}. 
\paragraph{Proofs of Lemma \ref{lem:lem_5}.}
Let us define 
\[
F_{k,t;\bm \eta}:=-\frac{1}{p}\log\left[\int \left\{\prod\limits_{i=1}^{p}Q(x_i)P(\beta_i|x_i)\right\}e^{\revsagr{-\mathcal{H}_{k,t;\bm \eta}(\bm x,\bm \beta;\bm \sigma_0,\bm \beta_0,\Theta)}}d\bm\beta\, d\bm x\right]
\]
\revsagr{Now, as mentioned earlier, $\bm \eta$ can be re-parametrized as $\bm \eta=\widetilde g(\widetilde \eta,\eta_2)$, where $\widetilde g$ is a differentiable function of $\widetilde \eta$ and $\eta_2$.} By simple calculation, we can show,
\[
\frac{\partial F_{k,t;\bm \eta}}{\partial\widetilde{\eta}}=\langle\mathcal{L}\rangle_\sr,
\]
and
\[
\frac{1}{p}\frac{\partial^2F_{k,t;\bm \eta}}{\partial\widetilde{\eta}^2}=-\lt(\lt\langle\mathcal{L}^2\rt\rangle_\sr-\lt\langle\mathcal{L}\rt\rangle^2_\sr\rt)+\frac{1}{4p^2\widetilde{\eta}^{3/2}}\sum\limits_{i=1}^{p}\lt\langle x_i\rt\rangle_\sr \widehat{Z}_i.
\]
By Gaussian Integration by parts, we get,
\begin{align}
\label{eq:R}
\frac{\partial f_{k,t;\bm \eta}}{\partial\widetilde{\eta}}&=\mathbb{E}\lt[\lt\langle\mathcal{L}\rt\rangle_\sr\rt]=-\frac{1}{2p}\sum\limits_{i=1}^{p}\mathbb{E}\lt[\lt\langle x_i\rt\rangle^2_\sr\rt]
\end{align}
\begin{align}
\label{eq:S}
\frac{1}{p}\frac{\partial^2f_{k,t;\bm \eta}}{\partial\widetilde{\eta}^2}&=-\mathbb{E}\lt[\lt\langle\mathcal{L}^2\rt\rangle_\sr-\lt\langle\mathcal{L}\rt\rangle^2_\sr\rt]+\frac{1}{4p^2\widetilde{\eta}}\sum\limits_{i=1}^{p}\mathbb{E}\lt[\lt\langle x^2_i\rt\rangle_\sr-\lt\langle x_i\rt\rangle^2_\sr\rt],
\end{align}
\revsag{where $f_{k,t;\bm \eta}$ is defined in \eqref{eq:free_energy_perturbed}.} We can also differentiate \eqref{eq:R} to get,
\begin{align}
\label{eq:convexity_interpolates}
\frac{1}{p}\frac{\partial^2f_{k,t;\bm \eta}}{\partial\widetilde{\eta}^2}&=\frac{1}{2p}\sum\limits_{i=1}^{p}\mathbb{E}\lt[2\langle x_i\rangle_\sr\langle x_i\mathcal{L}\rangle_\sr-2\langle x_i\rangle^2_\sr\langle\mathcal{L}\rangle_\sr\rt]\\
&=-\frac{1}{2p^2}\sum\limits_{i,j=1}^{p}\mathbb{E}\lt[\lt(\langle x_ix_j\rangle_\sr-\langle x_i\rangle_\sr\langle x_j\rangle_\sr\rt)^2\rt].
\end{align}
Hence $f_{k,t;\bm \eta}$ is concave in $\widetilde{\eta}$. From \eqref{eq:S}, we have,
\begin{align}
\mathbb{E}\lt[\lt\langle\lt(\mathcal{L}-\langle\mathcal{L}\rangle_\sr\rt)^2\rt\rangle_\sr\rt]&=-\frac{1}{p}\frac{\partial^2f_{k,t;\bm \eta}}{\partial\widetilde{\eta}^2}+\frac{1}{\revsagr{4}p^2\widetilde{\eta}}\sum\limits_{i=1}^{p}\mathbb{E}\lt[\langle x^2_i\rangle_\sr-\langle x_i\rangle^2_\sr\rt]\\
& \le -\frac{1}{p}\frac{\partial^2f_{k,t;\bm \eta}}{\partial\widetilde{\eta}^2}+\frac{\mathbb{E}[\Sigma^2]}{4p\widetilde{\eta}} \quad\quad \mbox{[by the Nishimori Identity]}\\
& \le -\frac{1}{p}\frac{\partial^2f_{k,t;\bm \eta}}{\partial\widetilde{\eta}^2}+\frac{\mathbb{E}[\Sigma^2]}{4p\eta_1},
\end{align}
where the final inequality follows as $\widetilde{\eta} \ge \eta_1$. Now we integrate both sides of the above display to get,
\begin{align}
\int\limits_{a^{(1)}_p}^{b^{(1)}_p}\int\limits_{a^{(2)}_p}^{b^{(2)}_p}\mathbb{E}\lt[\langle\lt(\mathcal{L}-\langle\mathcal{L}\rangle_\sr\rt)^2\rangle_\sr\rt]d\eta_2d\eta_1& \overset{(1)}{\le} -\frac{1}{p}\int\limits_{a^{(2)}_p}^{b^{(2)}_p}\left(\int\limits_{\widetilde{\eta}({a}^{(1)}_p,\,\eta_2)}^{\widetilde{\eta}({b}^{(1)}_p,\,\eta_2)}\frac{d\widetilde{\eta}}{J_1}\frac{\partial^2f_{k,t;\bm \eta}}{\partial\widetilde{\eta}^2}\right)d\eta_2\\
&\hskip 5em+\frac{\mathbb{E}[\Sigma^2]}{4p}\int\limits_{a^{(2)}_p}^{b^{(2)}_p}\int\limits_{a^{(1)}_p}^{b^{(1)}_p}\frac{d\eta_1}{\eta_1}d\eta_2\\
& \overset{(2)}{\le} \frac{1}{p}\int\limits_{a^{(2)}_p}^{b^{(2)}_p}\left[\frac{\partial f_{k,t;\bm \eta}}{\partial\widetilde{\eta}}\Big|_{\widetilde{\eta}({a}^{(1)}_p,\,\eta_2)}-\frac{\partial f_{k,t;\bm \eta}}{\partial\widetilde{\eta}}\Big|_{\widetilde{\eta}({b}^{(1)}_p,\,\eta_2)}\right]d\eta_2\\
&\quad\quad+\frac{\mathbb{E}[\Sigma^2]}{4p}(b^{(2)}_p-a^{(2)}_p)[\log b^{(1)}_p-\log a^{(1)}_p]\\
& \overset{(3)}{\le} \frac{\mathbb{E}[\Sigma^2]}{4p}[\log b^{(1)}_p-\log a^{(1)}_p] \\
& \hskip 6em \mbox{[as $a^{(2)}_p,b^{(2)}_p\le 1$ and \eqref{eq:convexity_interpolates}]}\\
& \le \frac{\mathbb{E}[\Sigma^2]}{p}\left[1+\frac{|\log a^{(1)}_p|}{4}\right],
\end{align}
where the inequality $(1)$ follows by the change of variable $(\eta_1,\eta_2) \mapsto (\widetilde{\eta},\eta_2)$ where the Jacobian $|J_1|=\partial \widetilde{\eta}/\partial \eta_1 \ge 1$. \revsagr{Furthermore, observe that by \eqref{eq:convexity_interpolates}, for all $a^{(2)}_p<\eta_2<b^{(p)}_2$, the integrand 
\[
\left[\frac{\partial f_{k,t;\bm \eta}}{\partial\widetilde{\eta}}\Big|_{\widetilde{\eta}({a}^{(1)}_p,\,\eta_2)}-\frac{\partial f_{k,t;\bm \eta}}{\partial\widetilde{\eta}}\Big|_{\widetilde{\eta}({b}^{(1)}_p,\,\eta_2)}\right] \le 0.
\]
Hence, the inequality (3) follows.}
\paragraph{Proof of Lemma \ref{lem:lem_6}} \revsagr{By assumptions of the lemma, we have
\begin{align}
\label{eq:def_b_m}
P(|B| \le s_{\max}\,|\,\Sigma=i)=1, \quad \mbox{for $i \in \{0,1\}$.}
\end{align}}
Consider the functions,
\[
\widetilde{F}(\widetilde{\eta})=F_{k,t;\bm \eta}+\frac{\sqrt{\widetilde{\eta}}}{p}\sum\limits_{i=1}^{p}s_{\max}|\widehat{Z}_i|,
\]
and
\begin{equation}
\label{eq:one_two_two}
\widetilde{f}(\widetilde{\eta})=f_{k,t;\bm \eta}+\frac{\sqrt{\widetilde{\eta}}}{p}\sum\limits_{i=1}^{p}s_{\max}\,\mathbb{E}|\widehat{Z}_i|.
\end{equation}
Note that both these functions are concave functions of $\widetilde{\eta}$. Hence for all $\delta>0$,
\begin{align}
\frac{\partial\widetilde{F}(\widetilde{\eta})}{\partial\widetilde{\eta}}-\frac{\partial\widetilde{f}(\widetilde{\eta})}{\partial\widetilde{\eta}} & \le \frac{\widetilde{F}(\widetilde{\eta})-\widetilde{F}(\widetilde{\eta}-\delta)}{\delta}-\frac{\partial\widetilde{f}(\widetilde{\eta})}{\partial\widetilde{\eta}}\\
& \le \frac{\widetilde{F}(\widetilde{\eta})-\widetilde{f}(\widetilde{\eta})}{\delta}-\frac{\widetilde{F}(\widetilde{\eta}-\delta)-\widetilde{f}(\widetilde{\eta}-\delta)}{\delta}\\
& \quad \quad +\frac{\partial\widetilde{f}(\widetilde{\eta}-\delta)}{\partial\widetilde{\eta}}-\frac{\partial\widetilde{f}(\widetilde{\eta})}{\partial\widetilde{\eta}}.
\end{align}
Similarly,
\begin{align}
\frac{\partial\widetilde{F}(\widetilde{\eta})}{\partial\widetilde{\eta}}-\frac{\partial\widetilde{f}(\widetilde{\eta})}{\partial\widetilde{\eta}} & \ge \frac{\widetilde{F}(\widetilde{\eta}+\delta)-\widetilde{f}(\widetilde{\eta}+\delta)}{\delta}-\frac{\widetilde{F}(\widetilde{\eta})-\widetilde{f}(\widetilde{\eta})}{\delta}\\
& \quad \quad +\frac{\partial\widetilde{f}(\widetilde{\eta}+\delta)}{\partial\widetilde{\eta}}-\frac{\partial\widetilde{f}(\widetilde{\eta})}{\partial\widetilde{\eta}}.
\end{align}
Let $-C^{-}(\widetilde{\eta})=\frac{\partial\widetilde{f}(\widetilde{\eta}+\delta)}{\partial\widetilde{\eta}}-\frac{\partial\widetilde{f}(\widetilde{\eta})}{\partial\widetilde{\eta}} \le 0$ and $\revsagr{C^{+}(\widetilde{\eta})=\frac{\partial\widetilde{f}(\widetilde{\eta}-\delta)}{\partial\widetilde{\eta}}-\frac{\partial\widetilde{f}(\widetilde{\eta})}{\partial\widetilde{\eta}} \ge 0}$. Combining these we get,
\begin{align}
\label{eq:T}
 \frac{\widetilde{F}(\widetilde{\eta}+\delta)-\widetilde{f}(\widetilde{\eta}+\delta)}{\delta}-\frac{\widetilde{F}(\widetilde{\eta})-\widetilde{f}(\widetilde{\eta})}{\delta}-C^{-}(\widetilde{\eta}) &\le \frac{\partial\widetilde{F}(\widetilde{\eta})}{\partial\widetilde{\eta}}-\frac{\partial\widetilde{f}(\widetilde{\eta})}{\partial\widetilde{\eta}}\\
 & \le  \frac{\widetilde{F}(\widetilde{\eta})-\widetilde{f}(\widetilde{\eta})}{\delta}-\frac{\widetilde{F}(\widetilde{\eta}-\delta)-\widetilde{f}(\widetilde{\eta}-\delta)}{\delta}+C^{+}(\widetilde{\eta}).
\end{align}
Now,
\[
\widetilde{F}(\widetilde{\eta})-\widetilde{f}(\widetilde{\eta})=F_{k,t;\bm \eta}-f_{k,t;\bm \eta}+\sqrt{\widetilde{\eta}}s_{\max}A,
\]
\revsagr{where $s_{\max}>0$ is defined in \eqref{eq:def_b_m}} and $A=(1/p)\sum_{i=1}^{p}(|\widehat{Z}_i|-\mathbb{E}|\widehat{Z}_i|)$. Using \eqref{eq:R}, we get,
\begin{align}
\label{eq:ref_properly}
    \frac{\partial\widetilde{F}(\widetilde{\eta})}{\partial\widetilde{\eta}}-\frac{\partial\widetilde{f}(\widetilde{\eta})}{\partial\widetilde{\eta}}=\langle\mathcal{L}\rangle_\sr-\mathbb{E}\lt[\langle\mathcal{L}\rangle_\sr\rt]+\frac{s_{\max}}{2\sqrt{\widetilde{\eta}}}A, 
\end{align}
\revsagr{where $s_{\max}>0$ is defined in \eqref{eq:def_b_m}.} From \eqref{eq:one_two_two}  and \eqref{eq:ref_properly}, it is easy to observe that \eqref{eq:T} implies the following:
\begin{align}
\label{eq:star_3}
\lt|\langle \mathcal{L}\rangle_\sr-\mathbb{E}\lt[\langle\mathcal{L}\rangle_\sr\rt]\rt| & \le \delta^{-1}\sum\limits_{u_1 \in \{\widetilde{\eta}-\delta,\widetilde{\eta},\widetilde{\eta}+\delta\}}\left(|F_{k,t;\bm u}-f_{k,t;\bm u}|+s_{\max}|A|\sqrt{u_1}\right)\\
& \quad \quad +C^{+}(\widetilde{\eta})+C^{-}(\widetilde{\eta})+\frac{s_{\max}}{2\sqrt{\widetilde{\eta}}}|A|.
\end{align}
We shall show,
\begin{equation}
\label{eq:div_f_kl}
\revsagr{\mathbb{E}[(F_{k,t;\bm \eta}-f_{k,t;\bm \eta})^2]=O(p^{-1}).}
\end{equation}
Since $\widetilde{\eta}\ge \eta_1$, taking expectation on both sides of \eqref{eq:star_3} and using $\mathbb{E}[A^2]=O(p^{-1})$, and \eqref{eq:div_f_kl} coupled with the Cauchy Schwartz inequality, we can show that
\revsagr{\begin{align}
\label{eq:dev_el}
\mathbb{E}\lt[\lt(\langle\mathcal{L}\rangle_\sr-\mathbb{E}\lt[\langle\mathcal{L}\rangle_\sr\rt]\rt)^2\rt] &\le \delta^{-2}O(p^{-1}+\delta^{-2}s_{\max}^{2})(\widetilde{\eta}+\delta)O(p^{-1})\nonumber\\
&\quad +2\,C^{+}(\widetilde{\eta})^2+2\,C^{-}(\widetilde{\eta})^2+\frac{s_{\max}^2}{4\widetilde\eta}O(p^{-1}).
\end{align}}
By the change of variable $(\eta_1,\eta_2)\mapsto (\widetilde{\eta},\eta_2)$, with the determinant of the Jacobian greater than or equal to $1$, we have the following inequalities.
\begin{align}
\label{eq:bound_shift_1}
\int\limits_{a^{(1)}_p}^{b^{(1)}_p}\int\limits_{a^{(2)}_p}^{b^{(2)}_p}(C^{+}(\widetilde{\eta})^2+C^{-}(\widetilde{\eta})^2)d\eta_2d\eta_1 & = \int\limits_{a^{(2)}_p}^{b^{(2)}_p}\left[\int\limits_{\widetilde{\eta}(a^{(1)}_p,\,\eta_2)}^{\widetilde{\eta}(b^{(1)}_p,\,\eta_2)}\frac{d\widetilde{\eta}}{J_1}(C^{+}(\widetilde{\eta})^2+C^{-}(\widetilde{\eta})^2)\right]d\eta_2\\
& \overset{(1)}{\le} 2\left(\mathbb{E}[\Sigma^2]+\frac{s_{\max}}{\sqrt{\widetilde{\eta}}}\right)\int\limits_{a^{(2)}_p}^{b^{(2)}_p}\left[\int\limits_{\widetilde{\eta}(a^{(1)}_p,\,\eta_2)}^{\widetilde{\eta}(b^{(1)}_p,\,\eta_2)}(C^{+}(\widetilde{\eta})+C^{-}(\widetilde{\eta}))d\widetilde{\eta}\right]d\eta_2\\
& \quad \quad \mbox{\Bigg[as $\left|\frac{\partial\widetilde{f}(\widetilde{\eta})}{\partial\widetilde{\eta}}\right|\le \frac{1}{2}\left(\mathbb{E}[\Sigma^2]+\frac{s_{\max}}{\sqrt{\widetilde{\eta}}}\right)$ \revsagr{by \eqref{eq:R} and \eqref{eq:one_two_two}.}\Bigg]}\\
& \le 2\left(\mathbb{E}[\Sigma^2]+\frac{s_{\max}}{\sqrt{\widetilde{\eta}}}\right)\int\limits_{a^{(2)}_p}^{b^{(2)}_p}\Bigg[(\widetilde{f}(\widetilde{\eta}(b^{(1)}_p,\,\eta_2)-\delta)\\
&\hskip 15em-\widetilde{f}(\widetilde{\eta}(b^{(1)}_p,\,\eta_2)+\delta))\\
&\hskip 6em+(\widetilde{f}(\widetilde{\eta}(a^{(1)}_p,\,\eta_2)+\delta)-\widetilde{f}(\widetilde{\eta}(a^{(1)}_p,\,\eta_2)-\delta))\Bigg]d\eta_2\\
& \le 4\delta(b^{(2)}_p-a^{(2)}_p)\left[\mathbb{E}[\Sigma^2]+\frac{s_{\max}}{\sqrt{\widetilde{\eta}(a^{(1)}_p,a^{(2)}_p)-\delta}}\right]^2\\
& \le 4\delta\left[\mathbb{E}[\Sigma^2]+\frac{s_{\max}}{\sqrt{\widetilde{\eta}(a^{(1)}_p,a^{(2)}_p)-\delta}}\right]^2\\
&\le 4\delta\left(\mathbb{E}[\Sigma^2]+\frac{s_{\max}}{\sqrt{a^{(1)}_p-\delta}}\right)^2.
\end{align}
\revsagr{As $C^+,C^- \ge 0$, we have
\begin{align}
(C^{+}(\widetilde{\eta})^2+C^{-}(\widetilde{\eta})^2) &\le (C^{+}(\widetilde{\eta})+C^{-}(\widetilde{\eta}))^2\\
& \le \left[\left|\frac{\partial\widetilde{f}(\widetilde{\eta}+\delta)}{\partial\widetilde{\eta}}\right|+2\,\left|\frac{\partial\widetilde{f}(\widetilde{\eta})}{\partial\widetilde{\eta}}\right|+\left|\frac{\partial\widetilde{f}(\widetilde{\eta}-\delta)}{\partial\widetilde{\eta}}\right|\right](C^{+}(\widetilde{\eta})+C^{-}(\widetilde{\eta}))\\
& \le 2\left(\mathbb{E}[\Sigma^2]+\frac{s_{\max}}{\sqrt{\widetilde{\eta}}}\right)(C^{+}(\widetilde{\eta})+C^{-}(\widetilde{\eta})) 
\end{align}
where the last inequality follows from the Nishimori Identity, property of Gaussian variables, \eqref{eq:R} and \eqref{eq:one_two_two}. This implies inequality (1) of \eqref{eq:bound_shift_1}.}
Plugging in \eqref{eq:dev_el}, we get that 
\revsagr{
\begin{align}
&\int\limits_{a^{(1)}_p}^{b^{(1)}_p}\int\limits_{a^{(2)}_p}^{b^{(2)}_p}\mathbb{E}\lt[\lt(\langle\mathcal{L}\rangle_\sr-\mathbb{E}\lt[\langle\mathcal{L}\rangle_\sr\rt]\rt)^2\rt]d\eta_2d\eta_1\\
&\le \delta^{-2}O(p^{-1})+\delta^{-2}s^2_{\max}(B+\delta)O(p^{-1})\\
& +\frac{s^2_{\max}}{4}|\log a^{(1)}_p|O(p^{-1})+\revsagr{4\delta}\left(\mathbb{E}[\Sigma^2]+\frac{s_{\max}}{\sqrt{a^{(1)}_p-\delta}}\right)^2,
\end{align}}
for a constant $B\ge \widetilde{\eta}$, because $\widetilde{\eta}$ is bounded as $q_k$'s are bounded and $\eta_1 \le 1$. \revsagr{Then setting $\delta=a^{(1)}_pp^{-1/4}$, we get a constant $C>0$ such that
\[
\int\limits_{a^{(1)}_p}^{b^{(1)}_p}\int\limits_{a^{(2)}_p}^{b^{(2)}_p}\mathbb{E}\lt[\lt(\langle\mathcal{L}\rangle_\sr-\mathbb{E}[\langle\mathcal{L}\rangle_\sr\rt]\rt)^2]d\eta_2d\eta_1 \le C (a^{(1)}_p)^{-2}p^{-1/4}.
\]}
Finally, to show \eqref{eq:div_f_kl} we consider the following two lemmas.
\revsagr{
\begin{lem}
    \label{lem:second_lem}
    For $\Theta_1=(\{\bm \varepsilon_\ell\}_{\ell=1}^{K},\{\bm Z_\ell\}_{\ell=1}^{K},\{\widebar{\bm \varepsilon}_\ell\}_{\ell=1}^{K},\{\widebar{\bm Z}_\ell\}_{\ell=1}^{K},\{\widehat{Z}_i\}_{i=1}^{p},\{\widehat{\widehat{Z}}_i\}_{i=1}^{p})$, we have
    \[
    \mathbb{E}[(F_{k,t;\bm \eta}-\mathbb E_{\bm \Theta_1}[F_{k,t;\bm \eta}|\bm \Phi,\bm \beta_0,\bm \sigma_0])^2]=O(p^{-1}).
    \]
    The random variables contained in $\bm \Theta_1$ are defined in \eqref{eq:def_big_bold_theta}.    
\end{lem}
\begin{lem}
    \label{lem:first_lem}
    For $\Theta_1$ defined in Lemma \ref{lem:second_lem}, we have
    \[
    \mathbb{E}[(\mathbb E_{\bm \Theta_1}[F_{k,t;\bm \eta}|\bm \Phi,\bm \beta_0,\bm \sigma_0]-f_{k,t;\bm \eta})^2]=O(p^{-1}).
    \]
    Here $f_{k,t;\bm \eta}$ is defined in \eqref{eq:free_energy_perturbed}.    
\end{lem}}

From the above lemmas, \eqref{eq:div_f_kl} follows immediately. Now, we shall prove Lemma \ref{lem:second_lem}. The proof of Lemma \ref{lem:first_lem} follows using similar techniques (see, \citet{doi:10.1073/pnas.1802705116}).

\begin{proof}[Proof of Lemma \ref{lem:second_lem}]
    \revsagr{To prove the above lemma we shall use the Gaussian Poincar\'e inequality from \citet{doi:10.1073/pnas.1802705116}: Let $U=(U_1,...,U_n)$ be a vector of $n$ independent standard normal random variables. Let $g:\R^n\rightarrow \R$ be a continuously differentiable function. Then 
    \begin{align}
    \label{eq:poincr_ineq}
    \mbox{Var}(g(U)) \le \E[\|\nabla g(U)\|^2].
    \end{align}
    Observe that conditioned on $\bm \Phi,\bm \beta_0$ and $\bm \sigma_0$, $F_{k,t;\bm \eta}$ is a continuous and differential function of the Gaussian variables contained in $\bm \Theta_1$. We shall bound the expectation of 
    \begin{align}
    \sum_{\ell=1}^{K}\sum_{\substack{i \le j\\i,j=1}}^p\left|\frac{\partial F_{k,t;\bm \eta}}{\partial (\bm Z_\ell)_{ij}}\right|^2+\sum_{\ell=1}^{K}\sum_{\mu=1}^{n}\left|\frac{\partial F_{k,t;\bm \eta}}{\partial (\bm \varepsilon_\ell)_\mu}\right|^2
    +\sum_{\ell=1}^{K}\sum_{j=1}^{p}\left|\frac{\partial F_{k,t;\bm \eta}}{\partial (\widebar{\bm \varepsilon}_\ell)_j}\right|^2\\+\sum_{\ell=1}^{K}\sum_{j=1}^{p}\left|\frac{\partial F_{k,t;\bm \eta}}{\partial (\widebar{\bm Z}_\ell)_j}\right|^2+\sum_{j=1}^{p}\left|\frac{\partial F_{k,t;\bm \eta}}{\partial \widehat{Z}_j}\right|^2+\sum_{j=1}^{p}\left|\frac{\partial F_{k,t;\bm \eta}}{\partial \widehat{\widehat{Z}}_j}\right|^2.
    \end{align}
    Now, observe that from \eqref{eq:h_on_top}
    \begin{align}
    \mathbb E_{\Theta_1}\left[\sum_{\ell=1}^{K}\sum_{\substack{i \le j\\i,j=1}}^p\left|\frac{\partial F_{k,t;\bm \eta}}{\partial (\bm Z_\ell)_{ij}}\right|^2\Bigg|(\bm \Phi,\bm \sigma_0,\bm \beta_0)\right]&=\frac{\lambda}{p^3K}\sum_{\ell=k+1}^{K}\sum_{\substack{i \le j\\i,j=1}}^p \left[\mathbb E_{\Theta_1}\lt[\langle x_ix_j\rangle^2_{\mathcal H_{\ell,t;\bm \eta}}\Bigg|(\bm \Phi,\bm \sigma_0,\bm \beta_0)\rt] \right.\\
    &\left.+\frac{\lambda(1-t)}{p^3K}\mathbb E_{\Theta_1}\lt[\langle x_ix_j\rangle^2_{\mathcal H_{k,t;\bm \eta}}\Bigg|(\bm \Phi,\bm \sigma_0,\bm \beta_0)\rt] \right]= O(p^{-1}).
    \end{align}
    Similarly, we can show that
    \begin{align}
    \mathbb E_{\Theta_1}\left[\sum_{\ell=1}^{K}\sum_{i=1}^p\left|\frac{\partial F_{k,t;\bm \eta}}{\partial (\widebar{\bm Z}_\ell)_{i}}\right|^2\Bigg|(\bm \Phi,\bm \sigma_0,\bm \beta_0)\right]&=\sum_{\ell=1}^{k-1}\frac{\lambda q_\ell}{p^2K}\sum_{i=1}^p\left[\mathbb E_{\Theta_1}\lt[\langle x_i\rangle^2_{\mathcal H_{\ell,t;\bm \eta}}\Bigg|(\bm \Phi,\bm \sigma_0,\bm \beta_0)\rt]\right.\\
    &\left.+\frac{\lambda tq_k}{p^2K}\mathbb E_{\Theta_1}\lt[\langle x_i\rangle^2_{\mathcal H_{k,t;\bm \eta}}\Bigg|(\bm \Phi,\bm \sigma_0,\bm \beta_0)\rt]\right]= O(p^{-1}),
    \end{align}
    and
    \begin{align}
    \mathbb E_{\Theta_1}\left[\sum_{j=1}^p\left|\frac{\partial F_{k,t;\bm \eta}}{\partial \widehat{Z}_j}\right|^2\Bigg|(\bm \Phi,\bm \sigma_0,\bm \beta_0)\right]&=\frac{\eta_1}{p^2}\sum_{i=1}^p\mathbb E_{\Theta_1}\lt[\langle x_i\rangle^2_{\mathcal H_{k,t;\bm \eta}}\Bigg|(\bm \Phi,\bm \sigma_0,\bm \beta_0)\rt]= O(p^{-1}).
    \end{align}
    As $P(\,\cdot\,|\,\Sigma\,)$ is discretely supported with bounded support for $\Sigma \in \{0,1\}$, we can also show that
    \begin{align}
    \mathbb E_{\Theta_1}\left[\sum_{\ell=1}^{K}\sum_{i=1}^p\left|\frac{\partial F_{k,t;\bm \eta}}{\partial (\widebar{\bm \varepsilon}_\ell)_{i}}\right|^2\Bigg|(\bm \Phi,\bm \sigma_0,\bm \beta_0)\right]&=\frac{1}{p^2K\Psi^2_\ell}\sum_{\ell=k+1}^{K}\sum_{i=1}^p\left[\mathbb E_{\Theta_1}\lt[\langle (\beta_i-\beta_{0i})\rangle^2_{\mathcal H_{\ell,t;\bm \eta}}\Bigg|(\bm \Phi,\bm \sigma_0,\bm \beta_0)\rt]\right.\\
    &\left.+\frac{\lambda_k(t)}{p^2K}\mathbb E_{\Theta_1}\lt[\langle (\beta_i-\beta_{0i})\rangle^2_{\mathcal H_{k,t;\bm \eta}}\Bigg|(\bm \Phi,\bm \sigma_0,\bm \beta_0)\rt]\right]= O(p^{-1}),
    \end{align}
    and
    \begin{align}
    \mathbb E_{\Theta_1}\left[\sum_{j=1}^p\left|\frac{\partial F_{k,t;\bm \eta}}{\partial \widehat{\widehat{Z}}_j}\right|^2\Bigg|(\bm \Phi,\bm \sigma_0,\bm \beta_0)\right]&=\frac{\eta_2}{p^2}\sum_{i=1}^p\mathbb E_{\Theta_1}\lt[\langle (\beta_i-\beta_{0i})\rangle^2_{\mathcal H_{k,t;\bm \eta}}\Bigg|(\bm \Phi,\bm \sigma_0,\bm \beta_0)\rt]= O(p^{-1}).
    \end{align}
    Finally, we can get a constant $C_1>0$ independent of $p$, such that 
    \begin{align}
    \mathbb E_{\Theta_1}\left[\sum_{\ell=1}^{K}\sum_{\mu=1}^n\left|\frac{\partial F_{k,t;\bm \eta}}{\partial (\bm \varepsilon_\ell)_{\mu}}\right|^2\Bigg|(\bm \Phi,\bm \sigma_0,\bm \beta_0)\right]&=\frac{1}{p^2K\Delta}\sum_{\ell=k+1}^{K}\sum_{\mu=1}^n\left[\mathbb E\lt[\langle [\bm \Phi(\bm \beta-\bm\beta_{0})]_\mu\rangle^2_{\mathcal H_{\ell,t;\bm \eta}}\Bigg|(\bm \Phi,\bm \sigma_0,\bm \beta_0)\rt]\right.\\
    &\left.+\frac{\gamma_k(t)}{p^2K}\mathbb E\lt[\langle [\bm \Phi(\beta_i-\beta_{0i})]_\mu\rangle^2_{\mathcal H_{k,t;\bm \eta}}\Bigg|(\bm \Phi,\bm \sigma_0,\bm \beta_0)\rt]\right]\\
    &\le \frac{C_1}{p^2}\sum_{\mu=1}^{n}\lt(\sum_{i=1}^{p}\phi_{\mu i}\rt)^2.
    \end{align}
    The above relations along with \eqref{eq:poincr_ineq} imply
    \begin{align}
    \mathbb{E}\lt[(F_{k,t;\bm \eta}-\mathbb E_{\bm \Theta_1}[F_{k,t;\bm \eta}|\bm \Phi,\bm \beta_0,\bm \sigma_0])^2\Bigg|(\bm \Phi,\bm \sigma_0,\bm \beta_0)\rt]=\frac{C_1}{p^2}\sum_{\mu=1}^{n}\lt(\sum_{i=1}^{p}\phi_{\mu i}\rt)^2+O(p^{-1}).
    \end{align}
    Observing that $\phi_{\mu i} \sim \mathsf{N}(0,n^{-1})$ for all $(\mu,i) \in [n] \times [p]$, we can conclude that
    \[
    \mathbb{E}\lt[(F_{k,t;\bm \eta}-\mathbb E_{\bm \Theta_1}[F_{k,t;\bm \eta}|\bm \Phi,\bm \beta_0,\bm \sigma_0])^2\rt]=O(p^{-1}).
    \]}
\end{proof}

\subsection{Proof of Lemma \ref{lem:terminal_cond}}
We compute the partial derivatives of $f_{1,0;\bm \eta}$ with respect to $\eta_1$ and $\eta_2$. Let us observe that,
\[
\frac{\partial f_{1,0;\bm \eta}}{\partial \eta_1}=\frac{1}{p}\sum\limits_{i=1}^{p}\mathbb E\Bigg\{\frac{1}{2}\langle x^2_i\rangle_{\mathcal{H}_{1,0;\bm \eta}}-\langle x_i\rangle_{\mathcal{H}_{1,0;\bm \eta}}\sigma_{0i}-\frac{1}{2\sqrt{\eta_1}}\langle x_i\rangle_{\mathcal{H}_{1,0;\bm \eta}}\widehat{Z}_i\Bigg\}.
\]
Using Gaussian integration by parts and Nishimori Identity we get,
\[
\frac{\partial f_{1,0;\bm \eta}}{\partial \eta_1}=-\frac{1}{2p}\sum\limits_{i=1}^{p}\mathbb E[\langle x^2_i\rangle_{1,0;\bm \eta}].
\]
Next, observe that,
\[
\frac{\partial f_{1,0;\bm \eta}}{\partial \eta_2}=\frac{1}{p}\sum\limits_{i=1}^{p}\mathbb E\Bigg\{\frac{1}{2}\langle (\beta_i-\beta_{0i})^2\rangle_{\mathcal{H}_{1,0;\bm \eta}}-\frac{1}{2\sqrt{\eta_2}}\langle (\beta_i-\beta_{0i})\rangle_{\mathcal{H}_{1,0;\bm \eta}}\widehat{\widehat{Z}}_i\Bigg\}.
\]
Again, by similar argument we have,
\[
\frac{\partial f_{1,0;\bm \eta}}{\partial \eta_2}=-\frac{1}{2p}\sum\limits_{i=1}^{p}\mathbb E[\langle (\beta_i-\beta_{0i})^2\rangle_{\mathcal{H}_{1,0;\bm \eta}}].
\]
\revsagr{Since, $\mathbb E[\langle x^2_i\rangle_{\mathcal{H}_{1,0;\bm \eta}}]$ and $\mathbb E[\langle (\beta_i-\beta_{0i})^2\rangle_{\mathcal{H}_{1,0;\bm \eta}}]$ are uniformly bounded by the assumptions of the theorem}, we get a constant $C>0$ such that,
\[
\|\nabla_{\bm \eta}f_{1,0;\bm \eta}\| \le \frac{1}{2p}\sqrt{\Bigg(\sum\limits_{i=1}^{p}\mathbb E[\langle x^2_i\rangle_{\mathcal{H}_{1,0;\bm \eta}}]\Bigg)^2+\Bigg(\sum\limits_{i=1}^{p}\mathbb E[\langle (\beta_i-\beta_{0i})^2\rangle_{\mathcal{H}_{1,0;\bm \eta}}]\Bigg)^2} \le C.
\]
Therefore, by the Mean Value Theorem,
\[
|f_{1,0;\bm \eta}-f_{1,0;\bm 0}| \le C\,\|\bm \eta\|.
\]
The last inequality follows from the Lipschitz continuity of the free energy $f_{K,1;\bm \eta}$; of the decoupled scalar system. We refer the readers to \citet{5730572} for further explanation.
\subsection{Proof of Lemma \ref{lem:overlap_approx}}
Let us observe that,
\begin{align}
&\mathbb{E}[\langle s_{\bm \sigma_0,\bm x}\rangle_{\mathcal{H}_{k,t;\bm \eta}}]-\mathbb{E}[\langle s_{\bm \sigma_0,\bm x}\rangle_{\mathcal{H}_{k,0;\bm \eta}}]\\
& = \int\limits_0^t \left(\frac{d}{ds}\mathbb{E}[\langle s_{\bm \sigma_0,\bm x}\rangle_{\mathcal{H}_{k,s;\bm \eta}}]\right)\,ds\\
& = \int\limits_0^t \mathbb{E}\left[\langle s_{\bm \sigma_0,\bm x}\rangle_{\mathcal{H}_{k,s;\bm \eta}}\left\langle\frac{d\mathcal{H}_{k,s;\bm \eta}}{ds}\right\rangle_{\mathcal{H}_{k,s;\bm \eta}}-\left\langle s_{\bm \sigma_0,\bm x}\frac{d\mathcal{H}_{k,s;\bm \eta}}{ds}\right\rangle_{\mathcal{H}_{k,s;\bm \eta}}\right]\,ds\\
& = \int\limits_0^t \mathbb{E}\left[\langle s_{\bm \sigma_0,\bm x}\rangle_{\mathcal{H}_{k,s;\bm \eta}}\left\langle\left(\frac{d\mathcal{H}_{k,s;\bm \eta}}{ds}(\bm x^\prime,\bm \beta^\prime;\revsagr{\bm \sigma_0,\bm \beta_0},\bm \Theta)-\frac{d\mathcal{H}_{k,s;\bm \eta}}{ds}(\bm x,\bm \beta;\revsagr{\bm \sigma_0,\bm \beta_0},\bm \Theta)\right)\right\rangle_{\mathcal{H}_{k,s;\bm \eta}}\right]\,ds,
\end{align}
where $(\bm x,\bm \beta)$ and $(\bm x^\prime,\bm \beta^\prime),$  are i.i.d replicas. Let,
\begin{align}
g(\bm x,\bm x^\prime,\bm \beta,\bm \beta^\prime;\bm \sigma_0,\bm \beta_0)&=\lambda q_k\sum\limits_{i=1}^{p}\left(\frac{x_ix^\prime_i}{2}-x_i\sigma_{0i}\right)-\lambda\sum\limits_{i \le j=1}^{p}\left(\frac{x_ix_jx^\prime_ix^\prime_j}{2p}-\frac{x_ix_j\sigma_{0i}\sigma_{0j}}{p}\right)\\
& +\frac{d\gamma_k(t)}{dt}\frac{1}{2}\sum\limits_{\mu=1}^{n}\mathbb{E}\left[\langle [\bm \Phi(\bm \beta-\bm \beta_0)]_\mu\rangle_{\mathcal{H}_{k,s;\bm \eta}}\langle [\bm \Phi(\bm \beta^\prime-\bm \beta_0)]_\mu\rangle_{\mathcal{H}_{k,s;\bm \eta}}\right]\\
& +\frac{d\lambda_k(t)}{dt}\frac{1}{2}\sum\limits_{i=1}^{p}\mathbb{E}\left[\langle(\bm \beta_i-\bm \beta_{0i})\rangle_{\mathcal{H}_{k,s;\bm \eta}}\langle(\bm \beta^\prime_{i}-\bm \beta_{0i})\rangle_{\mathcal{H}_{k,s;\bm \eta}}\right].
\end{align}
Then observe that,
\begin{align}
&\mathbb{E}[\langle s_{\bm \sigma_0,\bm x}\rangle_{\mathcal{H}_{k,t;\bm \eta}}]-\mathbb{E}[\langle s_{\bm \sigma_0,\bm x}\rangle_{\mathcal{H}_{k,0;\bm \eta}}]\\
&=\frac{1}{K}\int\limits_{0}^{t}\mathbb{E}\Bigg[\left\langle s_{\bm \sigma_0,\bm x}(g(\bm x^\prime,\bm x^{\prime\prime},\bm \beta^\prime,\bm \beta^{\prime\prime};\bm \sigma_0,\bm \beta_0)-g(\bm x,\bm x^\prime,\bm \beta,\bm \beta^\prime;\bm \sigma_0,\bm \beta_0))\right\rangle_{\mathcal{H}_{k,s;\bm \eta}}\Bigg]\,ds.
\end{align}
\revsagr{Here, $(\bm x^{\prime\prime},\bm \beta^{\prime\prime})$ is an i.i.d replica of $(\bm x, \bm \beta)$ under the Gibbs measure $P_{k,s;\bm \eta}$.}
Using the Cauchy-Schwartz Inequality, we get
\begin{align}
&\Bigg|\mathbb{E}[\langle s_{\bm \sigma_0,\bm x}\rangle_{\mathcal{H}_{k,t;\bm \eta}}]-\mathbb{E}[\langle s_{\bm \sigma_0,\bm x}\rangle_{\mathcal{H}_{k,0;\bm \eta}}]\Bigg|\\
&=O\left(\frac{1}{K}\sqrt{\mathbb{E}[\langle s^2_{\bm \sigma_0,\bm x}\rangle_{\mathcal{H}_{k,s;\bm \eta}}]\mathbb{E}[\langle g^2(\bm x,\bm x^\prime,\bm \beta,\bm \beta^\prime;\bm \sigma_0,\bm \beta_0)\rangle_{\mathcal{H}_{k,s;\bm \eta}}]}\right)
\end{align}
\revsagr{Since $\bm \beta$ and $\bm x$ have bounded support under the Gibbs measure $P_{k,s;\bm \eta}$ and $(\bm \sigma_0,\bm \beta_0)$ is almost surely bounded}, we have,
\[
\mathbb{E}[\langle g^2(\bm x,\bm x^\prime,\bm \beta,\bm \beta^\prime;\bm \sigma_0,\bm \beta_0)\rangle_{\mathcal{H}_{k,s;\bm \eta}}]=O(p^2),
\]
and
\[
\mathbb{E}[\langle s^2_{\bm \sigma_0,\bm x}\rangle_{\mathcal{H}_{k,s;\bm \eta}}]=O(1).
\]
This implies the result.

\section{Proof of Theorem \ref{thm:state_evol_graph}}
\label{proof_amp_graph}
\revsag{We shall show the case when $\lambda=0$. For $\lambda>0$, the proof can be extended using the techniques used to prove Theorem 7.2 of \citet{ma_nandy}.} Let us consider the following AMP orbits.
\begin{align}
\label{eq:amp_iterates_poly}
\widetilde{\bm \sigma}^{t+1} &= \frac{\bm W}{\sqrt{p}} \widetilde{\bm f}_t(\widetilde{\bm \sigma}^t,\bm S^\top\widetilde{\bm z}^{t-1}+\widetilde{\bm \beta}^{t-1})\\
&\quad\quad-(\mathcal A\widetilde{\bm f}_t)(\widetilde{\bm \sigma}^t,\bm S^\top\widetilde{\bm z}^{t-1}+\widetilde{\bm \beta}^{t-1})\widetilde{\bm f}_{t-1}(\widetilde{\bm \sigma}^{t-1},\bm S^\top\widetilde{\bm z}^{t-2}+\widetilde{\bm \beta}^{t-2}),\\
\end{align}
and,
\begin{align}
\label{eq:amp_iterates_1_poly}
\widetilde{\bm z}^t & = \bm y^\circ - \bm S\widetilde{\bm \beta}^t-\frac{1}{\kappa}\widetilde{\bm z}^{t-1}(\mathcal A\widetilde{\bm \zeta}_{t-1})(\bm S^\top\widetilde{\bm z}^{t-1}+\widetilde{\bm \beta}^{t-1},\widetilde{\bm \sigma}^{t})\\
\widetilde{\bm \beta}^{t+1} &= \widetilde{\bm \zeta}_t(\bm S^\top\widetilde{\bm z}^t+\widetilde{\bm \beta}^t,\widetilde{\bm \sigma}^{t+1})\\
\end{align}
where $\bm W$ is defined by \eqref{eq:rank_one_graph_deformation}. The state evolution of this AMP is characterized by a set of parameters $\widetilde{\tau}_t,\widetilde{\eta}_t,\widetilde{\nu}_t$ and $\widetilde{\sigma}_t$ defined using the formula \eqref{eq:state_evol} and the polynomial update functions $\widetilde{f}_t$ and $\widetilde{\zeta}_t$. \revsag{Given $\varepsilon>0$, let us consider polynomial denoisers $\wt f_t:\mathbb R^2 \rightarrow \R$ and $\wt \zeta_t:\R^2 \rightarrow \R$ satisfying} 
\begin{align}
\label{eq:conc_1}
    \mathbb{E}[(\widetilde{f}_t(\widetilde{\nu}_{t}Z_2,B+\widetilde{\tau}_{t-1}Z_1)-f_t(\widetilde{\nu}_{t}Z_2,B+\widetilde{\tau}_{t-1}Z_1))^2]<\varepsilon, 
\end{align}
and
\begin{align}
\label{eq:conc_2}
    \mathbb{E}[(\widetilde{\zeta}_t(B+\widetilde{\tau}_tZ_1,\widetilde{\nu}_{t+1}Z_2)-\zeta_t(B+\widetilde{\tau}_tZ_1,\widetilde{\nu}_{t+1}Z_2))^2]<\varepsilon,
\end{align}
for all $t\ge 0$. 
The above AMP orbits can be rewritten as
\revsag{
\begin{align}
\label{eq:graph_AMP_pol}
 \widetilde{\bm q}^t&=\widetilde{\bm \ell}_t(\widetilde{\bm h}^t,\widetilde{\bm\sigma}^{t+1},\bm \beta_0),\\
 \widetilde{\bm e}^t&= \bm S \widetilde{\bm q}^t-\widetilde{\lambda}_t\widetilde{\bm m}^{t-1}\\
 \widetilde{\bm m}^t&=\widetilde{\bm g}_t(\widetilde{\bm e}^t,\widebar\varepsilon)\\
\wt{\bm h}^{t+1} & =  \bm S^\top\widetilde{\bm m}^t-\widetilde{c}_t\widetilde{\bm q}^t,
\end{align}
and
\begin{align}
\label{eq:gaussian_AMP_2}
\wt{\bm r}^{t} & =\wt{\bm f}_t(\wt{\bm \sigma}^{t},\bm \beta_0-\wt{\bm h}^{t-1}),\\
\wt{\bm \sigma}^{t+1}&= \frac{\bm W}{\sqrt{p}}\wt{\bm r}^{t}-\wt{\upsilon}_t\wt{\bm r}^{t-1},
\end{align}
where
\begin{equation}
\label{eq:onsager_terms_graPH}
    \wt{c}_t=\frac{1}{n}\sum_{i=1}^{n}\wt{g}^\prime_t(\wt{e}^t_i,\widebar\varepsilon_i), \quad  \wt{\lambda}_t=\frac{1}{p\kappa}\sum_{i=1}^p \wt{\ell}^\prime_t(\wt{h}^t_i,\wt{\sigma}^{t+1}_i,\beta_{0i}), \quad \wt\upsilon_t=\frac{1}{p}\sum_{i=1}^p \wt{f}^{\prime}_t(\wt{\sigma}^t_i,\beta_{0i}-\wt{h}^t_i).
\end{equation}}
where the derivative is with respect to the first argument of the functions, $\bm \varepsilon$ is defined in \eqref{eq:lin_model_1} and $\widebar\varepsilon = \varepsilon/\sqrt{\kappa}$. Further, the denoisers $\widetilde \ell_t$ and $\widetilde g_t$ are defined as:
\begin{align}
\widetilde{\ell}_t(s,r,x_0)&=\widetilde{\zeta}_{t-1}(x_0-s,r)-x_0\\
\widetilde{g}_t(s,w)&=s-w.
\end{align}
Next, let us define a \emph{synchronised system of diagonal tensor networks}.
\begin{defn}
A \emph{synchronized system of diagonal tensor networks} indexed by the matrices $\bm W_1$ and $\bm W_2$ is given by the collection $T=\{\mathcal{U},\mathcal{V},\mathcal{E}_1,\mathcal{E}_2,\{p_u\}_{u \in \mathcal{U}},\{q_v\}_{v \in \mathcal{V}}\}$ in $(k_1,k_2,k_3,\ell)$ variables is comprised of two tree like graphs $\mathcal G_1$ and $\mathcal G_2$, where $\mathcal G_1=(\mathcal U,\mathcal E_1)$ with $\mathcal E_1 \subseteq \mathcal U\times\mathcal U$, $\mathcal G_2=(\mathcal U \cup \mathcal V,\mathcal E_2)$ with $\mathcal E_2 \subseteq \mathcal U\times\mathcal V$ and a value function given by
\begin{align}
&\mathsf{val}_T(\bm W_1,\bm W_2,\bm x_1,\cdots,\bm x_{k_1};\widebar{\bm x}_1,\cdots,\widebar{\bm x}_{k_1};\bm y_1,\cdots,\bm y_{k_1};\bm f_1,\cdots,\bm f_{k_2};\bm g_1,\cdots,\bm g_\ell)\\
&=\frac{1}{n}\sum_{\bm \alpha \in [p]^{\mathcal U}}\sum_{\bm \beta \in [n]^{\mathcal{V}}}p_{\bm \alpha|T}q_{\bm \beta|T}W_{1,\bm \alpha|T}W_{2,\alpha,\beta|T},
\end{align}
\end{defn}
where 
\begin{align}
& p_{\bm \alpha|T}=\prod_{u \in \mathcal U}p_u(x_1[\alpha_u],\cdots,x_{k_1}[\alpha_u],\widebar{x}_1[\alpha_u],\cdots,\widebar{x}_{k_1}[\alpha_u],f_1[\alpha_u],\cdots,f_{k_2}[\alpha_u])\\
& q_{\bm \beta|T}=\prod_{v \in \mathcal V}q_v(y_1[\beta_v],\cdots,y_{k_1}[\beta_v],g_1[\beta_v],\ldots,g_{\ell}[\beta_v]),\\
& W_{1,\alpha|T}=\prod_{(a,b)\in \mathcal E_1}W_1(\alpha_a,\alpha_b), \quad \mbox{and}\\
& W_{2,\alpha,\beta|T}=\prod_{(a,b)\in \mathcal E_2}W_2(\alpha_a,\beta_b).
\end{align}
Next, we consider the following lemma.
\begin{lem}
\label{lem:pol_tree}
Fix any $t \ge 1$ and $\{\widetilde{\bm h}^{t},\widetilde{\bm m}^{t},\widetilde{\bm \sigma}^{t},\widetilde{\bm r}^{t},\widetilde{\bm e}^{t},\widetilde{\bm q}^{t}:t \ge 1\}$ be the AMP iterates. For any polynomials $p:\mathbb{R}^{4t+2}\rightarrow \mathbb{R}$ and $q:\mathbb{R}^{2t+2}\rightarrow \mathbb{R}$, and for two finite sets $\mathcal F_1$ and $\mathcal F_2$ of \emph{synchronized system of diagonal tensor networks} in $(1,1,0,1)$ variables with value functions $\mathsf{val}^{(1)}_T$ and $\mathsf{val}^{(2)}_T$ satisfying
\begin{align}
\label{eq:p_val_tree}
   &\frac{1}{p}\sum_{i=1}^{p}p(\widetilde{h}^{0}_{i},\cdots,\widetilde{h}^{t}_{i};\widetilde{r}^{1}_{i},\cdots,\widetilde{r}^{t}_{i};\beta_{0i},\widebar{\varepsilon}_i) =\sum_{T \in \mathcal{F}_1}\mathsf{val}^{(1)}_T\left(\bm W/\sqrt{p},\bm S,\bm h^0,\bm \sigma^1,\bm \beta_0,\widebar{\bm \varepsilon}\right),
\end{align}
and
\begin{align}
\label{eq:q_val_tree}
   &\frac{1}{n}\sum_{i=1}^{n}q(\widetilde{e}^{0}_{i},\cdots,\widetilde{e}^{t}_{i};\beta_{0i},\varepsilon_i) =\sum_{T \in \mathcal{F}_2}\mathsf{val}^{(2)}_T\left(\bm W/\sqrt{p},\bm S,\bm h^0,\bm \sigma^1,\bm \beta_0,\widebar{\bm \varepsilon}\right).
\end{align}
\end{lem}
\begin{proof}
The proof of the lemma follows using the definition of \emph{synchronized system of diagonal tensor networks} and the techniques of the proofs of Lemma 3.8 and Lemma 4.3 of \citet{wang_zhong_fan}.
\end{proof}
Now, we consider the following lemma which characterizes the universality of the terms in the right-hand sides of \eqref{eq:p_val_tree} and \eqref{eq:q_val_tree}.
\begin{lem}
\label{lem:val_universality}
Let us consider $\bm W$ as defined in \eqref{eq:rank_one_graph_deformation} and its variance profile $\bm W_{V}$ as defined in \citet{wang_zhong_fan}. If $\widebar{d}_p(1-\widebar{d}_p)\ge C\,\log p/p$ for some constant $C>0$ then we have a deterministic value $\mathrm{V}^{(1)}_T$ and $\mathrm{V}^{(2)}_T$ such that:
\begin{align}
    &\lim_{n \rightarrow \infty}\mathsf{val}^{(1)}_T(\bm W/\sqrt{p},\bm S,\bm h^0,\bm \sigma^1,\bm \beta_0,\widebar{\bm \varepsilon})=\lim_{n \rightarrow \infty}\mathsf{val}_T(\bm Z,\bm S,\bm h^0,\bm \sigma^1,\bm \beta_0,\widebar{\bm \varepsilon})=\mathrm{V}^{(1)}_T,
\end{align}
and
\begin{align}
    &\lim_{n \rightarrow \infty}\mathsf{val}^{(2)}_T(\bm W/\sqrt{p},\bm S,\bm h^0,\bm \sigma^1,\bm \beta_0,\widebar{\bm \varepsilon})=\lim_{n \rightarrow \infty}\mathsf{val}^{(2)}_T(\bm Z,\bm S,\bm h^0,\bm \sigma^1,\bm \beta_0,\widebar{\bm \varepsilon})=\mathrm{V}^{(2)}_T,
\end{align}
where $\bm Z$ is the Gaussian matrix defined in \eqref{eq:Gaussian_model}.
\end{lem}
\begin{proof}
If $\widebar{d}_p(1-\widebar{d}_p)\ge C\,\log p/p$ then the matrix $\bm W/\sqrt{p}$ satisfies the assumptions of Lemma 2.11 of \citet{wang_zhong_fan}. Then using the techniques to prove Lemma 2.11 and Lemma 4.4 of \citet{wang_zhong_fan} we the result follows.
\end{proof}
Since $\|\bm W/\sqrt{p}\|_{op}<\infty$ and the update functions of the AMP orbits \eqref{eq:amp_iterates_poly} satisfies \eqref{eq:conc_1} and \eqref{eq:conc_2}, using the techniques used to prove Lemmas 2.10 and 3.12 of \citet{wang_zhong_fan}, for any pseudo-Lipschitz functions and $\{\widecheck{\bm \sigma}^t,\widecheck{\bm z}^t,\widecheck{\bm \beta}^t: t \ge 0\}$ defined as:
\begin{align}
\label{eq:amp_iterates_mean_zero}
\widecheck{\bm \sigma}^{t+1} &= \frac{\bm W}{\sqrt{p}} \bm f_t(\widecheck{\bm \sigma}^t,\bm S^\top\widecheck{\bm z}^{t-1}+\widecheck{\bm \beta}^{t-1})\\
&\quad\quad-(\mathcal A\bm f_t)(\widecheck{\bm \sigma}^t,\bm S^\top\widecheck{\bm z}^{t-1}+\widecheck{\bm \beta}^{t-1})\bm f_{t-1}(\widecheck{\bm \sigma}^{t-1},\bm S^\top\widecheck{\bm z}^{t-2}+\widecheck{\bm \beta}^{t-2}),\\
\end{align}
and,
\begin{align}
\label{eq:amp_iterates_1_mean_zero}
\widecheck{\bm z}^t & = \bm y^\circ - \bm S\widecheck{\bm \beta}^t-\frac{1}{\kappa}\widecheck{\bm z}^{t-1}(\mathcal A\bm \zeta_{t-1})(\bm S^\top\widecheck{\bm z}^{t-1}+\widecheck{\bm \beta}^{t-1},\widecheck{\bm \sigma}^{t})\\
\widecheck{\bm \beta}^{t+1} &= \bm \zeta_t(\bm S^\top\widecheck{\bm z}^t+\widecheck{\bm \beta}^t,\widecheck{\bm \sigma}^{t+1}),\\
\end{align}
we have:
\begin{align}
&\frac{1}{p}\sum\limits_{i=1}^{p}\psi([\bm S^\top\widetilde{\bm z}^t+\widetilde{\bm \beta}^t]_i,\widetilde{\sigma}^{t+1}_i,\beta_{0i},\sigma_{i})-\frac{1}{p}\sum\limits_{i=1}^{p}\psi([\bm S^\top\widecheck{\bm z}^t+\widecheck{\bm \beta}^t]_i,\widecheck{\sigma}^{t+1}_i,\beta_{0i},\sigma_{i}) \xrightarrow{a.s.}0. 
\end{align}
Next, let us consider the AMP iterates given by
\begin{align}
\label{eq:amp_iterates_mean_zero_g}
\breve{\bm \sigma}^{t+1} &= \frac{\bm Z}{\sqrt{p}} \widetilde{\bm f}_t(\breve{\bm \sigma}^t,\bm S^\top\breve{\bm z}^{t-1}+\breve{\bm \beta}^{t-1})\\
&\quad\quad-(\mathcal A\widetilde{\bm f}_t)(\breve{\bm \sigma}^t,\bm S^\top\breve{\bm z}^{t-1}+\breve{\bm \beta}^{t-1})\widetilde{\bm f}_{t-1}(\breve{\bm \sigma}^{t-1},\bm S^\top\breve{\bm z}^{t-2}+\breve{\bm \beta}^{t-2}),\\
\end{align}
and,
\begin{align}
\label{eq:amp_iterates_1_mean_zero_g}
\breve{\bm z}^t & = \bm y^\circ - \bm S\breve{\bm \beta}^t-\frac{1}{\kappa}\breve{\bm z}^{t-1}(\mathcal A\widetilde{\bm \zeta}_{t-1})(\bm S^\top\breve{\bm z}^{t-1}+\breve{\bm \beta}^{t-1},\breve{\bm \sigma}^{t})\\
\breve{\bm \beta}^{t+1} &= \widetilde{\bm \zeta}_t(\bm S^\top\breve{\bm z}^t+\breve{\bm \beta}^t,\breve{\bm \sigma}^{t+1}),\\
\end{align}
where $\bm Z$ is defined in \eqref{eq:Gaussian_model}. Using Lemmas \ref{lem:pol_tree}, \ref{lem:val_universality} and the techniques used to prove Lemma 2.10 of \citet{wang_zhong_fan}, we get
\begin{align}
\label{eq:gaussian_pol_connect}
&\frac{1}{p}\sum\limits_{i=1}^{p}\psi([\bm S^\top\widetilde{\bm z}^t+\widetilde{\bm \beta}^t]_i,\widetilde{\sigma}^{t+1}_i,\beta_{0i},\sigma_{0i})-\frac{1}{p}\sum\limits_{i=1}^{p}\psi([\bm S^\top\breve{\bm z}^t+\breve{\bm \beta}^t]_i,\breve{\sigma}^{t+1}_i,\beta_{0i},\sigma_{0i}) \xrightarrow{a.s.}0.
\end{align}
Using the proof of Theorem \ref{thm:state_evol}, Theorem 7.1 of \citet{ma_nandy}, Lemma 3.8 of \citet{wang_zhong_fan} and \eqref{eq:gaussian_pol_connect}, we have
\begin{align}
&\lim\limits_{p \rightarrow \infty}\frac{1}{p}\sum\limits_{i=1}^{p}\psi([\bm S^\top\widecheck{\bm z}^t+\widecheck{\bm \beta}^t]_i,\widecheck{\sigma}^{t+1}_i,\beta_{0i},\sigma_{0i})\overset{a.s.}{=}\mathbb{E}[\psi(B+\tau_t Z_1,\nu_{t+1}Z_2,B,\Sigma)]. 
\end{align}
where $\tau_t,\nu_t,\sigma_t$ are defined in Theorem \ref{thm:state_evol}. This implies the result.

\end{document}